%% file: nearalt.tex
\documentclass[11pt, reqno]{amsart}
\usepackage{amsmath,amssymb,amsthm,graphicx, url, verbatim}
\usepackage{color, enumerate, mathtools, float}
\usepackage{relsize}
\usepackage{thmtools}
\usepackage{thm-restate}
\usepackage{hyperref}
\usepackage{cleveref}

\DeclarePairedDelimiter\floor{\lfloor}{\rfloor}
\usepackage[margin=1.25in]{geometry}
\theoremstyle{definition}
\newtheorem{defn}{Definition}[section]
\newtheorem{rem}[defn]{Remark}
\newtheorem{eg}[defn]{Example}
\theoremstyle{plain}
\newtheorem{thm}{Theorem}
\newtheorem{conj}[defn]{Conjecture}

\newtheorem{lem}[defn]{Lemma} 
\newtheorem{cor}[defn]{Corollary}

\newtheorem{restate2}{Theorem}
\newtheorem{restate3}{Theorem}
\newcommand{\Sk}{\mathcal{S}}
\newcommand{\jwproj}{\vcenter{\hbox{\includegraphics[scale=.1]{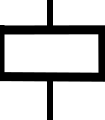}}}}

\newcommand{\sgn}{\text{sgn}}

\graphicspath{{Skein/}}

\begin{document}

\title{Jones slopes and  coarse volume of near-alternating knots}
\author[C. Lee]{Christine Ruey Shan Lee}

\address[]{Department of Mathematics and Statistics, University of South Alabama, Mobile AL 36688}
\email[]{crslee@southalabama.edu}

\begin{abstract} 
We study \emph{near-alternating links} whose diagrams satisfy conditions generalized from the notion of semi-adequate links. We extend many of the results known for adequate knots relating their colored Jones polynomials to the topology of essential surfaces and the hyperbolic volume of their complements: we show that the Strong Slope Conjecture is true for near-alternating knots with spanning Jones surfaces, their colored Jones polynomials admit stable coefficients, and the stable coefficients provide two-sided bounds on the volume of the knot complement. 
\end{abstract}

\maketitle

\tableofcontents

\section{Introduction}
Since the discovery of the Jones polynomial and related quantum knot invariants, a central problem in quantum topology has been to understand the connection between those invariants and the geometry of the knot complement. An important example of these quantum invariants is the colored Jones polynomial, which assigns a sequence $\{J_K(v, n)\}_{n=2}^{\infty}$ of Laurent polynomials from the representation theory of $U_q(\mathfrak{sl}_2)$ to a link $K\subset S^3$, and contains the Jones polynomial as the first term of the sequence, see Definition \ref{defn:cp}. Conjectures such as the Volume Conjecture \cite{Kas97, MM01, MM02} and the Strong Slope Conjecture \cite{Gar11, KT15} predict that the colored Jones polynomial is closely related to the hyperbolic geometry and topology of surfaces in the knot complement. 

Much evidence for this relationship comes from the class of \emph{semi-adequate} links. They are a class of links satisfying a diagrammatic condition, see Definition \ref{defn:adequate-diagram}. An adequate knot satisfies the Strong Slope Conjecture, see Conjecture \ref{conj:slopes}, and certain stable coefficients of its colored Jones polynomial give volume bounds on the complement of an adequate knot \cite{DL07, FKP08, FKP11, FKP13}. For these results, a key ingredient is the existence of \emph{essential} spanning surfaces, see Definition \ref{defn:essential}, along which the complement may be decomposed into simpler geometric components. Such surfaces have also been shown to be fundamental to the characterization of alternating knots \cite{Jos15, Ho15} and adequate knots \cite{Kal16}.  

In this paper, we are motivated by the question of when we can expect the Strong Slope Conjecture to be realized by spanning surfaces from state surfaces of the knot diagram beyond semi-adequate knots, and when we can expect the Coarse Volume Conjecture to be satisfied outside the class of adequate links.  Our answer to this question in this paper is the introduction of the class of \emph{near-alternating} links, to be defined below in Definition \ref{defn:near-alternating}. For a near-alternating knot, we find its Jones slopes, and show that there exist essential spanning surfaces in its exterior realizing the Strong Slope Conjecture. For a near-alternating link, we prove that the first, second, penultimate, and the last coefficient of its colored Jones polynomial are stable. If the near-alternating link diagram is prime, twist-reduced, and highly twisted with more than 7 crossings in each twist region, then the link is hyperbolic by \cite{FKP08}, and we show that these stable coefficients provide coarse volume bounds for the link exterior. These results closely mirror those for adequate links. However we show that near-alternating knots are not adequate, thus they give a new class of links satisfying the above conjectures. 

We give the necessary definitions in order to state the main results below. We shall always consider a knot or a link $K\subset S^3$. The theorems and conjectures will be stated in the fullest generality possible, where it will be indicated whether we are considering a knot or a link. The indices $i, j, k$ should be considered independently in each instance unless explicitly stated otherwise. 

\subsection{Near-alternating link}

Let $G$ be a finite, weighted planar graph in $S^2$. For each edge $e$ of $G$ let $\omega_e \in \mathbb{Z}\setminus 0$ be the weight. We may replace each vertex $v$ of $G$ with a disk $\mathcal{D}^2$ and each edge $e$ with a twisted band $B$ consisting of $|\omega_e|$ right-handed (positive) or left-handed (negative) half twists if $\omega_e>0$, or if $\omega_e<0$, respectively. See Figure \ref{fig:exneara} for the definition of right-handed and left-handed half twists in this paper.  Note that this is opposite of the convention where right-handed half twists are negative and left-handed half-twists are positive, see for example \cite{conway}. We denote the resulting surface by $F_G$ and consider the link diagram $D = \partial(F_G)$. Every link diagram $D$ may be represented as $\partial(F_G)$ for some finite, weighted planar graph $G$.

A \emph{path} in a weighted graph $G$ with vertex set $V$ and a weighted edge set $E$ is a finite sequence of distinct vertices $v_1, v_2, \ldots, v_k$ such that $(v_i, v_{i+1}) \in E$ for $i = 1, 2, \cdots, k-1$. We define the \emph{length} of a path $W$ as  
\begin{equation} \label{eq.l} \ell(W):= 2+\sum_{i=1}^{k-1} (|\omega_i|-2),\end{equation}
where $\omega_i$ is the weight of the edge $(v_i, v_{i+1})$ in $W$. 

A graph $G$ is said to be \emph{$2$-connected} if it does not have a vertex whose removal results in a disconnected graph. Such a vertex is called a \emph{cut vertex}. 

\begin{defn} \label{defn:near-alternating} 
Let $D$ be a non-split link diagram $D = \partial(F_G)$, where $G$ is a 2-connected, finite, weighted planar graph without one-edged loops (an edge between the same vertex) with a single negative edge $e=(v, v')$ of weight $r<0$ and $|r| \geq 2$. Let $G\setminus e$ be the graph obtained from $G$ by deleting the edge $e$ and let $G/e$ be the graph obtained from $G$ by contracting $G$ along $e$. We say that $D$ is \emph{near-alternating} if the graph $G$ satisfies the following conditions.
\begin{enumerate}
\item Let $\omega$ be the minimum of $\ell(W)$ taken over all paths $W$ in $G \setminus e$ starting at $v$ and ending at $v'$, and let $t$ be the total number of such paths. Then $t>2$, and 
\[\frac{\omega}{t} > |r|. \] 
\item \label{d.case2} The graph $G\setminus e$ remains 2-connected, and the diagram $D^e = \partial(F_{G \setminus e})$ is prime and twist-reduced (see \cite{Lac04} for a pictorial definition); the diagram $D_e=\partial(F_{G/e})$ is adequate, see Definition \ref{defn:adequate-diagram}. 
\end{enumerate}

Condition \eqref{d.case2} is imposed to ensure that a near-alternating link is $-$-adequate to reduce the technicalities in the conditions of the results. See Definition \ref{defn:adequate-diagram} for the definition of $+$-or $-$-adequate links.  

A link $K$ is said to be \emph{near-alternating} if it admits a near-alternating diagram. See Figure \ref{fig:exneara} for an example and the conventions for a negative or a positive twist region. 
\end{defn} 

\begin{eg} A pretzel link $P(t_1, t_2, \ldots,  t_m)$ is near-alternating if $m>3$, $t_1 \leq -2< 0 < t_i$ for all $1<i\leq m$, and \[ \frac{\min_{1<i\leq m} \left\{ t_i \right\}}{m-1} > |t_1|. \]  
\end{eg}

\begin{eg}
\end{eg} 
\begin{figure}[H]  
\centering
    \def\svgwidth{1\columnwidth}
    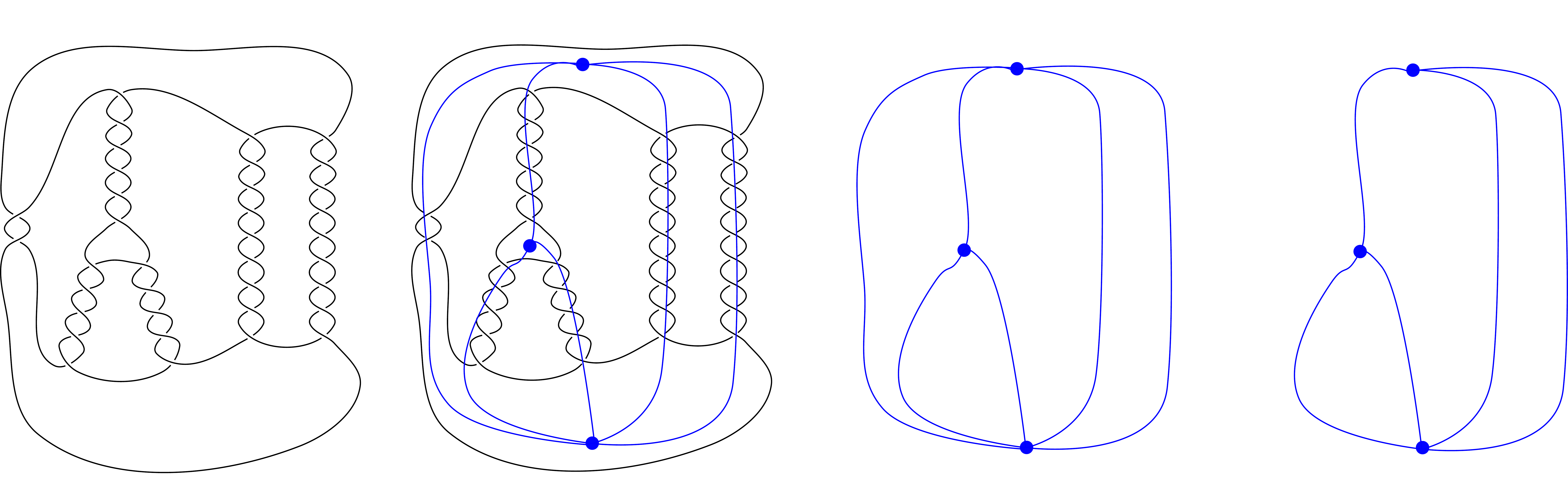
    \caption{\label{fig:exneara} An example of a near-alternating link diagram $D=\partial(F_G)$ with the graph $G$ shown in blue and the negatively-weighted edge $e = (v, v')$. For this example, we have $\frac{\omega}{t} = \frac{9}{4} > 2.$ }
\end{figure}

\subsection{The Strong Slope Conjecture} 

Let $D$ be a link diagram. A \emph{Kauffman state} $\sigma$ is a choice of replacing every crossing of $D$ by the $+$- or $-$-resolution as in Figure \ref{fig:abres}, with the (dashed) segment recording the location of the crossing before the replacement.

\begin{figure}[ht]
 \centering
    \def\svgwidth{.2\columnwidth}
    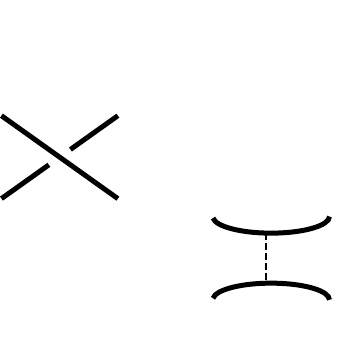
\caption{The $+$-and $-$-resolution of a crossing and the corresponding segments.}
\label{fig:abres}
\end{figure}

Applying a Kauffman state results in a set of disjoint circles called \emph{state circles}. We form a $\sigma$-\emph{state graph} $s_{\sigma}(D)$ for each Kauffman state $\sigma$ by letting the resulting state circles be vertices and the segments be edges. The \emph{all-$+$} state graph $s_+(D)$ comes from the Kauffman state which chooses the $+$-resolution at every crossing of $D$. Similarly, the \emph{all-$-$} state graph $s_-(D)$ comes from the Kauffman state which chooses the $-$-resolution at every crossing of $D$.

Let
\begin{equation}
h_n(D) = -(n-1)^2c(D) -2(n-1)|s_+(D)| + \omega(D) (n^2-1), \label{eq:lowerbound}
\end{equation}
where $c(D)$ is the number of crossings of $D$, and $\omega(D) = c_+(D)-c_-(D)$, the difference between the number of positive crossings and the number of negative crossings of $D$, is the writhe of $D$ with an orientation. Lastly, $|s_+(D)|$ is the number of vertices in the all-$+$ state graph. We can now state the main result of this paper. 

Let $d(n)$ be the minimum degree of $J_K(v, n)$, the \emph{$n$th colored Jones polynomial of $K$.}
\begin{thm}\label{thm:degree} Let $K \subset S^3$ be a link admitting a near-alternating diagram $D$ with a single negative twist region of weight $r<0$ and let $h_n(D)$ be defined by \eqref{eq:lowerbound}, then
\begin{equation} d(n) = h_n(D) - 2r(n^2-n). \end{equation} 
\end{thm} 

This is the main result of the paper. The single negative twist region of the near-alternating diagram is used to write a special state sum for the colored Jones polynomial, that is particularly suited to finding the degree.

Now we consider the case when $K$ is a knot. Note that the case for 3-tangle pretzel knots with a near-alternating diagram was already studied in  \cite{LV},  and the degree of the colored Jones  polynomial  was computed in \cite{HTY00} for a family of pretzel knots which are generally not near-alternating.  

Theorem \ref{thm:degree} implies the Strong Slope Conjecture for near-alternating knots which we now describe. Let $N(K)$ be a tubular neighborhood of $K$ in $S^3$. We will denote by $S^3\setminus K$ the closure of $S^3\setminus N(K)$. An orientable, connected, and properly embedded surface $S \subset S^3 \setminus K$ is \emph{essential} if it is incompressible, boundary-incompressible, and non boundary-parallel. If $S$ is non-orientable, then $S$ is \emph{essential} if its orientable double cover in $S^3\setminus K$ is essential. 

\begin{defn} \label{defn:essential} Let $S$ be an essential and orientable surface with non-empty boundary in $S^3 \setminus K$. A fraction $\frac{p}{q} \in \mathbb{Q} \cup \{\frac{1}{0}\}$ is a \emph{boundary slope} of $K$ if $p\mu + q\lambda$ represents the homology class of $\partial S$ in $H_1(\partial N(K))$, where $\mu$ and $\lambda$ are the canonical meridian and longitude basis of $H_1(\partial N(K))$. The boundary slope of an essential non-orientable surface is that of its orientable double cover. 
\end{defn} 

Let $d^*(n)$ be the maximum degree in $v$ of $J_K(v, n)$. Garoufalidis showed in \cite{Gar11} that since the colored Jones polynomial is $q$-holonomic \cite{GL05}, the functions $d(n)$ and $d^*(n)$ are \emph{quadratic quasi-polynomials} viewed as functions from $\mathbb{N} \rightarrow \mathbb{N}$. For a fixed knot $K$, this means that there exist integers $p_K$, $C_K \in \mathbb{N}$ and rational numbers $a_j, b_j, c_j, a^*_j, b^*_j, c^*_j$ for each $0\leq j < p_K$, such that for all $n > C_K$,  
\[d(n) = a_jn^2+ b_jn + c_j \text{ if } n = j \pmod{p_K}, \] and  
\[d^*(n) = a^*_jn^2+ b^*_jn + c^*_j \text{ if } n = j \pmod{p_K}.\] 

We consider the sets $js_K:= \{a_j\}$ and $js^*_K:= \{a^*_j\}$. An element $\frac{p}{q} \in js_K \cup js^*_k$ is called a \emph{Jones slope}. Similarly, define $jx_K := \{\frac{b_j}{2} \}$ and $jx^*_K:= \{\frac{b^*_j}{2}\}$. 

We may now state the Strong Slope Conjecture.  
\begin{conj}{(\cite{Gar11, KT15})}\label{conj:slopes} 
Let $K$ be a knot. Given a Jones slope of $K$, say $\frac{p}{q} \in js_K$, with $q>0$ and \text{gcd}$(p, q)=1$, there is an essential surface $S\subset S^3\setminus K$ with $|\partial S|$ boundary components such that each component of $\partial S$ has slope $\frac{p}{q}$, and 
\[ -\frac{\chi(S)}{|\partial S|q} \in jx_K. \] 
Similarly, given $\frac{p^*}{q^*} \in js^*_K$ with $q^*>0$ and \text{gcd}$(p^*, q^*)=1$, there is an essential surface $S^*\subset S^3\setminus K$ with $|\partial S^*|$ boundary components such that each component of $\partial S^*$ has slope $\frac{p^*}{q^*}$, and 
\[ \frac{\chi(S^*)}{|\partial S^*|q^*} \in jx^*_K. \] 
\end{conj}
An essential surface in $S^3\setminus K$ satisfying the conditions described in the conjecture is called a \emph{Jones surface}.

\subsection*{Normalization convention}
The difference in our convention from \cite{Gar11, KT15} is that in this paper the asterisk $*$ indicates the corresponding quantity from the maximum degree $d^*(n)$, rather than the minimum degree, indicated by $d(n)$, of the $n$th colored Jones polynomial $J_K(v, n)$. Also, we substitute $v = \frac{1}{A}$, where $A$ is the variable for the Kauffman bracket, for the colored Jones polynomial. See Definition \ref{defn:cp} for our choice of the normalization convention.

\subsection{Known results}
The Strong Slope Conjecture is currently known for alternating knots \cite{Gar11}, adequate knots \cite{FKP11, FKP13}, which is a generalization of alternating knots, see Definition \ref{defn:adequate-diagram}, iterated  $(p, q)$-cables of torus knots and iterated cables of adequate knots \cite{KT15}, graph knots \cite{MT17, BMT18}, and families of 3-tangle pretzel knots \cite{LV}, as well as families of 3-tangle Montesinos knots \cite{LYL18}. It is also known for all knots  with up to 9 crossings \cite{Gar11, KT15, Howie}, and an infinite family of arborescent non-Montesinos  knots \cite{HD17}. The Slope Conjecture is also known for 2-fusion knots \cite{GR14}.  

A major difficulty in studying the Conjecture is determining the Jones slope of a knot. Compared to the approaches of the previous results, the techniques developed in this paper does not rely on specific structure of the graph $G$ giving rise to $D = \partial(F_G)$. The choice of a single negative twist region is made to simplify the exposition. With more work, it would be possible to extend  Theorem \ref{thm:degree} to links with diagrams obtained from Murasugi sums of an adequate diagram with a non-adequate torus link diagram, and highly twisted links with more than one negative twist region satisfying additional graphical constraints.

Furthermore, preliminary evidence from pretzel links with an arbitrary number of tangles suggests that the approach developed in this paper may also be used to determine Jones slopes for links for which the conditions for being near-alternating do not hold.  In other words,  if $D = \partial(F_G)$ is a link diagram where $G$ is a 2-connected, finite, weighted planar graph without one-edged loops with a single negative edge of weight $r<0$, so that the quantities $\omega$ and $t$ still make sense, we expect that $\frac{\omega}{t} \leq |r|$ implies that the Jones slope is non-integral, or, it is not realized by a state surface. We explore this in an up-coming paper on the Slope Conjecture for Montesinos knots \cite{GLV18}.

Another difficulty in approaching Conjecture \ref{conj:slopes} is in finding surfaces with boundary slopes equal to the Jones slopes and proving that they are essential. In the context of the Conjecture, Theorem \ref{thm:degree} says that $js_K= \{-2c_-(D)-2r\}$ and $jx_K=\{c(D)-|s_+(D)|+r\}$. A surface realizing $js_K$ and $jx_K$ from Theorem \ref{thm:degree} is a state surface corresponding to a Kauffman state. In this case, we are fortunate that the criteria for essential spanning surfaces by the works of Ozawa \cite{Oza11} and Ozawa and Rubinstein \cite{OR12} readily apply to show that it is a Jones surface.

\begin{defn} Given a Kauffman state $\sigma$ on a link diagram $D$, we may form the \emph{$\sigma$-state surface}, denoted by $S_{\sigma}(D)$, by filling in the disjoint circles in $s_{\sigma}(D)$ with disks, and replacing each segment recording the previous location of the crossing by a half-twisted band as shown in Figure \ref{fig:abressurface}. 
\end{defn} 
\begin{figure}[ht]
 \centering
    \def\svgwidth{.3\columnwidth}
    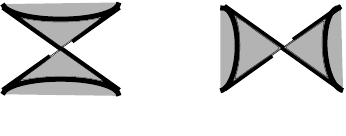
\caption{\label{fig:abressurface}}
\end{figure}

For a near-alternating knot $K$ with $\partial(F_G) = $ a near-alternating diagram $D$ of $K$ for some graph $G$, the surface $F_G$ is essential by \cite[Theorem 2.15]{OR12} and is given by the state surface $S_{\sigma}(D)$ where $\sigma$ chooses the $-$-resolution on the $|r|$ crossings corresponding to the single edge with negative weight $r$ in  $G$, and the $+$-resolution everywhere else. This surface is easily visualized from the knot diagram, see Figure \ref{f.sigmasurface} for an example.
\begin{figure}[H]
\centering
 \def\svgwidth{.3\columnwidth}
    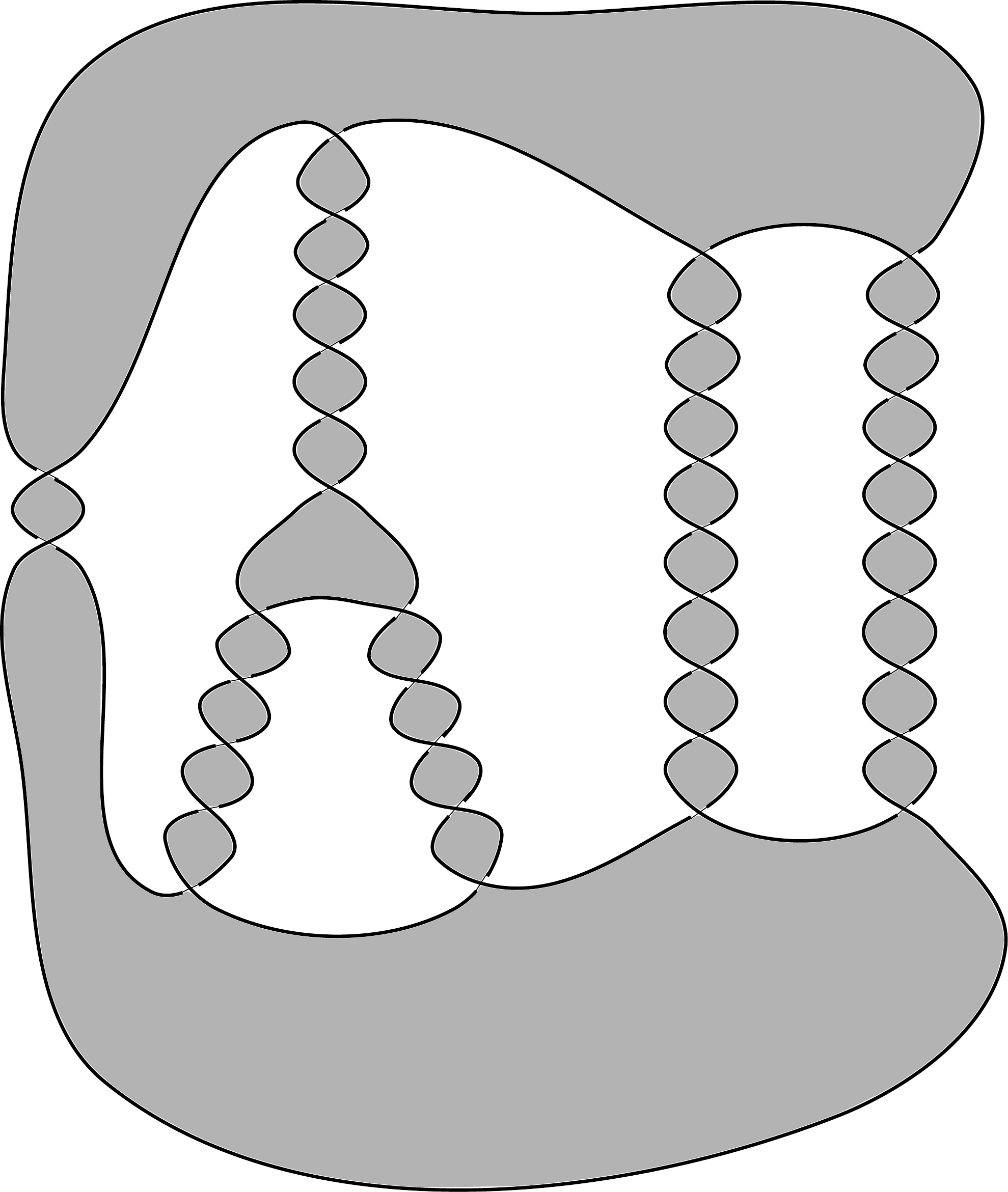
    \caption{\label{f.sigmasurface}The surface $S_{\sigma}(D)$.}
\end{figure} 
Let $|s_+(D)|$ be the number of disjoint circles in $s_+(D)$.  We show that the boundary slope and Euler characteristic of this surface match with $js_K$ and $jx_K$.
\begin{thm} \label{thm:jsurface} Let $K\subset S^3$ be a knot admitting a near-alternating diagram $D = \partial (F_G)$ with a single negative twist region of weight $r<0$. The surface $F_G$ is essential with 1 boundary component with boundary slope $-2c_-(D)-2r$ and 
\[-\chi(F_G) = c(D)-|s_+(D)|+r.  \]  
\end{thm}

To see the Jones surface  $S^* \subset S^3\setminus K$ with boundary slope $\frac{p^*}{q^*}$ matching $js^*_K$ and $\frac{\chi(S^*)}{|\partial S^*|q^*}$ matching $jx^*_K$, we use the fact that a near-alternating link is \emph{$-$-adequate}, see Lemma \ref{lem:nabad}, as defined below. The notation of adequacy is originally due to \cite{LT88}.

\begin{defn} \label{defn:adequate-diagram}
A link diagram $D$ is \emph{$+$-adequate} (resp. \emph{$-$-adequate}) if its all-$+$ (resp. all-$-$) state graph $s_+(D)$ (resp. $s_-(D)$) has no one-edged loops.  A link $K$ is \emph{semi-adequate} (\emph{$+$-or $-$-adequate}) if it admits a diagram that is $+$-or $-$-adequate. If a link $K$ admits a diagram that is both $+$-and $-$-adequate, then we say that $K$ is \emph{adequate}.
\end{defn}
Note that alternating links form a subset of adequate links.

Let 
\begin{equation}
h^*_n(D) = (n-1)^2c(D) +2(n-1)|s_-(D)| + \omega(D) (n^2-1). \label{eq:upperbound}
\end{equation}

It is well known that for any link diagram $D$, we have $h_n(D)\leq d(n)$, $d^*(n) \leq h^*_n(D)$ and the first equality is achieved when $D$ is $+$-adequate, while the second equality is achieved when $D$ is $-$-adequate. This follows from \cite{LT88}, \cite[Lemma 5.4]{Lic97}, and \cite{FKP11, FKP13}. Therefore, if $K$ is $+$-adequate (resp. $-$-adequate) then there is a single Jones slope in $js_K$ (resp. in $js^*_K$).

If $D$ admits a $+$-(resp. $-$-)adequate diagram, then \cite{Oza11} implies that the all-$+$ (resp. all-$-$) state surface is essential. An all-$+$ or all-$-$ state surface was shown by \cite{FKP11, FKP13} to realize $js_K, jx_K$, or $js^*_K, jx^*_K$, respectively. We show that a near-alternating diagram is $-$-adequate in Lemma \ref{lem:nabad}, so its all-$-$ state surface realizes $js^*_K$ and $jx^*_K$. Thus, the surface $F_G$ and the all-$-$ state surface of a near-alternating diagram verify the Strong Slope Conjecture for these knots, see Corollary \ref{cor:ssj}.

As for the question of whether a near-alternating knot can admit a $+$-adequate diagram, we show, using the Kauffman polynomial, that a near-alternating knot cannot admit a diagram that is both $+$-and $-$-adequate.
 
\begin{thm} \label{thm:naknoadequate} A near-alternating knot does not admit an adequate diagram.
\end{thm} 

Theorem \ref{thm:naknoadequate} does not rule out the possibility that a near-alternating knot admits a $+$-adequate diagram that is not also $-$-adequate. However, it seems a very difficult problem to determine whether a knot admits a $+$-adequate or $-$-adequate diagram, given a diagram that is not $+$- or $-$-adequate, respectively. As far as the author knows, there is no characterization of semi-adequacy that can be applied to decide if a near-alternating knot admits a $+$-adequate diagram. It is an interesting question whether the colored Jones polynomial can be used to develop such a characterization by obstructing the existence of a $+$-adequate diagram for a near-alternating knot. The criterion from \cite{Lee16} may be applied if there is information restricting the number of positive crossings in a diagram. We will pursue this question in a future project. 

\subsection*{Relation to almost alternating links}
A diagram of a link is \emph{almost alternating} if one crossing change makes the diagram alternating. If a link admits an almost alternating diagram, then it is said to be almost alternating. Almost alternating links forms another interesting class of links that have nice topological and geometric properties \cite{Aetal92, AL17, DL18, Ito18, LS17}. 

Directly applying the proof of \cite[Theorem 3.1]{Aetal92} shows that near-alternating links form a sub-class of almost alternating links. This may be of independent interest. 

\begin{lem}
Every near-alternating link is almost alternating. 
\end{lem} 
\begin{proof} A near-alternating link admits a near-alternating diagram with a single negative twist region. We isotope this diagram to be almost alternating as shown in the following (local) picture.  
\begin{figure}[H]
\centering
\def \svgwidth{.6\columnwidth}
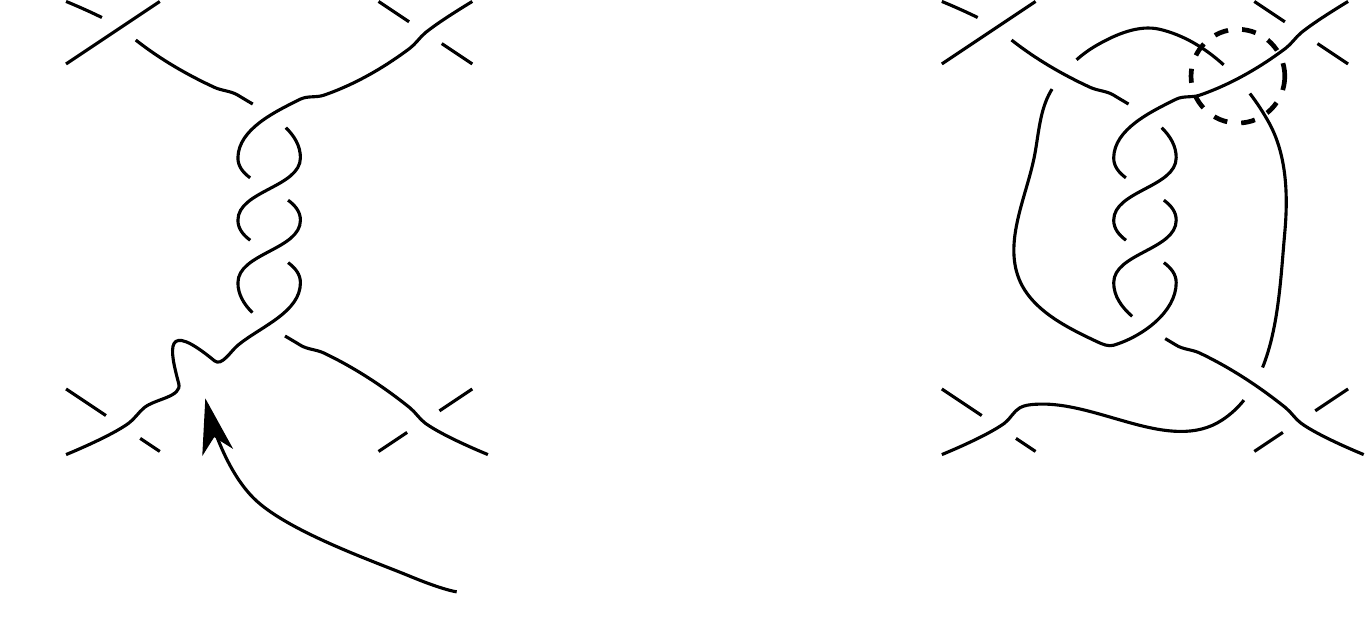
\caption{Local isotopy that turns a near-alternating diagram (left) into an almost alternating diagram (right). The single crossing change to be made in the resulting diagram to make it alternating is in the dashed circle on the right. The two diagrams agree everywhere except for the portion shown.}
\end{figure}

\end{proof}

\subsection{Stable coefficients and Coarse volume}

Let $\alpha_{i, n}$ be the coefficient of $v^{d(n)+4i}$ of the \emph{reduced} colored Jones polynomial $\widehat{J}_K(v, n):=J_K(v, n)/J_{\vcenter{\hbox{\includegraphics[scale=.05]{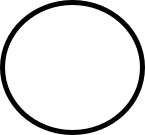}}}}(v, n)$, where $J_{\vcenter{\hbox{\includegraphics[scale=.05]{circ.png}}}}(v, n)$ is the $n$th colored Jones polynomial of the unknot, and let $\alpha'_{i, n}$ be the coefficient of $v^{d^*(n)-4i}$, so that $\alpha_{0, n}, \alpha_{1, n}, \alpha'_{1, n}, \alpha'_{0, n}$ are the first, second, penultimate, and last coefficient of $\widehat{J}_K(v, n)$, respectively.

\begin{defn} Let $i\geq 0$, the first $i$th coefficient (resp. last $i$th coefficient) of the reduced colored Jones polynomial is \emph{stable} if $\alpha_{i, j} = \alpha_{i, i+2}$ (resp. $\alpha'_{i, j} = \alpha'_{i, i+2}$) for all $j\geq i+2$. 
\end{defn} 

It is known that for an adequate knot, the first and last $i$th coefficients are stable for all $i \geq 0$ \cite{Arm}. The cases $i=0, 1$, and $2$ were first shown by \cite{Sto04, DL06}. They also gave explicit formulas for the stable coefficients from the all-$+$ and all-$-$ state graphs of an adequate diagram of a knot. These results were used to give a two-sided volume bound for hyperbolic alternating knots \cite{DL07}. Futer, Kalfagianni, and Purcell used these coefficients to give two-sided bounds on the volume of a hyperbolic, adequate knot \cite{FKP13}. These results establish that for an adequate knot that is hyperbolic, the stable coefficients of the colored Jones polynomial are \emph{coarsely related} to the volume as defined below. 

\begin{defn} Let $f, g: Z \rightarrow \mathbb{R}_+$ be functions from some (infinite) set $Z$ to the non-negative real numbers. We say that $f$ and $g$ are \emph{coarsely related} if there exist universal constants $C_1\geq 1$ and $C_2\geq 0$ such that 
\[C_1^{-1}f(x)-C_2 \leq g(x) \leq C_1f(x) + C_2 \ \ \ \forall x\in Z. \]  
\end{defn} 
The Coarse Volume Conjecture \cite[Question 10.13]{FKP13} predicts the existence of a function $B(K)$ of the coefficients of the colored Jones polynomial of a hyperbolic knot $K$, such that $B(K)$ is coarsely related to the hyperbolic volume $vol(S^3 \setminus K)$. Here the infinite set $Z$ is taken to be the set of hyperbolic knots.

We show that a near-alternating knot has stable first, second, penultimate, and last coefficients which are determined by state graphs of a near-alternating diagram. We give a two-sided bound on the volume of a highly twisted, near-alternating knot based on these coefficients. 

Let $\mathbb{G}$ be a graph without one-edged loops, an edge $e = (v, v')$ is called \emph{multiple} if there is another edge $e' = (v, v')$ in $\mathbb{G}$. The \emph{reduced graph} of $\mathbb{G}$, denoted by $\mathbb{G}'$, is obtained from $\mathbb{G}$ by keeping the same vertices but replacing each set of multiple edges between a pair of vertices $v, v'$ by a single edge. The \emph{first Betti number} of a graph,  denoted by $\chi_1(\mathbb{G})$, is the number $v-e+k$, where $v$ is the number of vertices of $\mathbb{G}$, $e$ is the number of edges of $\mathbb{G}$, and $k$ is the number of connected components of $\mathbb{G}$. 

\begin{thm} \label{thm:tail} Let $K$ be a link admitting a near-alternating diagram $D = \partial (F_G)$, where $G$ is a finite 2-connected, weighted planar graph with a single negatively-weighted edge of weight $r < 0$. Then 
\begin{enumerate}[(1)]
\item the first and second coefficient, $\alpha_{0, n}, \alpha_{1, n}$, respectively, of the reduced colored Jones polynomial $\widehat{J_K}(v, n)$ of a near-alternating link $K$ are stable. The last and penultimate coefficient, $\alpha'_{0, n}, \alpha'_{1, n}$, respectively, are also stable.
\item Write $\alpha = \alpha_{0, n}$ and $\beta = \alpha_{1, n}$, and write $\alpha' = \alpha'_{0, n}$ and $\beta' = \alpha'_{1, n}$ for $n>3$. We have $|\alpha| = 1$ and $|\beta| = \chi_1(s_{\sigma}(D)')$, where $\sigma$ is the Kauffman state giving the state surface $F_G$ and $\chi_1(s_{\sigma}(D)')$ is the first Betti number of the reduced graph of $s_{\sigma}(D)$. Similarly, we have $|\alpha'|=1$ and $|\beta'| = \chi_1(s_{-}(D)')$.
\end{enumerate} 

Furthermore,  if the diagram $D$ is also prime and twist-reduced with more than 7 crossings in each twist region, then $K$ is hyperbolic, and
\[.35367(|\beta|+|\beta'| -1)  < vol(S^3\setminus K) < 30v_3(|\beta|+|\beta'| - 2). \]
Here $v_3\approx 1.0149$ is the volume of a regular ideal tetrahedron.
In other words, there is a function on the stable coefficients of $K$ which is coarsely related to the volume of $S^3\setminus K$. 
\end{thm} 

 The second stable coefficient $\beta$ is given in terms of the Euler characteristic of the state surface $F_G = S_{\sigma}(D)$ in a formula similar to those given in \cite{DL06, DL07} for adequate knots. Numerical experiments suggest that more coefficients of the reduced colored Jones polynomial should be stable. However, we do not pursue this question in this paper. 
 For the two-sided bound on volume, we use estimates based on the twist number of a knot diagram developed in \cite{FKP08} using the works of Adams, Agol, Lackenby, and Thurston. For other examples of volume estimates based on link diagrams, see \cite{BMPW15} and \cite{Gia15, Gia16}.

\subsection*{Organization}
In Section \ref{sec:prelim}, we give a definition of the colored Jones polynomial in terms of skein theory and summarize elementary results needed for Theorem \ref{thm:degree}, which is proven in Section \ref{sec:jslope} by way of Theorem \ref{thm:bracketdegree}. In Section \ref{sec:jsurface}, we prove Theorem \ref{thm:jsurface} by computing the boundary slope and the Euler characteristic of the surface $F_G$. We show Theorem \ref{thm:naknoadequate},  which says that a near-alternating knot is not adequate in Section \ref{sec:nadequate}. Finally, we compute stable coefficients and give a coarse volume bound to prove Theorem \ref{thm:tail} in Section \ref{sec:cvolume}. 

\subsection*{Acknowledgements}
This is a side project that grew out of a project with Roland van der Veen. I would like to thank him for our conversations which made this spin-off possible. I would like to thank Cameron Gordon for suggesting the name ``near-alternating." I would also like to thank Efstratia Kalfagianni, Stavros Garoufalidis, and Oliver Dasbach for their comments and encouragement on this work, and for their hospitality during my visits. Lastly, I would like to thank Mustafa Hajij for interesting discussions on stability properties of the colored Jones polynomial, Adam Lowrance for pointing out that near-alternating knots are almost alternating, and Joshua Howie for interesting conversations on the Slope Conjecture. I would also like to acknowledge the support by NSF grant DMS-1502860.

\section{Graphical skein theory}  \label{sec:prelim}
We follow the approach of \cite{Lic97} in defining the Temperley-Lieb algebra. The original source of the formulas is \cite{MV94}. Let $F$ be an orientable surface (with or without boundary) which has a finite (possibly empty) collection of points specified on $\partial F$. A link diagram on $F$ consists of finitely many arcs and closed curves on $F$ such that 
\begin{itemize}
\item There are finitely many transverse crossings with an over-strand and an under-strand. 
\item The endpoints of the arcs form a subset of the specified points on $\partial F$. 
\end{itemize} 
Two link diagrams on $F$ are isotopic if they differ by a homeomorphism of $F$ isotopic to the identity. The isotopy is required to fix $\partial F$. 

\begin{defn}\label{defn:skein} Let $A$ be a fixed complex number. The \emph{linear skein module} $\mathcal{S}(F)$ of $F$ is the vector space of formal linear sums over $\mathbb{C}$ of isotopy classes of link diagrams in $F$ quotiented by the relations 
\begin{enumerate}[(i)]
\item $D \sqcup \vcenter{\hbox{\includegraphics[scale=.10]{circ.png}}} = (-A^2-A^{-2}) D,$ \text{ and}
\item $ \vcenter{\hbox{\includegraphics[scale=.2]{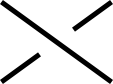}}} = A^{-1} \ \vcenter{\hbox{\includegraphics[scale=.2]{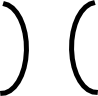}}} \ + A \ \vcenter{\hbox{\includegraphics[scale=.2]{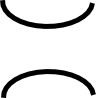}}} \ .$
\end{enumerate} 

\end{defn} 
We consider the linear skein module $\mathcal{S}(\mathcal{D}^2, n)$ of the disc $\mathcal{D}^2$, visualized as a square, with $n$ points specified on its top and bottom boundary. For $D_1, D_2 \in \Sk(\mathcal{D}^2,n)$, there is a natural multiplication operation $D_1\cdot D_2$ defined by identifying the top boundary of $D_1$ with the bottom boundary of $D_2$.  This makes $\mathcal{S}(\mathcal{D}^2, n)$ into an algebra $TL_n$, called the \emph{Temperley-Lieb algebra}. The algebra $TL_n$ is generated by crossing-less matchings $1_n, e^{1}_n, \ldots, e^{n-1}_n$ of $2n$ points of the form shown in Figure \ref{fig:TLgen}. 
\begin{figure}[ht] 
 \centering
    \def\svgwidth{.6\columnwidth}
    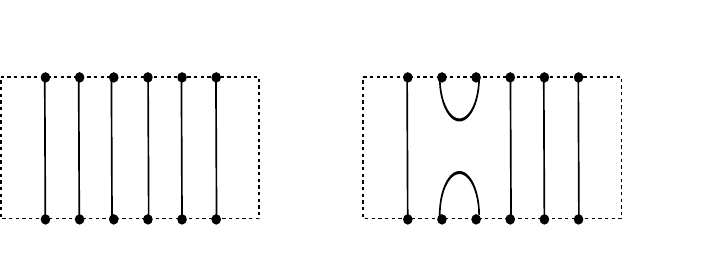
\caption{An example of the identity element $1_n$ and a generator $e^i_n$ of $TL_n$ for $n=6$ and $i=2$. \label{fig:TLgen}}
\end{figure}

Suppose that $A^4$ is not a $k$th root of unity for $k\leq n$. There is an element $\jwproj_n$ in $TL_n$ called the \emph{Jones-Wenzl idempotent}, which is uniquely defined  by the following properties. For the original reference where the idempotent was defined and studied, see \cite{Wen87}. 
\begin{enumerate}[(i)]
\item $\jwproj_n \cdot e^i_n = e^i_n \cdot \jwproj_n =0$ for $1 \leq i \leq n-1$. \label{list:prop1}
\item $\jwproj_n -1_n $ belongs to the algebra generated by $\{e^1_n, e^2_n,\ldots, e^{n-1}_n\}$. 
\item $\jwproj_n \cdot \jwproj_n = \jwproj_n$, 
\item Let $\mathcal{S}(\mathbb{R})$ be the linear skein of the plane. The image of \ $\jwproj_n$ in $\mathcal{S}(\mathbb{R})$ obtained by joining the $n$ boundary points on the top with the those at the bottom is equal to 
\[ \triangle_n = (-1)^n[n] \cdot \text{the empty diagram on $\mathbb{R}$}, \] \label{list:prop4}
\end{enumerate}
where $[n]$ is the \emph{quantum integer} defined by
\[ [n]:= \frac{A^{2(n+1)} - A^{-2(n+1)}}{A^{2}-A^{-2}}. \] 

From the defining properties, the Jones-Wenzl idempotent also satisfies a recursion relation and two other identities as indicated in Figures \ref{fig:jw1} and \ref{fig:jw2}. 
\begin{figure}[ht]
\begin{equation} \label{eq:jwrecursive}
 \centering
    \def\svgwidth{.9\columnwidth}
    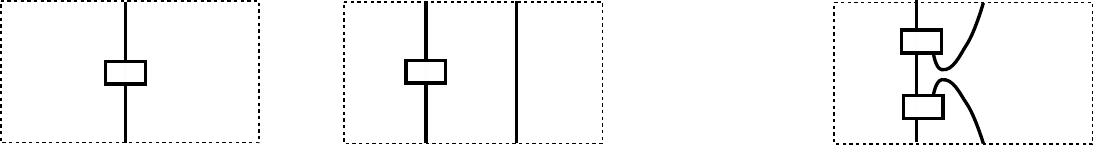
\end{equation}
\caption{\label{fig:jw1} A recursive relation for the Jones-Wenzl projector.}\end{figure} 
\begin{figure}[ht]
   \centering
   \begin{equation} \label{eq:jwidentity}
    \def\svgwidth{.5\columnwidth}
    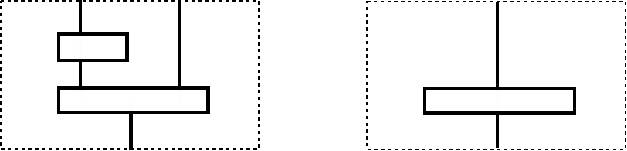
    \end{equation}
\caption{\label{fig:jw2} The larger projector absorbs the smaller one. }
\end{figure}

\begin{defn}\label{defn:cp}
Let $D$ be a diagram of a link $K\subset S^3$ with $k$ components. For each component $D_i$ for $i \in \{1,\ldots, k\}$ of $D$ take an annulus $A_i$ via the blackboard framing.
 Let 
\[ f_D: \underbrace{\Sk(S^1\times I) \times \cdots \times \Sk(S^1 \times I)}_{k \text{ times }} \rightarrow \Sk(\mathbb{R}^2),     \] be the map which sends a $k$-tuple of elements $(s_1, \ldots, s_k)$ to $\mathcal{S}(\mathbb{R}^2)$ by immersing the collection of annuli containing the skeins in the plane such that the over- and under-crossings of $D$ are the over- and under-crossings of the annuli. 

The Kauffman bracket $\langle \mathcal{S} \rangle$ of a skein element $\mathcal{S}$ in $\Sk(\mathbb{R}^2)$ is the polynomial multiplying the empty diagram after reducing by the skein relation of Definition \ref{defn:skein}. %Here, the Kauffman bracket is extended by linearity.
 The \emph{$n$th unreduced colored Jones polynomial} $J_K(v, n)$ may be defined as 
\[ J_K(v, n) := ((-1)^{n-1}v^{n^2-1})^{\omega(D)} \left\langle f_D\underbrace{\left(\vcenter{\hbox{\includegraphics[scale=.3]{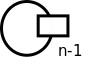}}},  \vcenter{\hbox{\includegraphics[scale=.3]{jwprojc.png}}}, \cdots, \vcenter{\hbox{\includegraphics[scale=.3]{jwprojc.png}}} \right)}_{k \text{ times }} \right\rangle |_{A = v^{-1}}.\] 

Note that this gives $J_{\vcenter{\hbox{\includegraphics[scale=.05]{circ.png}}}}(v, n+1) = (-1)^{n}\frac{v^{-2(n+1)}-v^{2(n+1)}}{v^{-2}-v^2}$ as the normalization for the colored Jones polynomial of the unknot.

 We will denote the skein 
 \[f_D\left(\vcenter{\hbox{\includegraphics[scale=.3]{jwprojc.png}}},  \vcenter{\hbox{\includegraphics[scale=.3]{jwprojc.png}}}, \cdots, \vcenter{\hbox{\includegraphics[scale=.3]{jwprojc.png}}} \right) \] by $D^{n-1}_{\jwproj}$ from now on.   
\end{defn}

Let 
\begin{figure}[H]
  \centering
    \def\svgwidth{.4\columnwidth}
    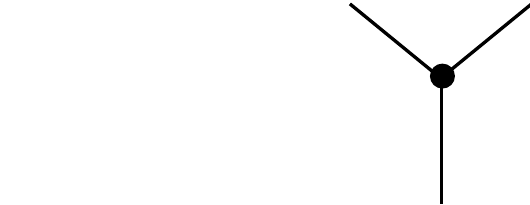, 
\end{figure}
with $x = \frac{a+b-c}{2}, z = \frac{a+c-b}{2},$ and $ y = \frac{b+c-a}{2}$.

We will use the identities indicated in Figure \ref{fig:fusion} to simplify $\langle D^n_{\jwproj}\rangle$.  

\begin{defn} A triple of non-negative integers $a, b, c$ is called \emph{admissible} if $a, b$, and $c$ are even and $|a-b|\leq c \leq a+b$.
\end{defn} 

\begin{figure}[H]
\centering
    \def\svgwidth{\columnwidth}
    \begin{equation}\label{eq:fusion}
    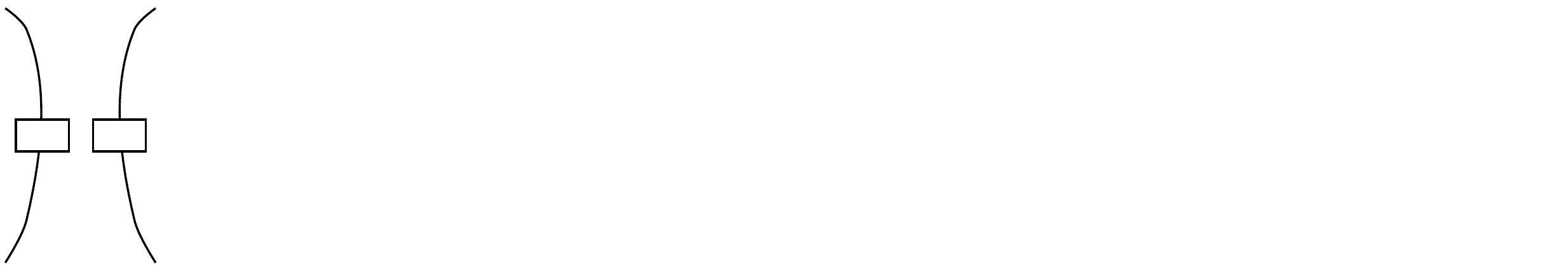.
    \end{equation}
\caption{The fusion and untwisting formulas. \label{fig:fusion}}
\end{figure}

For admissible $a, b, c$, let $\theta(a, b, c)$ be the Kauffman bracket of the skein shown in Figure \ref{fig:theta}.
\begin{figure}[ht]
\centering
    \def\svgwidth{.1\columnwidth}
   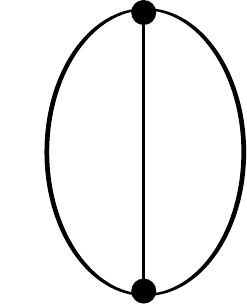
   \caption{\label{fig:theta}}
\end{figure}
 
\begin{lem}{\cite[Lemma 14.5]{Lic97}.}
Let $\triangle_n!:= \triangle_1 \cdot \triangle_2 \cdot \cdots \triangle_n$ and $\triangle_0! = 1$. Also let $x = \frac{a+b-c}{2}, z = \frac{a+c-b}{2},$ and $ y = \frac{b+c-a}{2}$, then $\theta(a, b, c)$ is given explicitly by the following formula. 
\begin{equation}
\theta(a, b, c)= \frac{\triangle_{x+y+z}!\triangle_{x-1}! \triangle_{y-1}! \triangle_{z-1}!}{\triangle_{y+z-1}!\triangle_{z+x-1}!\triangle_{x+y-1}!}.
\end{equation}
\end{lem}

Let $f$ be a rational function of $A$, and let $\deg{f}$ be the maximum degree of a Laurent series expansion of $f$ where the maximum power of $A$ is bounded. For convenience, we will list the degrees of $\triangle_c$ and $\theta(a, b, c)$ here. They are obtained by examining the formulas. 
\begin{align} \label{eq:degs}
\deg{\triangle_c} &= 2c, \text{ and } \notag \\ 
\deg{\theta(a, b, c)} &= a + b + c.
\end{align}

We will be using the following lemma from \cite{Arm}. 

\begin{defn}
Let $\Sk$ be a crossing-less skein in $\Sk(\mathbb{R})$ decorated by Jones-Wenzl idempotents $\jwproj_n$, and consider the skein $\overline{\Sk}$ obtained from $\Sk$ by replacing each of the idempotents by the identity $1_n$, so $\overline{\Sk}$ consists of disjoint circles. The skein $\Sk$ is called \emph{adequate} if no circle in $\overline{\Sk}$ passes through any of the regions previously decorated by an idempotent more than once.
\end{defn}

\begin{lem}[{\cite[Lemma 4]{Arm}}]\label{lem:jwad} Let $\Sk \in \Sk(\mathbb{R}^2)$ be a skein decorated by Jones-Wenzl idempotents $\jwproj_n$, and $\overline{\Sk}$ be the skein obtained by replacing each Jones-Wenzl idempotent by the identity element $1_n$, then
\[\deg\langle \Sk \rangle \leq \deg\langle \overline{\Sk} \rangle.   \] 
If $\Sk$ is a crossing-less skein that is adequate, then 
\[\deg\langle \Sk \rangle = \deg\langle \overline{\Sk} \rangle. \] 
\end{lem}

We also use an additional identity from \cite{MV94}. 
\begin{lem}[{\cite[Lemma 4]{MV94}}] \label{lem:jwid3}
For $y\geq 1$, 
    \def\svgwidth{.6\columnwidth}
\begin{equation} 
    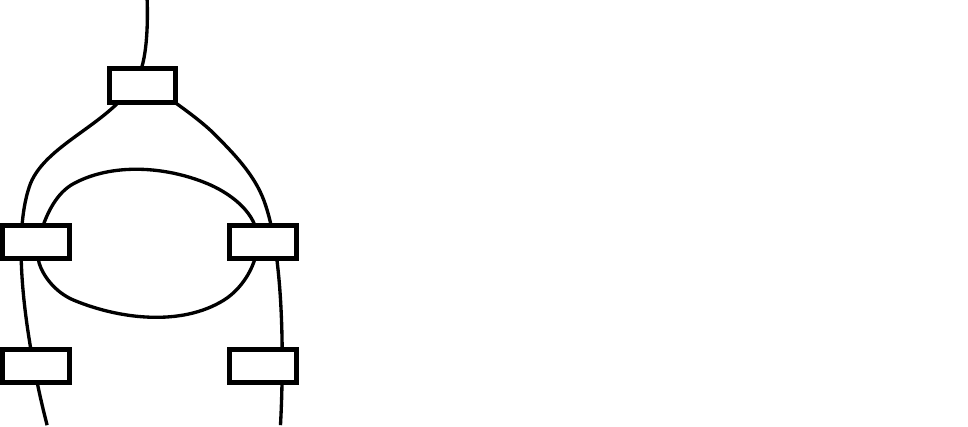
    \end{equation}
\end{lem}
The slight difference with \cite{MV94} in the coefficient multiplying the right-hand side is due to our slightly different convention for the quantum integer. Their $[n]$ is $[n-1]$ in this paper.

\section{Jones slopes} \label{sec:jslope}

We prove Theorem \ref{thm:degree} in this Section. Let $H_n(D) = -h_{n+1}(D)+\omega(D)(n^2+2n)$. We will only deal with the Kauffman bracket from now on with the variable $A$. Theorem \ref{thm:degree} then follows from the following theorem. 

\begin{thm}\label{thm:bracketdegree} 
If $D$ is a near-alternating link diagram with a single negative twist region of weight $r<0$, then 
\begin{equation} \deg \langle D^n_{\jwproj} \rangle = H_n(D) + 2r(n^2+n). \end{equation}
\end{thm} 

\subsection{Overview}
Our main strategy is to find a suitable state sum for $\langle D^n_{\jwproj} \rangle$ which has a degree-dominating term. If $D$ is near-alternating, we may simplify the sum and disregard many of the terms whose skeins evaluate to zero. This is done in Section \ref{subsec:simplifyss}. In Section \ref{subsec:degree-dominating}, we highlight the term in the state sum which will be shown to be degree-dominating. The most laborious step of the proof comes from bounding the degree of a term coming from another state $\sigma$. We do this in Section \ref{subsec:sigmab}, where we characterize the crossings on which $\sigma$ chooses the $-$-resolution by Lemma \ref{lem:count}. The reason why this gives a bound on the degree is given by Lemma \ref{lem:sigmab}. This leads to the important corollary, Lemma \ref{lem:wcount}, which we can apply to the case where $D$ is a near-alternating diagram to bound the degree of the term in the state sum corresponding to $\sigma$. Finally in Section \ref{subsec:complete} we put the estimates together to finish the proof of Theorem \ref{thm:bracketdegree}. Upon first reading the reader may skip the proof of Lemma \ref{lem:count} to get a sense of how it is applied. 

\subsection{Simplifying the state sum} \label{subsec:simplifyss}

Let $D$ be a near-alternating link diagram, which means that it has a single negative twist region of weight $r<0$. We fix $n$. Given the skein $ D^n_{\jwproj}$, slide the idempotents along the link strands and make copies until there are four idempotents framing the negative twist region. See Figure \ref{fig:framenegt} below. 

\begin{figure}[ht]  
\centering
    \def\svgwidth{.3\columnwidth}
    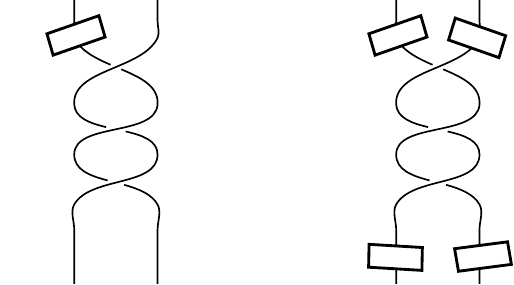
    \caption{\label{fig:framenegt} Framing the negative twist region with $r = -3$. }
\end{figure}

By the fusion and untwisting formulas \eqref{eq:fusion}, we may fuse the two strands of the negative twist region and get rid of the crossings. This results in a sum over the fusion parameter $a$ such that the triple $a, n, n$ is admissible. For a fixed $a$ consider a Kauffman state $\sigma$ on the set of remaining crossings. Applying $\sigma$ results in a skein $\Sk^a_{\sigma}$ that is the disjoint union of a connected component $J^a_{\sigma}$ decorated by Jones-Wenzl idempotents with circles as shown in Figure \ref{fig:case}. Let 
\begin{align*}
 \sgn(\sigma) &= \# \text{ of crossings on which $\sigma$ chooses the $+$-resolution} \\
 &- \# \text{ of crossings on which $\sigma$ chooses the $-$-resolution}. 
\end{align*} 

We have 
\begin{align} \label{eq:gssum}
\langle D^n_{\jwproj} \rangle &=  \sum_{\sigma, \ a \ : \ a, \ n, \ n \text{ admissible }} \frac{\triangle_a}{\theta(n, n, a)} ((-1)^{n-\frac{a}{2}}A^{2n-a+n^2-\frac{a^2}{2}})^{r} A^{\sgn(\sigma)} \langle \Sk^{a}_{\sigma} \rangle.  \\
\intertext{To simplify notation let $d(a, r)=r(2n-a+n^2-\frac{a^2}{2})$, and we write }
\langle D^n_{\jwproj} \rangle &= \sum_{\sigma, \ a \ : \ a, \ n, \ n \text{ admissible }} \frac{\triangle_a}{\theta(n, n, a)} (-1)^{rn-r\frac{a}{2}} A^{d(a, r) + \sgn(\sigma)} \langle J^a_{\sigma} \ \sqcup  \text{ disjoint circles}\rangle. \label{eq:statesum}
\end{align}

 After isotopy, we may assume that $J^a_{\sigma}$ has the form shown in Figure \ref{fig:case}, since other states evaluate to 0 by the Kauffman bracket with a cup/cap composed with an idempotent.

\begin{figure}[ht]  
\centering
    \def\svgwidth{.3\columnwidth}
    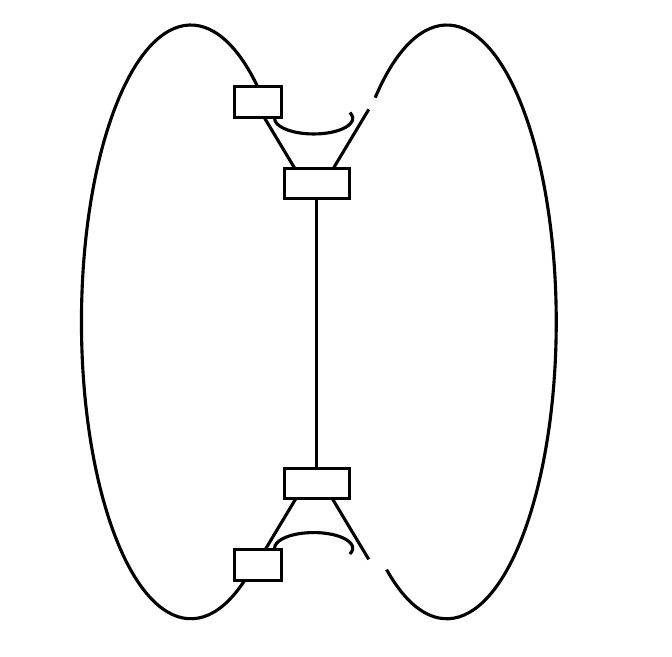
\caption{\label{fig:case} Let $0 \leq c \leq n$, the skein $J^a_{\sigma}$, which is the connected component decorated by the Jones-Wenzl idempotents is shown, where $\sigma$ has $2c$ split strands. The rest are disjoint circles.}
\end{figure}

\begin{defn} We say that the Kauffman state $\sigma$ has \emph{$2c$ split strands}, if after isotoping  the connected component $J^a_{\sigma}$ in  $\Sk^a_{\sigma}$ to the form in Figure \ref{fig:case}, there are $2c$ split strands connecting the top and bottom pairs of Jones-Wenzl idempotents. 
\end{defn} 

To further reduce the number of terms to consider in the sum of \eqref{eq:statesum}, we prove the following lemma. 

\begin{lem} \label{lem:localzero}
Consider a skein $\Sk$ with the following local picture. 

\begin{figure}[H]
 \centering
    \def\svgwidth{.3\columnwidth}
    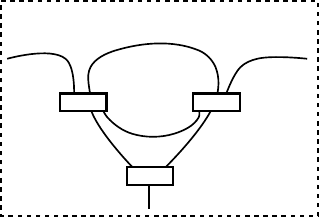
\end{figure} 
The skein is zero if  $\frac{a}{2}-c > 0$. 
\end{lem}
\begin{proof}
Note that $y = z = n-x$. If $\frac{a}{2}-c > 0$, then $n-c-x>0$, and the skein $\Sk$ is not adequate since we have a circle passing through the same idempotent twice, see Figure \ref{fig:localh} for an example of the circle. 

\begin{figure}[ht]
 \centering
    \def\svgwidth{.3\columnwidth}
    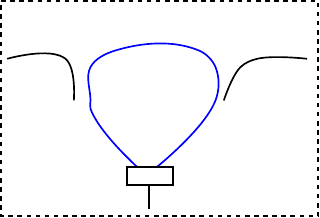
    \caption{\label{fig:localh} The circle passing through the same idempotent twice is shown in blue.}
\end{figure} 

Now if $x$ is zero, we can slide the top two idempotents down to the bottom one by \eqref{eq:jwidentity} and get a cap composed with a idempotent which gives 0 for the skein. When $x \not=0$, we show by induction on $x$ that every term in the sum of the skein from repeatedly expanding the idempotent via \eqref{eq:jwrecursive} has a cap composed with an idempotent after sliding by \eqref{eq:jwidentity}. Thus, every term in the sum is zero and $\langle \Sk\rangle$ is zero. 
Suppose $x =1$, there are two idempotents and therefore four terms in the sum from expanding via $\eqref{eq:jwrecursive}$, see Figure \ref{f.fourex}. This takes care of the base case: For any $n, c$ such that $n-c-1 > 0$, we have that $\langle  \Sk \rangle = 0$.  

\begin{figure}[ht]
 \centering
    \def\svgwidth{\columnwidth}
    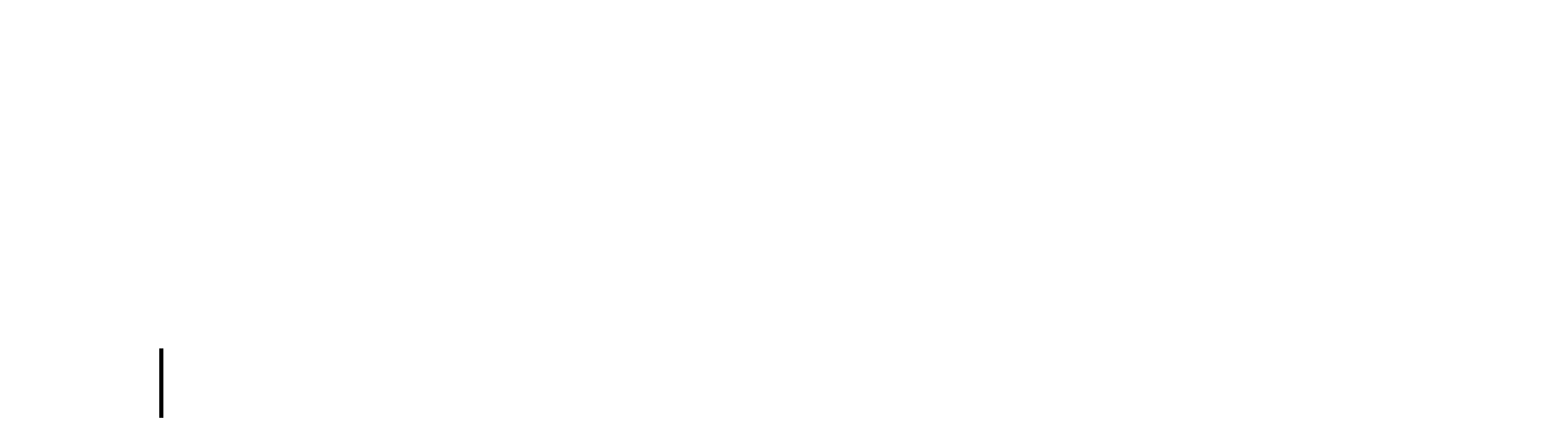
\caption{\label{f.fourex}The 4 terms in the expansion of $\Sk$ via the recursion relation \eqref{eq:jwrecursive} when  $x = 1$.}
\end{figure} 
 
Now suppose that $x=k+1$ and we have that every term in the expansion of $\Sk$ with $n-c-x > 0$ evaluates to 0 by the induction hypothesis for $x=k$. We expand the pair of idempotents to get the panel of four figures in Figure \ref{fig:induction}. 

\begin{figure}[H]
\centering
    \def\svgwidth{1.1\columnwidth}
    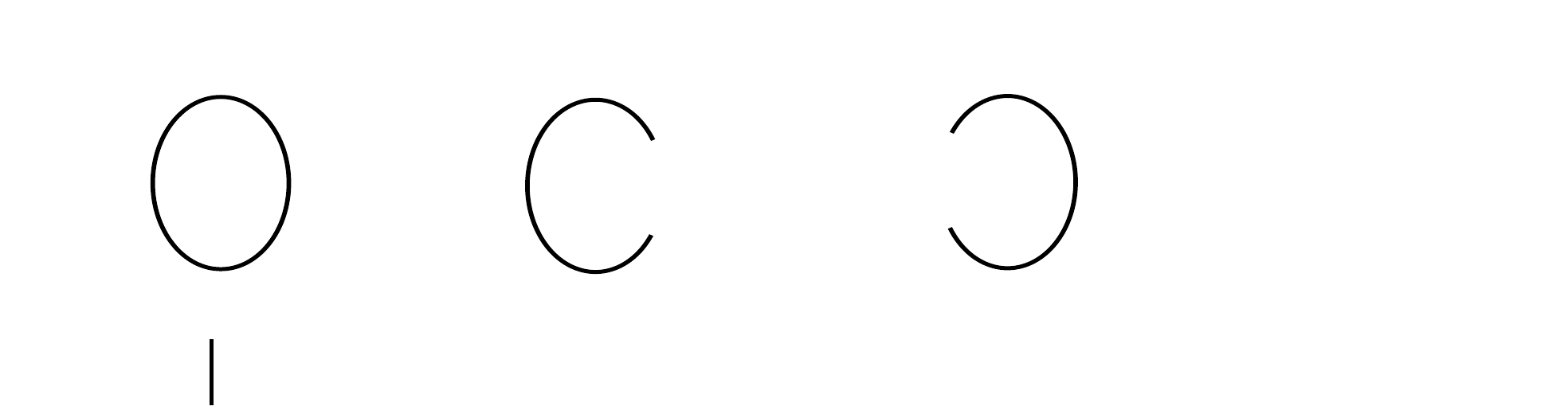
\caption{\label{fig:induction} If $x = k+1$, expand and then apply the induction hypothesis to the first 3 figures.}
\end{figure}  
The first three figures clearly reduce to that of the case $x=k$ and $n-1-c-(x-1)>0$. We simplify the last figure by applying Lemma \ref{lem:jwid3}. This is shown in Figure \ref{fig:induction2}.
\begin{figure}[H]
\centering
    \def\svgwidth{.7\columnwidth}
    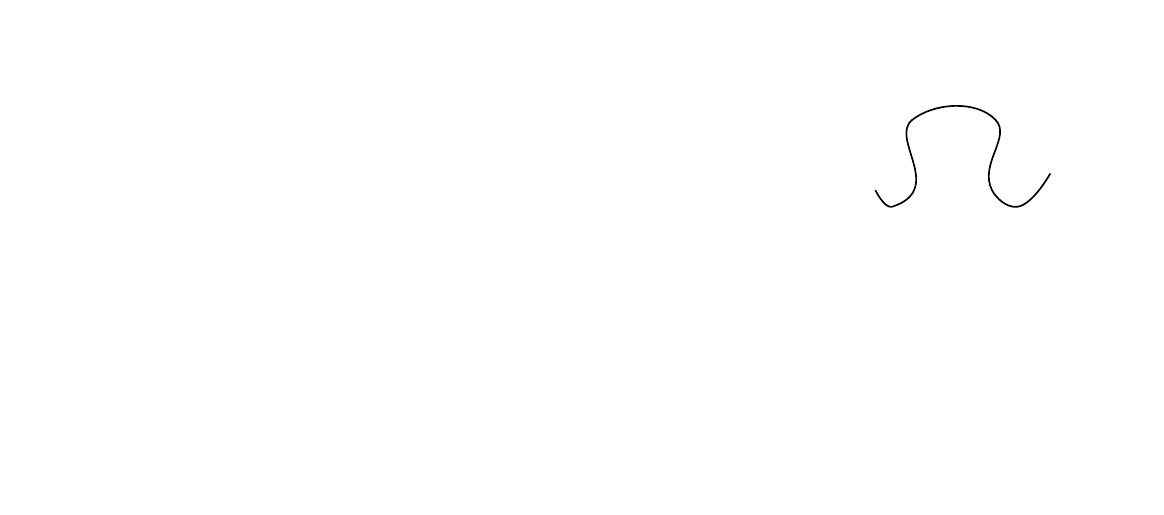
    \caption{\label{fig:induction2}}
\end{figure} 
If $x-2=0$, then we are done. Otherwise, we again expand the top pair of idempotents to get another panel of 4 figures as shown in Figure \ref{fig:induction3}.
\begin{figure}[H]
\centering
    \def\svgwidth{1.1\columnwidth}
    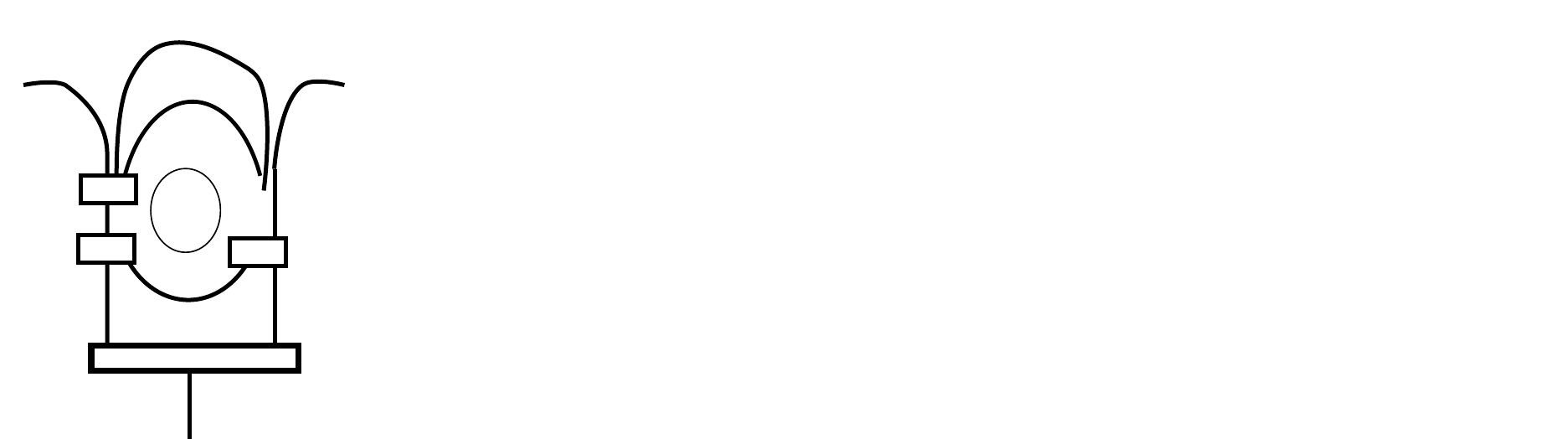
    \caption{\label{fig:induction3}}
\end{figure} 
The first three cases reduce to the case $x = k-1$ with $n-2-c-(x-2)>0$. For the last one we repeat the step of Figure \ref{fig:induction2} using Lemma \ref{lem:jwid3} to keep reducing $x$ until it is 0. Repeat with the step of expanding the top pair of idempotents as in Figure \ref{fig:induction3} and the step of Figure \ref{fig:induction2} as needed.  
\end{proof}

By Lemma \ref{lem:localzero}, we have that \eqref{eq:statesum} becomes 
\begin{align}
\langle D^n_{\jwproj} \rangle &= \sum_{\sigma, \ a \ : \ a, \ n, \ n, \text{ admissible }} \frac{\triangle_a}{\theta(n, n, a)}(-1)^{rn-r\frac{a}{2}} A^{d(a, r) + \sgn(\sigma)} \langle J^a_{\sigma} \sqcup  \text{ disjoint circles}\rangle \\ 
&= \sum_{\sigma, \ a \ : \ a, \ n, \ n, \text{ admissible}, \ \frac{a}{2} \leq c} \frac{\triangle_a}{\theta(n, n, a)} (-1)^{rn-r\frac{a}{2}} A^{d(a, r) + \sgn(\sigma)} \langle J^a_{\sigma} \sqcup  \text{ disjoint circles}\rangle.
\end{align} 

Now let 
\[\deg(\sigma, a) := \text{deg} \left( \frac{\triangle_a}{\theta(n, n, a)}(-1)^{rn-r\frac{a}{2}} A^{d(a, r) + \sgn(\sigma)} \langle J^a_{\sigma} \sqcup \text{disjoint circles} \rangle \right). \]

\subsection{The degree-dominating term in the state sum} \label{subsec:degree-dominating}

Consider the state $\sigma_+$ which chooses the $+$-resolution at all the crossings that remain in $D^n_{\jwproj}$ after getting rid of the negative twist region of weight $r<0$ using the fusion and the untwisting formulas. We have that $\Sk_{\sigma_+}^a$ has 0 split strands and thus $\langle J^a_{\sigma_{+}} \rangle = 0$ for all values of $a$ except $a=0$. A simple computation using Lemma \ref{lem:jwad} shows 
\begin{equation} \deg(\sigma_+, 0) = H_n(D) + 2r(n^2+n). \end{equation}
The strategy to prove Theorem \ref{thm:bracketdegree} is then to show that 
\begin{equation}
\deg(\sigma, a) < \deg(\sigma_+, 0) 
\end{equation}
for any other Kauffman state $\sigma$ and $a$ contributing to the state sum.

Given $a$ and $\sigma$ with $2c$ split strands such that $\frac{a}{2} \leq c$, the skein $J^a_{\sigma}$ is adequate, and thus by Lemma \ref{lem:jwad} and \eqref{eq:degs}, 
\begin{equation} \label{eqn:aeqc}
\deg(\sigma, a) = a-2n+ d(a, r) + \sgn(\sigma) + \deg \langle \overline{\Sk^{a}_{\sigma}} \rangle,  
\end{equation} 
where $\overline{\Sk^{a}_{\sigma}}$ is the skein obtained from $\Sk^{a}_{\sigma}$ by replacing all the idempotents with the identity. From this we can see that if $\frac{a}{2} < c$ then $\overline{\Sk^{a}_{\sigma}}$ has fewer circles than $\overline{\Sk^{2c}_{\sigma}}$, thus \[ \deg(\sigma, a) < \deg(\sigma, 2c), \] and we may assume that $\frac{a}{2} = c$, see Figure \ref{fig:casefuntwist}.  

\begin{figure}[H]

\def \svgwidth{.8\columnwidth}
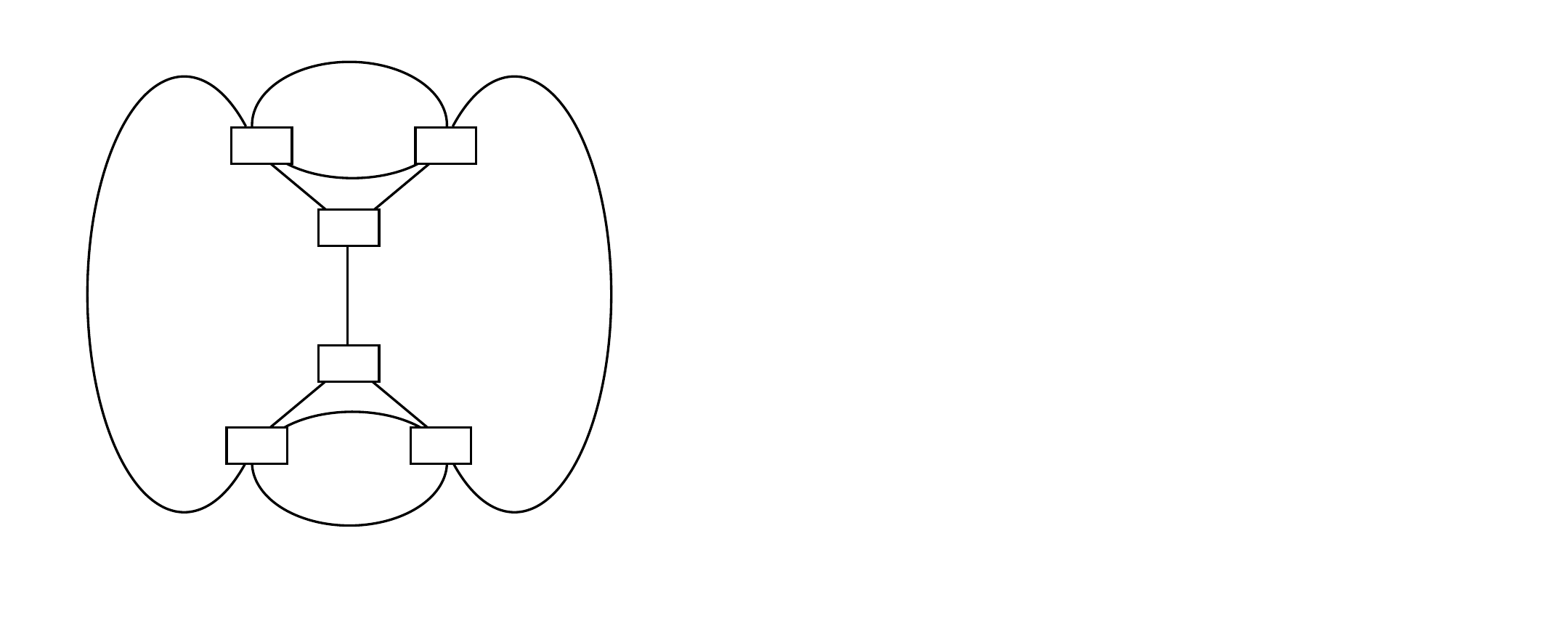
\caption{\label{fig:casefuntwist}}
\end{figure} 

In order to compare $\deg(\sigma, 2c)$ with $\deg(\sigma_+, 0)$, we use the concept of a sequence of states. 

\subsection{Crossings on which a state $\sigma \not= \sigma_+$ chooses the $-$-resolution} \label{subsec:sigmab}
In this section we characterize the set of crossings on which a state $\sigma \not= \sigma_+$ with $2c>0$ split strands chooses the $-$-resolution. We describe this by studying sequences of states from $\sigma_+$  to $\sigma$. The terminology of a sequence of states appears in \cite{Lic97}.

\begin{defn}
A \emph{sequence} $s$ of states starting at $\sigma_1$ and ending at $\sigma_f$ on a set of crossings in a skein $\Sk \in \Sk(\mathbb{R})$ is a finite sequence of Kauffman states $\sigma_1, \ldots, \sigma_f$, where $\sigma_{i}$ and $\sigma_{i+1}$ differ on the choice of the $+$-or $-$-resolution at only one crossing $x$, so that $\sigma_i$ chooses the $+$-resolution and $\sigma_{i+1}$ chooses the $-$-resolution at $x$.  \end{defn}

Let $s=\{\sigma_1, \ldots, \sigma_f \}$ be a sequence of states starting at $\sigma_1$ and ending at $\sigma_f$. Choosing the $-$-resolution at a crossing $\vcenter{\hbox{\includegraphics[scale=.2]{crossing1.png}}}$ corresponds to locally replacing 
$\vcenter{\hbox{\includegraphics[scale=.35]{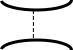}}}$ by $\vcenter{\hbox{\includegraphics[scale=.3]{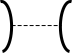}}}$ in the state graph. In each application from $\sigma_i$ to $\sigma_{i+1}$ either two circles of $\overline{\Sk_{\sigma_i}}$ merge into one or a circle of $\overline{\Sk_{\sigma_i}}$ splits into two. When two circles merge into one as the result of changing the $+$-resolution to the $-$-resolution, the number of circles of the skein decreases by 1 while the sign of the state decreases by 2. More precisely, let $\Sk_{\sigma}$ be the skein resulting from applying the Kauffman state $\sigma$, we have 
\begin{equation*} \sgn(\sigma_{i+1}) + \deg \langle \overline{\Sk_{\sigma_{i+1}}}\rangle = \sgn(\sigma_{i}) + \deg \langle \overline{\Sk_{\sigma_{i}}}\rangle -4, \end{equation*}
when a pair of circles merges from $\sigma_{i}$ to $\sigma_{i+1}$.

When a pair of circle is split from $\sigma_{i}$ to $\sigma_{i+1}$ in the sequence, we get instead
\begin{equation*} \sgn(\sigma_{i+1}) + \deg \langle \overline{\Sk_{\sigma_{i+1}}}\rangle = \sgn(\sigma_{i}) + \deg \langle \overline{\Sk_{\sigma_{i}}}\rangle. \end{equation*}

The above reasoning gives the following lemma which allows us to bound the degree $\sgn(\sigma_f) + \deg\langle \Sk_{\sigma_f} \rangle$ from applying a Kauffman state $\sigma_f$ to the crossings of a skein $\Sk$, by considering the number of pairs of circles that are merged in a sequence of states from $\sigma_1=\sigma_+$ to $\sigma_f$. 

\begin{lem} \label{lem:sigmab} Let $\Sk$ be a skein with crossings and $s = \{ \sigma_1, \ldots, \sigma_f \}$ be a sequence of Kauffman states on the crossings of $\Sk$. 
If $g$ is the number of pairs $(\sigma_i, \sigma_{i+1})$ in $s$ such that $\sigma_{i+1}$ merges a pair of circles in $\sigma_i$, then 
\begin{equation} \label{eq:mergelower}
\sgn(\sigma_{f}) + \deg \langle \overline{\Sk_{\sigma_{f}}} \rangle = \sgn(\sigma_{1}) + \deg \langle \overline{\Sk_{\sigma_{1}}} \rangle -4g. 
\end{equation} 
\end{lem} 

We use this to obtain an upper bound of $\deg(\sigma, 2c)$ by considering a sequence starting at $\sigma_+$ and ending at $\sigma$. We use the technical concept of the \emph{flow} of a Kauffman state through a set of crossings. 

\begin{defn} \label{d.throughstrands}
Let $x$ be a crossing and $x^n$ be the $n$-cable. Represent $x^n$ so that it is a skein in $\Sk(\mathcal{D}^2, 2n)$ and oriented as in the first figure of Figure \ref{fig:skeinflow}. Consider a Kauffman state $\sigma$ on $x^n$, and denote the skein resulting from applying $\sigma$ to $x^n$ by $x^n_{\sigma}$. We say that $\sigma$ has  $2k$ strands \emph{flowing through} the crossing $x$ if $x^n_{\sigma}$ has $2k$ arcs connecting $2k$ points on the top and the bottom. See Figure \ref{fig:skeinflow} for an example.
\end{defn}

Note that since $2n$ is even, there is always an even number of through strands. 

\begin{figure}[H]
\centering
    \def\svgwidth{.7\columnwidth}
    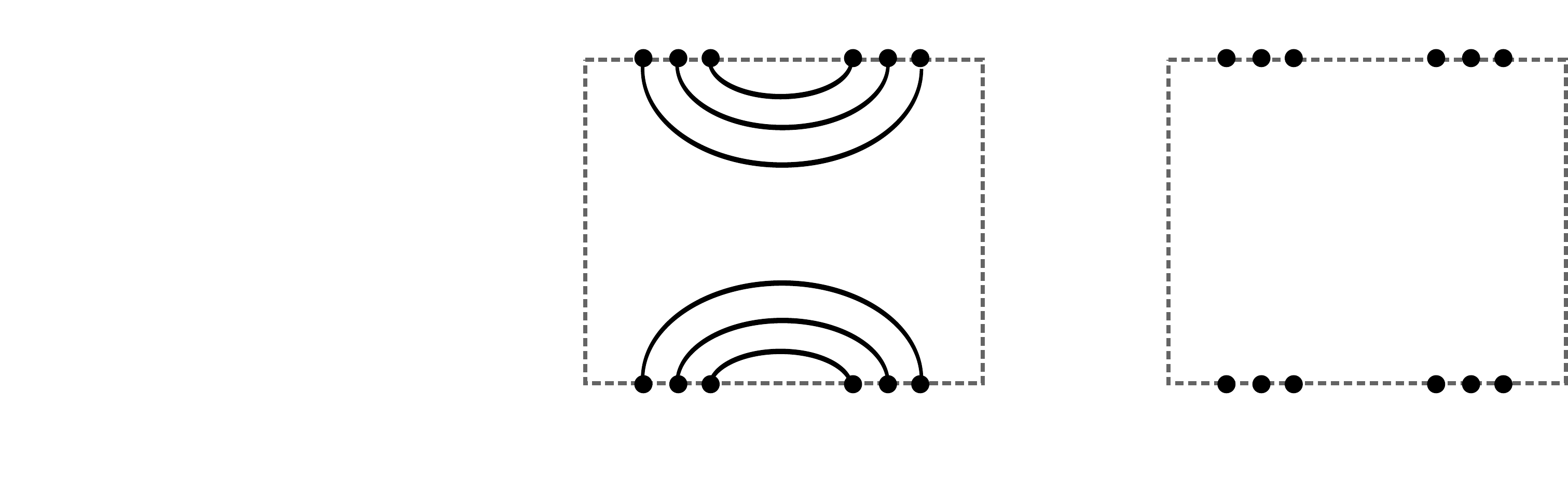
    \caption{\label{fig:skeinflow} Left: The 3-cabled crossing $x^3$. Middle: the all-$+$ state has 0 strands flowing through $x$. Right: a Kauffman state $\sigma$ here has 2 strands flowing through $x$. }
\end{figure}

\begin{rem} This is not a new concept. Works involving the Temperley-Lieb algebra have defined for an arbitrary crossing-less element of $TL_{m, n}$ (the algebra of skeins in a disk with $m$ points on top and $n$ points on the bottom) the quantity which counts the number of strands that connect $k$ points from the top to $k$ points on the bottom and called this quantity different names. For example, see \cite{Hog14} where the quantity is called the \emph{through-degree}, and \cite{Roz12}, where the quantity is called the \emph{width-deficit}. As far as the author is aware there does  not seem to be standard terminology for this quantity. The focus in this paper with Definition \ref{d.throughstrands}  is on the skeins from Kauffman states on a set of $n^2$ crossings, cabled from a single crossing.
\end{rem}

\subsection*{Notation and convention for graphical representation}

The following technical lemma, Lemma \ref{lem:count}, allows us to understand a sequence $s$ from $\sigma_+$ to $\sigma$, if $\sigma$ flows through a crossing with a certain number of strands. We essentially characterize the set of crossings on which $\sigma$ chooses the $-$-resolution. It is necessary to first establish some notations and labeling conventions. 

Firstly, we orient the disk $\mathcal{D}^2$ with $2n$ points on the top and bottom containing an $n$-cabled crossing $x^n$ as shown in Figure \ref{fig:markings} and identify it with $[-1, 1] \times [-1, 1]$. 
Let $U_1, \ldots, U_n$ be the set of arcs between the $2n$ points in the top half of the disk, innermost first, from applying the all-$+$ state on the set of crossings $x^n$. Similarly we have the lower arcs $L_1, \ldots, L_n$. The arcs cut up the disk into regions containing segments, which correspond to crossings in $x^n$ before taking the all-$+$ resolution. Let $C^u_i$ be the set of crossings whose corresponding segments in the all-$+$ state are between $U_i$ and $U_{i+1}$. Similarly, we have $C_i^{\ell}$, and the set of crossings corresponding to edges between $U_n$ and $L_n$ is denoted by $C^u_n = C^{\ell}_n$. See Figure \ref{fig:markings} for an illustration of these markings. 

\begin{figure}[ht]
\centering
    \def\svgwidth{.55\columnwidth}
    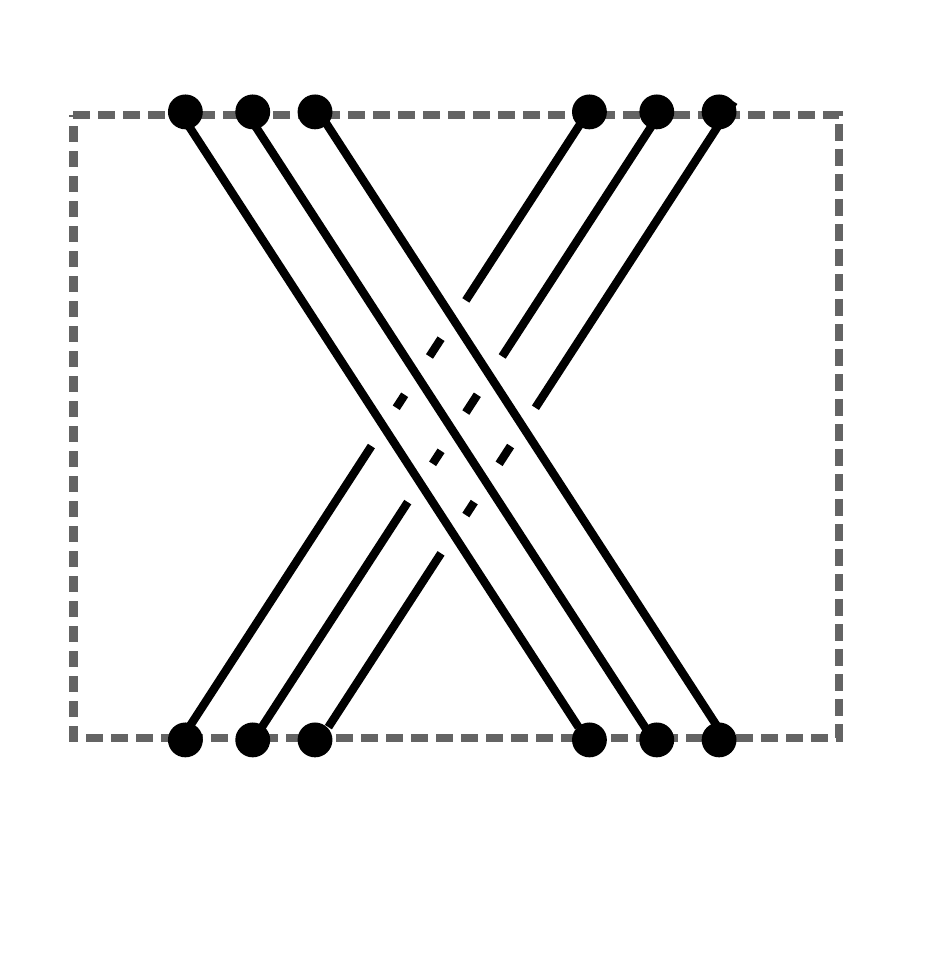
    \caption{\label{fig:markings} We indicate the division of the crossings into subsets deliminted by the regions and the orientation on the square.}
\end{figure}

We will represent a Kauffman state $\sigma$ on $x^n$, $x^n_{\sigma}$, by taking the all-$+$ state of $x^n$. Recall that this consists of the all-$+$ state circles and edges (dashed segments) corresponding to taking the $+$-resolution at every crossing. We make the following modification in order to represent an arbitrary Kauffman state $\sigma$ on $x^n$: 
\begin{enumerate}
\item If $\sigma$ chooses the $-$-resolution at a crossing, replace the corresponding segment in the all-$+$ state by a solid red edge. 
\item Remove all other edges from the state. 
\end{enumerate}

\begin{figure}[ht]
\centering
    \def\svgwidth{.3\columnwidth}
    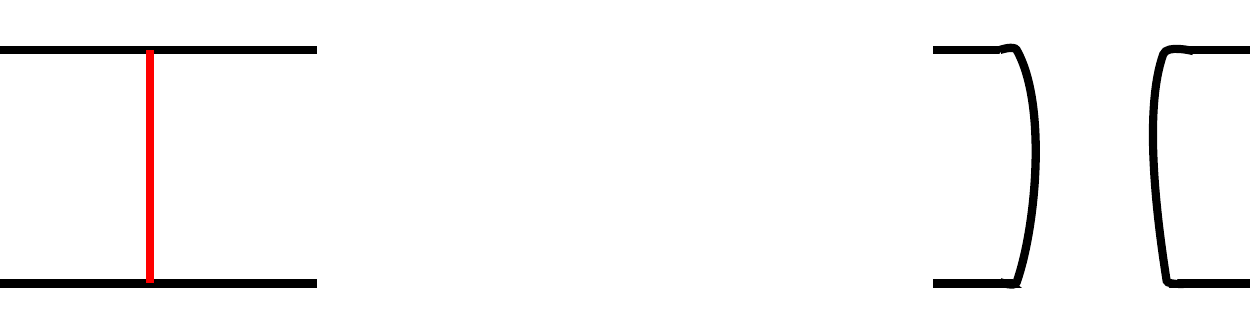
    \caption{\label{fig:skeinrepreplace}  The correspondence of a red edge with a Kauffman state choosing the $-$-resolution at a crossing corresponding to the red edge.}
\end{figure}

This representation will allow us to consider intersections of arcs in the disk $\mathcal{D}^2$ with $x^n_{\sigma}$. In particular, in this graphical representation of the Kauffman state $\sigma$ consisting of black arcs and red edges, intersection of an arc with a black arc counts as one intersection with the skein, and an intersection of an arc with a red edge counts as two. 

\begin{figure}[ht]
\centering
    \def\svgwidth{.6\columnwidth}
    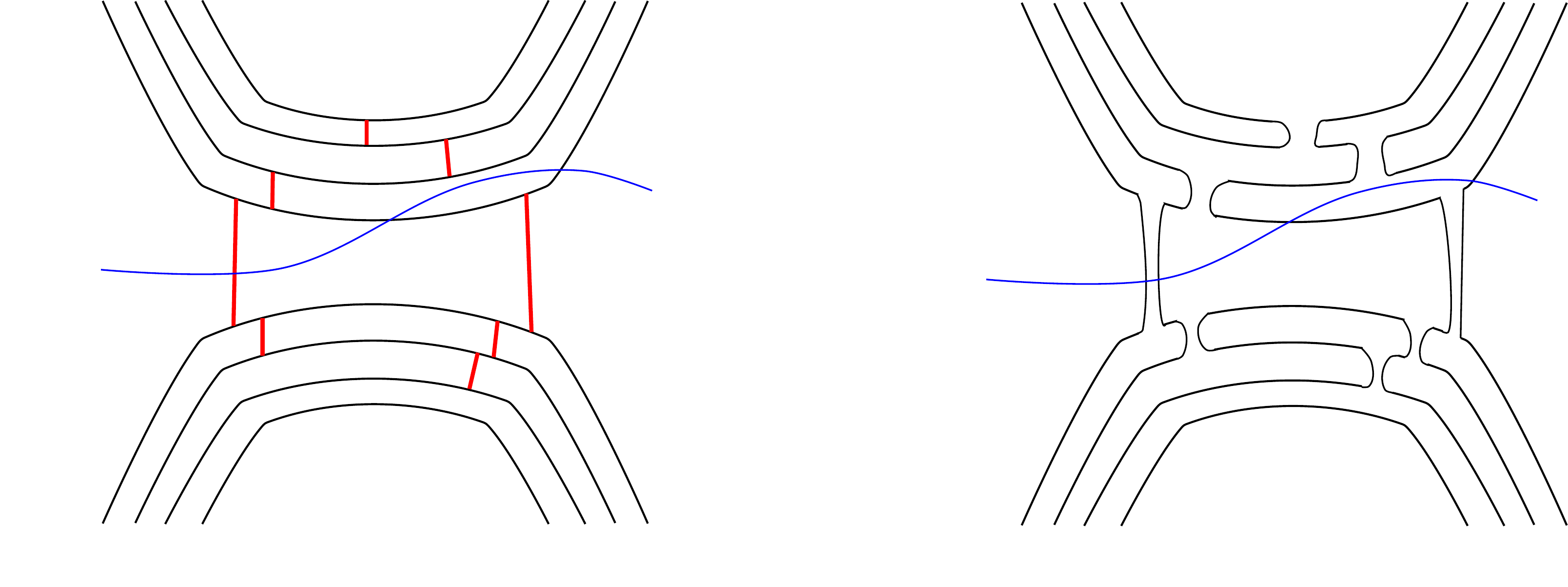
    \caption{\label{fig:repsigma}The picture shows an example of how one can recover the skein resulting from the application of a Kauffman state $\sigma$ from a representation of black arcs and red edges. Note how the intersection of the blue arc with the red edge counts as two intersections of the blue arc with the skein $x^n_{\sigma}$.}
\end{figure} 

With the orientation on the disk $\mathcal{D}^2$ shown as a square, it should be clear what we mean by an edge being on the left/right of another edge. This also explains what it means for a crossing in $x^n$ to be on the left/right of another crossing. We will frequently not distinguish between a crossing and its corresponding edge in the all-$+$ state whenever we are merely concerned with their relative positions.

\begin{lem} \label{lem:count} 
Let $\Sk$ be a skein with crossings, but without Jones-Wenzl idempotents, $\sigma$ be a Kauffman state on $\Sk$, and let $x^n$ be an $n$-cabled crossing contained in $\Sk$, with $x^n_\sigma$ the result of applying $\sigma$ to the crossings in $x^n$.  

\begin{enumerate}[(a)]
\item If $\sigma$ has $2k$ strands flowing through $x$, then $\sigma$ chooses the $-$-resolution on a set of $k^2$ crossings $C_{\sigma}$ of $x^n$, where $C_{\sigma} = \cup_{i=n-k+1}^n (u_i \cup \ell_i)$ is a union of crossings $u_i \subseteq C^u_i$ and $\ell_i \subseteq C^{\ell}_i$, such that

\begin{itemize} 
\item $u_i$, $\ell_i$ each has $k-n+i$ crossings for $n-k+1\leq i \leq n$.
\item For each $n-k+2\leq i \leq n$, and a pair of crossings $c, c'$ in $u_i$ (resp. $\ell_i$) whose corresponding red edges in the all-$+$ state of $x^n$ are adjacent, there is a crossing $c''$ in $u_{i-1}$ (resp. $\ell_{i-1}$), where the end of the red edge corresponding to $c''$ on $U_{i}$ (resp. $L_i$) lies between the ends of $c$ and $c'$. 
\end{itemize}

\item Consider a sequence $s= \{\sigma_+, \ldots, \sigma_f = \sigma\}$ of Kauffman states on the crossings of $\Sk$ and let $x^n$ be a set of $n$-cabled crossings in $\Sk$. Let $\sigma_+$ be the Kauffman state which chooses the $+$-resolution at every crossing in $x^n$, but agrees with $\sigma$ on all other crossings of $\Sk$. Suppose that in $\Sk_{\sigma_+}$, the $n$ arcs joining the top $2n$ points belong to $n$ circles disjoint from the $n$ arcs joining the bottom $2n$ points, which also belong to $n$ disjoint circles. Let $\sigma$ flow through $x$ with $2k$ through strands. Then sequence $s$ contains a subsequence $\sigma_+, \ldots, \sigma_f'$ with length $k^2$ such that $\Sk_{\sigma_f'}$ has $n$ fewer circles than $\Sk_{\sigma_+}$.
\end{enumerate} 
\end{lem}

As an example, if $n=3$ and $\sigma$ flows through a crossing $x$ with 4 strands, then $\sigma$ chooses the $-$-resolution on a subset of crossings of $x^n$ of the form as shown in Figure \ref{fig:flowconfig}. There may be other crossings on which $\sigma$ chooses the $-$-resolution, but the claim is that there must be a \emph{subset} of crossings on which $\sigma$ chooses the $-$-resolution of the form as described in Lemma \ref{lem:count}. 

\begin{figure}[ht]
\centering
    \def\svgwidth{.7\columnwidth}
    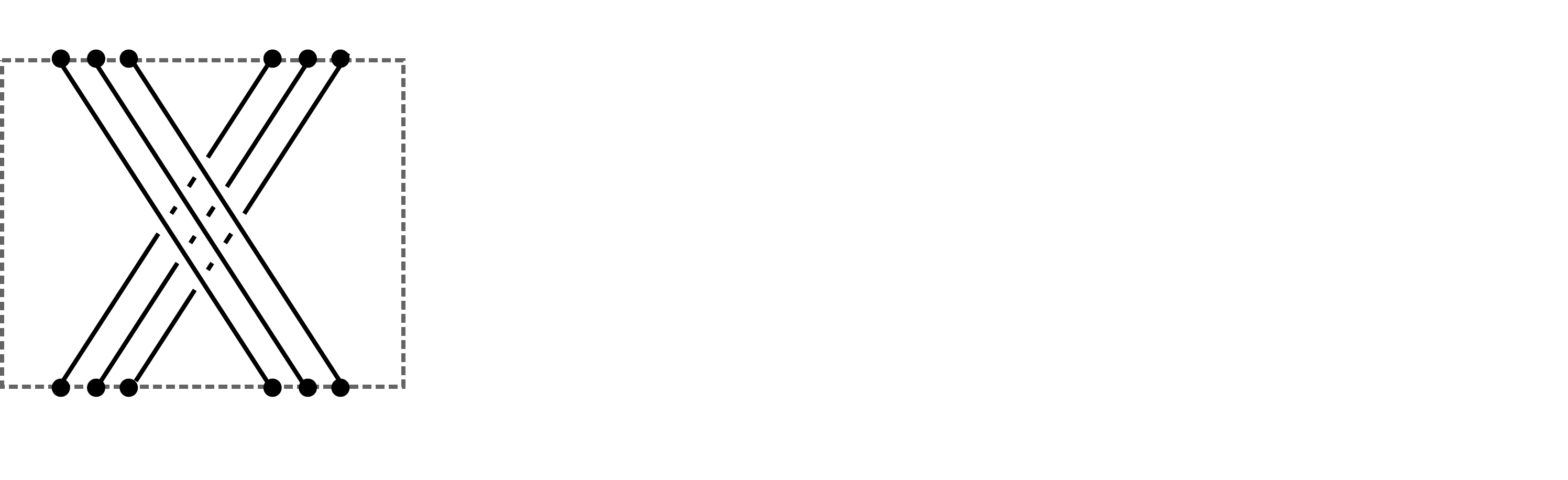
\caption{\label{fig:flowconfig} A subset of crossings on which $\sigma$ chooses the $-$-resolution satisfying the conditions of Lemma \ref{lem:count} is marked red in the second figure from the left.}
\end{figure}

In $C_{\sigma} = \cup_{2}^{3} (u_i\cup\ell_i)$, we have that $u_3=\ell_3$ contains 2 crossings and $u_2, \ell_2$ each contains 1 crossing. The red edge in the all-$+$ state of $x^n$ corresponding to the crossing in $u_2$ has an end on $U_3$ between the ends of the red edges corresponding to the two crossings in $u_3$. The same is true of the edge corresponding to the crossing in $\ell_2$. The total number of crossings in $C_{\sigma}$ is then $=4 =2^2$, which makes the total number of crossings of $x^n$ on which $\sigma$ chooses the $-$-resolution to be $\geq 4$.

\paragraph{\textbf{Proof of (a)}}

For a Kauffman state $\sigma$ which has $2k$ strands flowing through a crossing $x$, we first show that there are $k^2$ crossings on which $\sigma$ chooses the $-$-resolution. If we draw a line from the left end of the square to the right end, it must have $\geq 2k$ intersections with the curves of the skein resulting from applying the state. Isotope link strands so that the set of crossings $C^{\ell}_i$ for $1\leq i <n$ is  between the horizontal lines at height $h =-\frac{n-i}{n}$ and $h = -\frac{n-i+1}{n}$ (Recall that we identify the disk containing $x^n$ with $[-1, 1]\times [-1, 1]$). Similarly, isotope link strands so that the set of crossings $C^{u}_i$  for $1\leq i <n$ is between the horizontal lines at height $h=\frac{n-i}{n}$ and $h = \frac{n-i+1}{n}$. Now we isotope the crossings of $C^u_n = C^{\ell}_n$ so that it is between $h=-\frac{1}{n}$ and $h=\frac{1}{n}$, see Figure \ref{fig:line}.

Beginning with the set of crossings $C^{u}_n=C^{\ell}_n$, we see that $\sigma$ must choose the $-$-resolution on $k$ crossings, since the horizontal line $H$ at $h=0$ must intersect the resulting skein at least $2k$ times. 
Now isotope $H$ so that it enters and exits the region containing the crossings in $C^{u}_{n-1}$. 
 Then for $C^{u}_{n-1}$, $\sigma$ must choose the $-$-resolution on a set of $k-1$ crossings in $\frac{1}{n}<h<\frac{2}{n}$, since a pair of vertical lines provides 2 intersections with  $H$ between the two heights bounding the set of crossings in $C^{u}_{n-1}$. We repeat this argument for $C^{u}_{i}$ for $n-k+1\leq i \leq n-1$, isotoping $H$ to enter and exit the region bounding crossings of $C^{u}_i$ each time and noting that $H$ would already have $2(n-i)$ intersections with the strands of the skein. Then for each $i$, $\sigma$ must choose the $-$-resolution on $k-(n-i)$ crossings in $C^u_i$. 

\begin{figure}[H]
\centering
    \def\svgwidth{.9\columnwidth}
    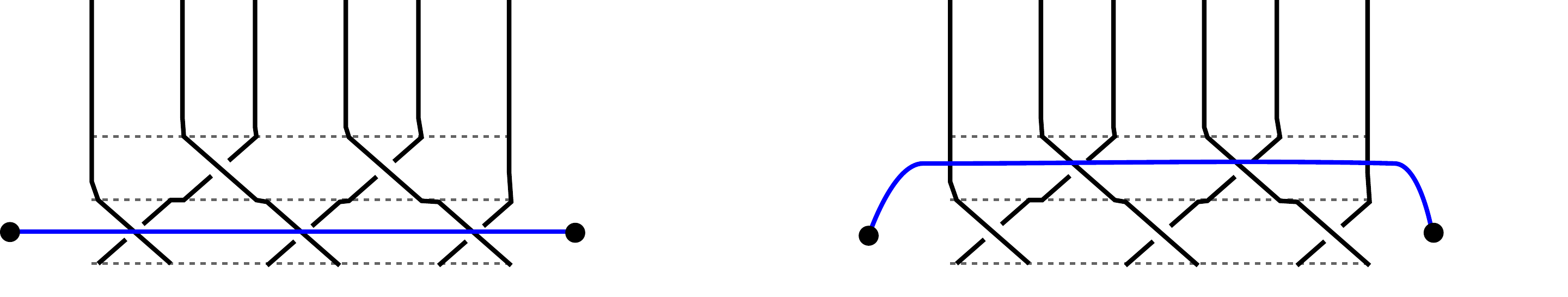
    \caption{\label{fig:line} The horizontal regions containing the crossings, the horizontal line $H$ (in blue) and the isotopies are shown for $C^u_n$ and $C^u_{n-1}$.}
\end{figure} 

The same argument works by symmetry when we consider lines intersecting the lower crossings $C^{\ell}_i$. Taking the sum over $n-k+1\leq i \leq  n$, the total number of these crossings on which $\sigma$ has to choose the $-$-resolution is

\[k + 2\sum_{i=1}^{k-1} i = k^2. \]

For the second part of (a) which specifies the structure of $C_{\sigma}$, we first prove that we can find a set of crossings $C_{\sigma}'$ of $x^n$  on which $\sigma$ chooses the $-$-resolution, where $C'_{\sigma} = \cup_{i=n-k+1}^n (u'_i \cup \ell'_i)$ is a union of crossings $u'_i \subseteq C^u_i$ and $\ell'_i \subseteq C^{\ell}_i$, such that

\begin{itemize}
\item[*] $u'_i$, $\ell_i'$ each has two crossings for $n-k+1 < i < n$, and one crossing for $i = n-k+1$. When $i=n$ and $k=1$, then $u'_n = \ell'_n$ has one crossing. Otherwise, $u'_n = \ell'_n$ and it has two crossings. 
\item[*] The two crossings in $u'_n = \ell'_n$ are furtherest possible in the sense that the two segments corresponding to the crossings in the all-$+$ state are furtherest possible. i.e., every segment corresponding to a crossing in $C^{u}_n = C^{\ell}_n$ on which $\sigma$ chooses the $-$-resolution lies between.  
\item[*] For each $n-k+1\leq i <n$, the end(s) of the segment(s) corresponding to the crossing(s) in $u'_i$ (resp. $\ell'_i$) on $U_{i+1}$ (resp. $L_{i+1}$) lie(s) between the two segments corresponding to the crossings in $u'_{i+1}$ (resp. $\ell'_{i+1}$). If there are two crossings in $u'_i$ (resp. $\ell'_i$), then they are the furtherest possible satisfying this condition.
\end{itemize} 

See Figure \ref{fig:firstset} for an illustration of these requirements.

\begin{figure}[H]
\centering
    \def\svgwidth{.2\columnwidth}
    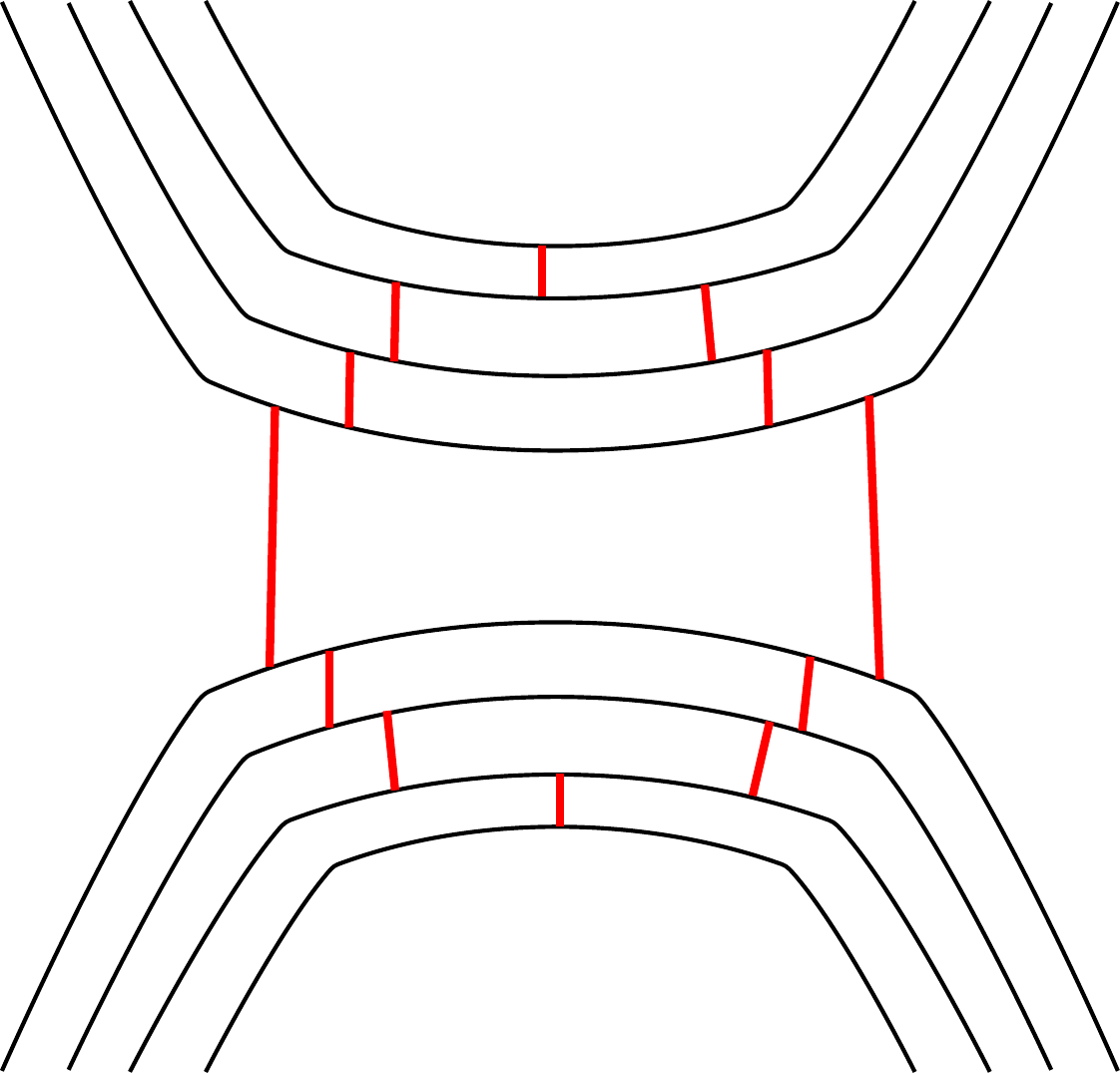
    \caption{\label{fig:firstset} The red edges correspond to the crossings in $C_{\sigma}'$. }
\end{figure} 

\begin{proof} 
For $i=n$, we know that a horizontal line $H$ in $D$ has to intersect with $x^n_{\sigma}$ in at least $2k$ points. Therefore, the number of crossings in $C^u_n = C^{\ell}_n$ on which $\sigma$ chooses the $-$-resolution is at least $k$, and we may take the two furtherest crossings for the set $u'_n = \ell'_n$. (There is nothing to prove if $k=1$, because then we can just take one crossing for $u'_n = \ell'_n$ and we have the set $C'_{\sigma}$, which will also satisfy the conditions for $C_{\sigma}$.) For $i = n-1, n-2, \ldots n-k+1$, if there are not two crossings in $C^u_i$ for which the ends of the corresponding segments on $U_{i+1}$ lie between the segments from the crossings of $u'_{i+1}$, then we can isotope $H$ such that it has fewer than $2k$ intersections with the skein $x^n_{\sigma}$, see Figure \ref{fig:firstsetH} below.

\begin{figure}[H]
\centering
    \def\svgwidth{.7\columnwidth}
    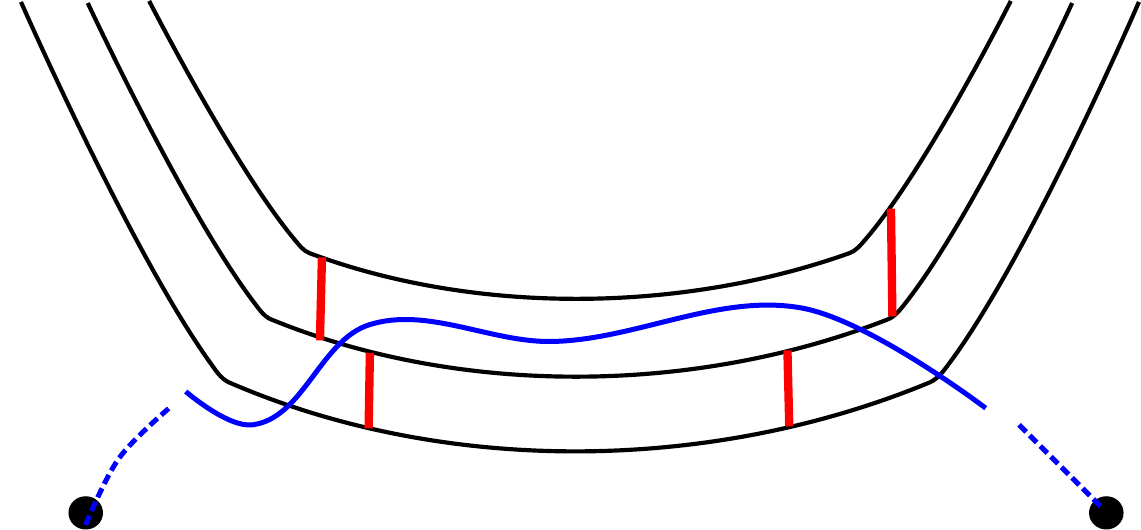
    \caption{\label{fig:firstsetH} The red segments correspond to the crossings in $C_{\sigma}'$. If all the red edges in $C_i^u$ lie outside of the two red segments in $C^u_{i+1}$, then we can draw the blue arc as shown to have only two intersections with $x^n_{\sigma}$ while entering and exiting the region between $U_i$ and $U_{i+1}$ containing $C_i^u$. }
\end{figure} 

We argue by assuming that $u'_{i+1}$ is already inductively constructed, and we would like to pick a set of crossings in $C^n_{i}$ to construct $u'_i$. Assuming that there are no crossings in $C^u_i$ on which $\sigma$ chooses the $-$-resolution, and whose corresponding segments lie between those of the crossings in $u'_{i+1}$,  Figure \ref{fig:firstsetH} shows an isotopy that will result in fewer than $2k$ intersections between $H$ and  $x^n_{\sigma}$. For $i\geq n-k + 2$, there has to be at least 4 intersections of $H$ with $x^n_{\sigma}$ in the region between $U_{i}$ and $U_{i+1}$, since $H$ will have at most $2(n-i)$ intersections before entering/exiting.  This gives at least two crossings in $C^u_{i}$ on which $\sigma$ chooses the $-$-resolution whose corresponding segments are between those of $u'_{i+1}$ . If $i = n-k+1$ then we require at least two intersections, hence the single crossing that we can pick for $u_{n-k+1}'$. The argument for constructing $\ell_i'$ is completely symmetric. 

\end{proof} 

To complete the rest of the proof of $(a)$, we add crossings to $C'_{\sigma}$ inductively to get a set $C_{\sigma}$ which satisfies the remaining requirements. Let $|u_i'|$ and $|\ell_i'|$ denote the number of crossings in $u'_i$ and $\ell'_i$, respectively. Let $u'_i$ (resp. $\ell'_i$) be such that $|u'_i|\geq 2$ (resp. $|\ell'_i|\geq 2$). Dividing the disc in half with a vertical line $0 \times h$,  we label the crossings in $u_i'$ (resp. $\ell_i'$)to the left of the vertical line by $-$ and the crossings to the right of the vertical line by $+$, so $-x$ denotes a left crossing and $-u_i'$ (resp. $-\ell_i'$) denotes the entire set of crossings in $u_i'$ (resp. $\ell_i'$) to the left of the vertical line. 

\subsection*{Algorithm for constructing $C_{\sigma}$}
We start with the constructed set $C'_{\sigma}$ that satisfies the three conditions marked by *.
\begin{enumerate}
\item Consider the difference $k - |u'_n|$, if this difference is 0 then terminate. $C'_{\sigma}$ is already a set of edges which satisfies the assumptions of part (a) of the lemma. Set $C_{\sigma} = C'_{\sigma}$. 

\item Otherwise, for $i = n$, $n-1 \ldots$, $n-k+1$, set $C = k-n+i-|u'_i|$. We assume inductively that $C'_{\sigma}$ satisfies the following for $n-k+1 \leq i \leq n$: 
\begin{enumerate}[(i)]
\item An edge in $-u'_i$ with two edges above and below to the left of it, is the leftmost possible for all edges to the right of the two edges. Similarly, An edge in $+u'_i$ with two edges above and below to the right of it, is the rightmost possible for all edges to the left of the two edges. We assume the same with $u'_i$ replaced by $\ell'_i$.
\item Let $\pm p$ be the midpoint of an edge whose corresponding crossing, say $\pm x$, is in $\pm u'_i$ , then there are two arcs $H^{\pm}$, where $H^-$ starts at (-1, 0) and ends at $-p$, and $H^+$ starts at $p$ and ends at $(1, 0)$, such that the numbers of intersections of $H^+$ and $H^-$ with $x^n_{\sigma}$ are given by (not counting the intersections with $-p$ and $p$):
\begin{align} \label{eqn:intersect1} 
&\text{If $x\in -u'_i$, }|H^- \cap x_{\sigma}| = 2 \left( \# \text{ of crossings to the left of $-x$ in $u'_i$ } \right) + (n-i).   \\ 
 \label{eqn:intersect2}
&\text{If $x\in +u'_i$, }|H^+ \cap x_{\sigma}| = 2 \left( \# \text{ of crossings to the right of $+x$ in $u'_i$} \right) + (n-i). 
\end{align}   
\end{enumerate} 

That these assumptions are valid through every iteration of $i$ follows from Lemma \ref{lem:conditions}. Before we prove the lemma, we proceed with the algorithm with those assumptions. 

\begin{itemize}
\item[If $C=1$:] Let $-x$ be the rightmost edge in $-u_i'$. There is an edge $x'$ in $C'_{\sigma}$ above in $u'_{i-1}$ and another edge $x''$ below it in $u'_{i+1}$, both to the right of $-x$. There are only a few possibilities for the edges in $C^u_{i}$ on which $\sigma$ chooses the $-$-resolution (shown in red) to the right of $-x$, whose ends on $U_i$ and $U_{i+1}$ are not to the right of both $x'$ and $x''$, respectively. They are shown as slanted dashed edges in Figure \ref{fig:skeinHcomb1}. 

\begin{figure}[H]
\def \svgwidth{.9\columnwidth}
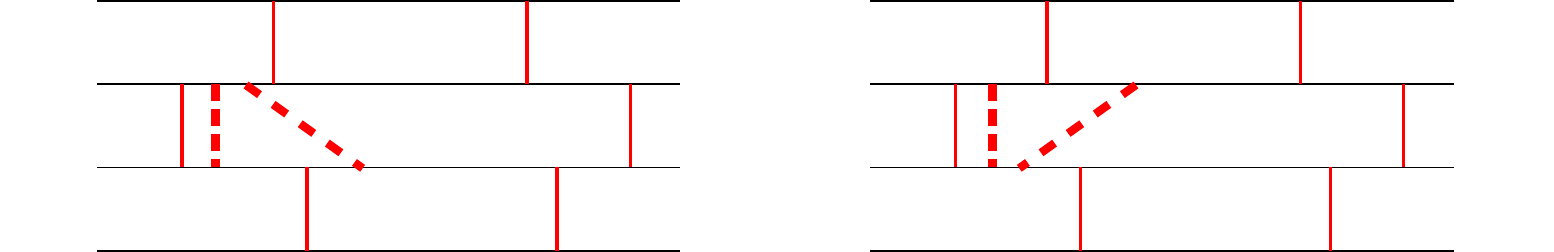
\caption{\label{fig:skeinHcomb1} The thickened dashed edges indicate possible multiple edges.}
\end{figure}

Let $-p$ be the midpoint of $-x$ and $-p'$ be a point between $U^{i}$ and $U^{i+1}$ immediately to the right of both $x'$ and $x''$ and to the left of any crossings in $C^u_{i}$ on which $\sigma$ chooses the $-$-resolution to the right of both $x'$ and $x''$. Either we can draw an arc from the left of $-p$ to $-p'$ that only has 2 intersections with $x^n_{\sigma}$, see Figure \ref{fig:skeinHcomb2}, or, there are two choices for the existence of a red edge $y$ in either $C^u_{i+1}$ or $C^u_{i-1}$. This is shown in Figure \ref{fig:skeinHcomb3}.  

\begin{figure}[H]
\def \svgwidth{.9\columnwidth}
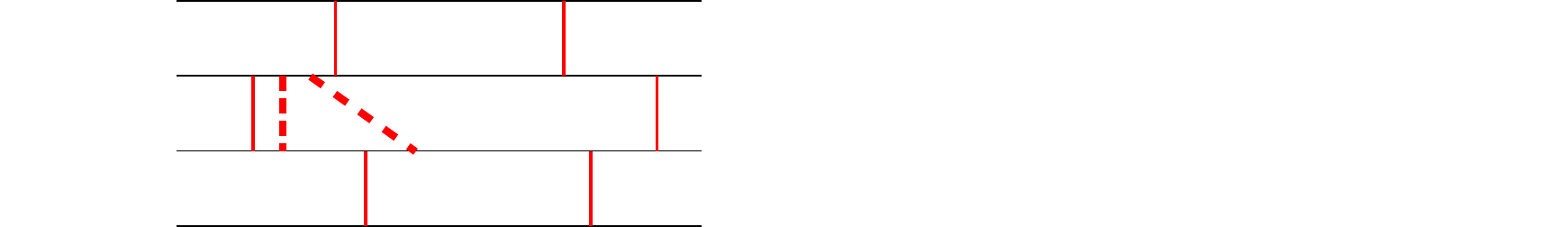
\caption{\label{fig:skeinHcomb2}The point $-p$ is marked with a red dot and the point $-p'$ is marked with a black dot.}
\end{figure}

\begin{figure}[H]
\def \svgwidth{.9\columnwidth}
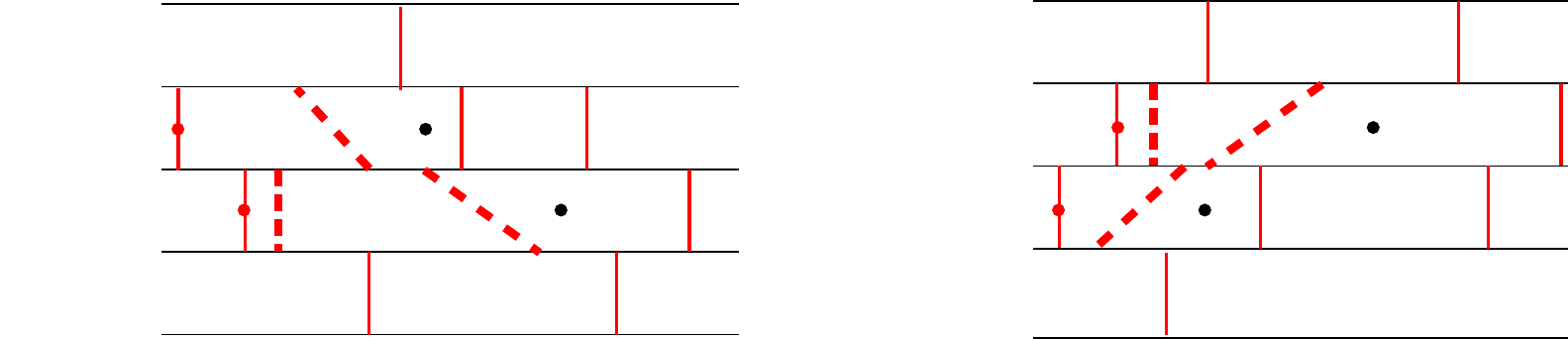
\caption{\label{fig:skeinHcomb3} The edges $y$ for both of these cases prevent the arcs as in Figure \ref{fig:skeinHcomb2} from being drawn without two more intersections with $x^n_{\sigma}$.}
\end{figure}

Without loss of generality we will just assume that it is in $C^u_{i-1}$ where we have the edge $y$, and we consider the rightmost such edge.
Now we consider $-x_1$ which is the nearest edge in $u'_{i-1}$ to the left of $y$. Let $-p_1$ be the midpoint of $-x_1$ and $-p'_1$ be the point between $y$ and the nearest edge $-z_1=x'$ in $u'_{i-1}$ to the right of $y$. Again, we see if we can draw an arc from the left of $p_1$ to $p'_1$ that only has 2 intersections with $U_{i-1}$. If not, there exists another red edge $y_1$ which obstructs this. We repeat the same steps with $y_1$ to obtain a necessarily finite sequence of edges $y, y_1, \ldots, y_m$. For $y_m$ we draw an arc from the left of $-p_m$ to $-p'_m$ that has only 2 intersections with $x^n_{\sigma}$. Then, we connect $-p'_j$ with $-p'_{j-1}$ for each $j$ with an arc that is parallel to the rest of $y_j$'s and to the left of the $-z_j$'s, see Figure \ref{fig:skeinHcomb4} below.

\begin{figure}[H]
\def \svgwidth{\columnwidth}
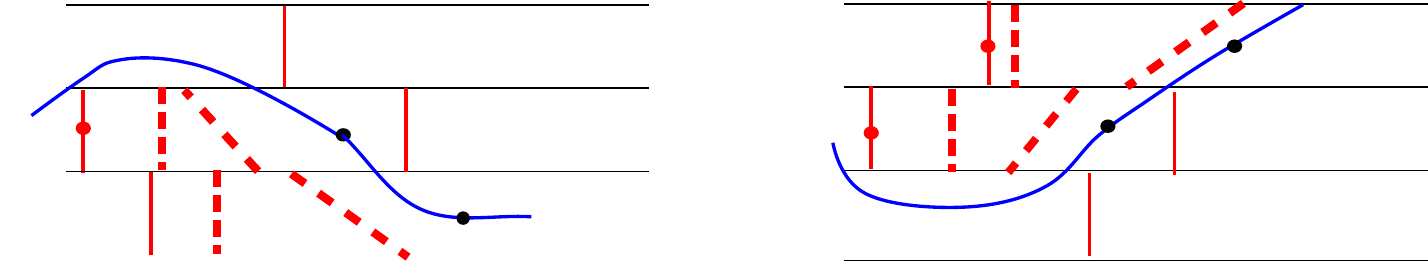
\caption{\label{fig:skeinHcomb4} We extend to $-p'_{m-1}$ the arc going from $-p_m$ to $-p'_m$ by another arc parallel to $y_{m-1}$. }
\end{figure}

There is only a single intersection of the arc between $-p'_j$ and $-p'_{j-1}$ with $x^n_{\sigma}$ because of assumption (i). Putting all these arcs together, we get an arc from $-p_m'$ to $-p'$ that has $m$ intersections with $x^n_{\sigma}$. Now 
\begin{equation} \label{eq.compareppm}
\# \text{ of edges in $u'_{i-j}$ to the left of $-p_j$} = (\# \text{ of edges in $u'_i$ to the left of $-p$}) - m. 
\end{equation} 

Using assumption (ii) on $-p_m$, we get an arc $H^-$ from $(-1, 0)$ to $-p'$ with the number of intersections with $x^n_{\sigma}$ as follows. The arc $H^-=H^-_1\cup H^-_2$ is the union of two arcs: The arc $H^-_1$ from $(-1, 0)$ just to the left of $-p_m$, and the arc $H^-_2$ from $-p_m$ to $-p'$. Their intersections with $x^n_{\sigma}$ are respectively given by using \eqref{eqn:intersect1}. 
\begin{align*}
|H^-_1\cap x^n_{\sigma}| &=  2 \left( \# \text{ of crossings to the left of $-x_m$ in $u'_{i-m}$ } \right) + n-(i-m), \text{ and } \\
|H^-_2 \cap x^n_{\sigma}| &= m, \text{  based on the preceding discussion. 
} 
\end{align*}
Taking the intersections of $H^-_1$ and $H^-_2$ together,  the number of intersections between $H^-$ and $x^n_{\sigma}$ is
\begin{equation*} |H^-\cap x^n_{\sigma}| =  2\left( \# \text{ of crossings to the left of $-x'$ in $u'_i$ } \right) + n-i. \end{equation*}

Similarly, with the same argument replacing $-$ with $+$, ``right" with ``left," and ``left" with ``right", we can get another arc $H^+$ from $(1, 0)$ to $+p'$ that has the number of intersections with $x^n_{\sigma}$ given by
\[ |H^+ \cap x^n_{\sigma}| =  2\left( \# \text{ of crossings to the right of $+p'$ in $u'_i$ } \right) + n-i. \]  
 Now consider the straight line segment $L$ from $-p'$ to $+p'$. If $\sigma$ does not choose the $-$-resolutoin on any crossing in $C^u_{i}$ between $-p'$ and $+p'$, then we get an arc $H'' = H^-\cup L\cup H^+$ that has $\leq 2(k-1) < 2k$ intersections with $x^n_{\sigma}$, which is a contradiction. We add this crossing to $u'_i$ and move on to the next $i$ in the iteration. 

\item[If $C>1$:] This is similar to the case when $C=1$. The arguments are the same except that at the last stage we can add a furtherest pair of edges, each marked with $-$ and $+$ for left and right, to $u'_i$. After this we move onto the next $i$ in the iteration. 
\end{itemize}
\item We repeat from Step (1) until $k-|u'_n| = 0$. 
\end{enumerate}

Running the same algorithm for $\ell'_n$ with the obvious adjustment by symmetry gives us $C_{\sigma}$. 

\begin{lem} \label{lem:conditions} Every iteration of $C'_{\sigma}$ through the algorithm satisfies conditions $(i)$ and $(ii)$. 
\end{lem} 

\begin{proof}
For the first iteration of $C'_{\sigma}$, condition (i) is vacuously true. For a crossing in $-u'_{i}$, the arc as shown satisfies condition (ii). The same arc by reflection also works for a crossing in $+u'_i$. 

\begin{figure}[H]
\def \svgwidth{.3\columnwidth}
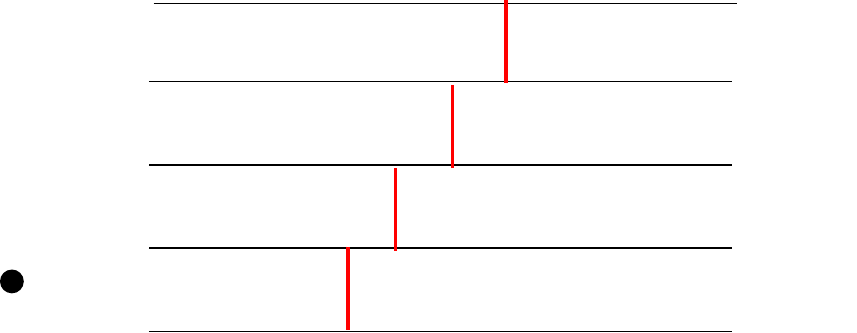
\caption{\label{fig:basearc} There are no other intersections of the blue arc with $x^n_{\sigma}$ other than those shown because the initial construction of $C'_{\sigma}$ requires that each pair in $u'$ are the furtherest possible, one of the conditions marked by $*$.}
\end{figure}

For each subsequent iteration of $C'_{\sigma}$, the edges added are specifically chosen to satisfy both $(i)$ and $(ii)$. 
\end{proof}

\paragraph{\textbf{Proof of (b)}}
This is immediate by considering the sequence of states $\sigma_1=\sigma_+, \sigma_2, \ldots, \sigma_f', \ldots , \sigma_f =\sigma$ where the first part of the sequence from $\sigma_+$ to $\sigma_f'$ comes from changing the resolution from $+$ to $-$ on the set of $k^2$ crossings with structure as described in part (a), and counting the number of circles in $\sigma_f'$.

\begin{figure}[H]
\def \svgwidth{.8\columnwidth}
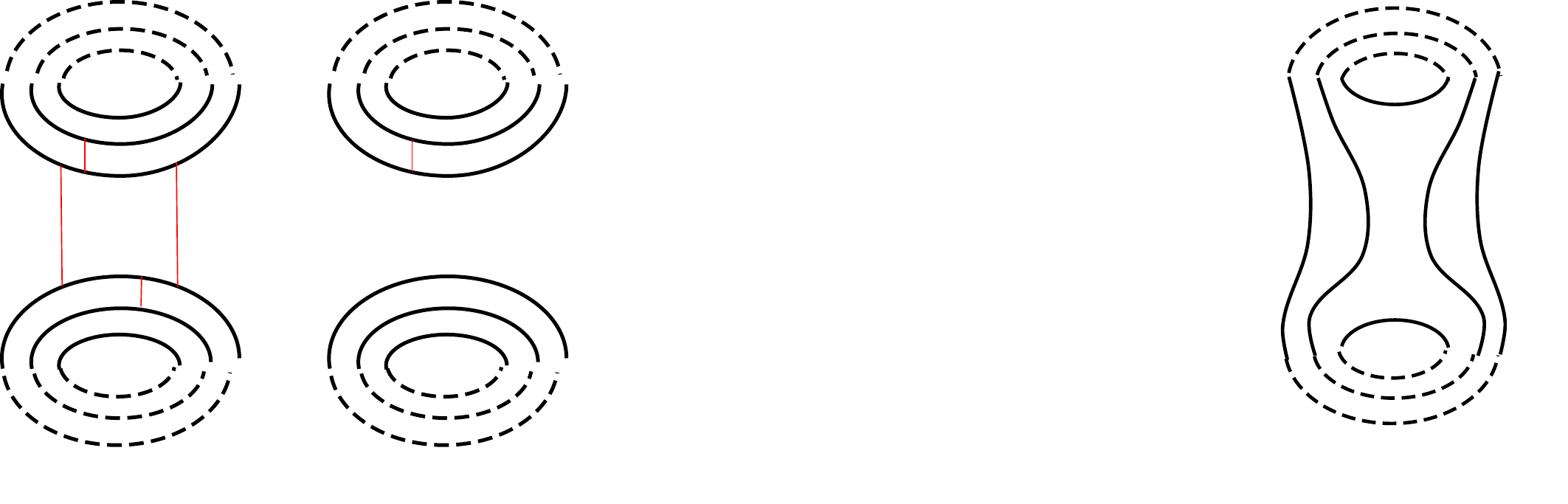
\caption{In this example, $n=3$ and we show $\sigma'_f$ as well as the skeins of a sequence of states from $\sigma_+$ to $\sigma_f'$. The number of through strands here is 4, thus $k =2$ and we change the resolutions on $2^2=4$ crossings, resulting in a sequence of length 4. }
\end{figure} 
\qed

\begin{defn}
 Let $D$ be a link diagram and $G$ be a 2-connected, weighted planar graph such that $D = \partial (F_G)$. For a positively-weighted edge $\epsilon$ of $G$ corresponding to a maximal positive twist region $T$ in $D = \partial(F_G)$, orient the $n$-cabled twist region $T^n$  as an element in $\Sk(\mathcal{D}^2, 2n)$, so that all the crossings are as in Figure \ref{fig:flowconfig}.  We say that a Kauffman state $\sigma$ in the state sum of \eqref{eq:statesum} on $D^n_{\jwproj}$ \emph{flows through} $\epsilon$ with $2k$ strands if the skein in $\Sk(\mathcal{D}^2, 2n)$ resulting from applying $\sigma$ to the $n$-cabled twist region $T^n$ has $2k$ arcs connecting $2k$ points on the top and the bottom. 
\end{defn}

 An immediate consequence of Lemma \ref{lem:count} is the following.

\begin{lem} \label{lem:wcount} Let $D$ be a link diagram and $G$ be a 2-connected, weighted planar graph such that $D = \partial (F_G)$. Let $\epsilon = (v, v')$ be an edge in $G$ corresponding to a maximal positive twist region with $\omega \geq 2$ crossings, and $\sigma$ is a Kauffman state from the state sum of \eqref{eq:statesum} on $D^n_{\jwproj}$ that flows through $\epsilon$ with $2k$ strands. Let $\sigma_+$ be the Kauffman state that chooses the $+$-resolution all the crossings in $T^n$ but agrees with $\sigma$ everywhere else. Then the sequence of states from $\sigma_+$ to $\sigma_f$ contains a subsequence $\sigma_+, \ldots, \sigma_f'$ of length $\omega k^2$ and $\overline{\Sk_{\sigma_f'}}$ has $(\omega-2)k$ fewer circles than $\overline{\Sk_{\sigma_+}}$.  
\end{lem} 

\begin{proof}
If $\sigma$ flows through the edge $\epsilon$ with $2k$ strands than it flows through every crossing in $T$ represented by $\epsilon$ with at least $2k$ strands. We apply Lemma \ref{lem:count}(a) and add up the number of crossings on which $\sigma$ chooses the $-$-resolution over each $x^n$ for a crossing $x \in T$. This gives that $\sigma$ chooses the $-$-resolution on at least $\omega k^2$ crossings. In a twist region with $\omega$ crossings we have that in the all-$+$ state on $T^n$ there are $(\omega-1)$ sets of $n$ disjoint circles. Thus we can apply part (b) of Lemma \ref{lem:count} $\omega - 2$ times. 
\end{proof} 

\subsection{Proof of Theorem \ref{thm:bracketdegree}} \label{subsec:complete}
Now we complete the proof of Theorem \ref{thm:bracketdegree}. Recall that 
from Section \ref{subsec:simplifyss} we have 
\[\langle D^n_{\jwproj} \rangle = \sum_{\sigma, \ a \ : \ a, \ n, \ n, \text{ admissible }, \ \frac{a}{2} \leq c}  \frac{\triangle_a}{\theta(a, n, n)}(-1)^{rn-r\frac{a}{2}} A^{d(a, r) + \sgn(\sigma)} \langle J^a_{\sigma} \sqcup  \text{ disjoint circles}\rangle, \] 
and we would like to show that 
\[\deg(\sigma, a) < \deg(\sigma_+, 0), \] where $\deg(\sigma_+, 0) = H_n(D) + 2r(n^2+n)$, and $\deg(\sigma, a)$ is the maximum degree of a term indexed by $\sigma, a$ in the state sum of $\langle D^n_{\jwproj}\rangle$. Recall also that $2c$ is the number of split strands of $\sigma$ and that by \eqref{eqn:aeqc}, we need only to consider states $\sigma$ with parameter $a$ such that $\frac{a}{2}=c$.

If $\sigma$ is a state with $a=c=0$ that is not the all-$+$ state, then it must choose the $-$-resolution at a crossing in a positive maximal twist region, which will merge at least one pair of circles compared to the all-$+$ state. Hence, a sequence $s$ from $\sigma_+$ to $\sigma$ for $a=0$ contains at least one pair of states that merges a pair of circles. This implies that
\[ \deg(\sigma, 0) \leq \deg(\sigma_+, 0) - 4, \] so 
\[ \deg(\sigma, 0) < \deg(\sigma_+, 0). \]

If $\sigma$ is a state with $c > 0$, then the skein $J^a_{\sigma}$ can be decomposed along a square $(\mathcal{D}^2, 2n)$ with $2n$ points marked above and below, containing the Jones-Wenzl idempotents as shown in the following figure, so that we get two skeins $\Sk_1$ and $\Sk_2$ in $\Sk(\mathcal{D}^2, 2n)$. 

\begin{figure}[ht]
\centering
    \def\svgwidth{.5\columnwidth}
    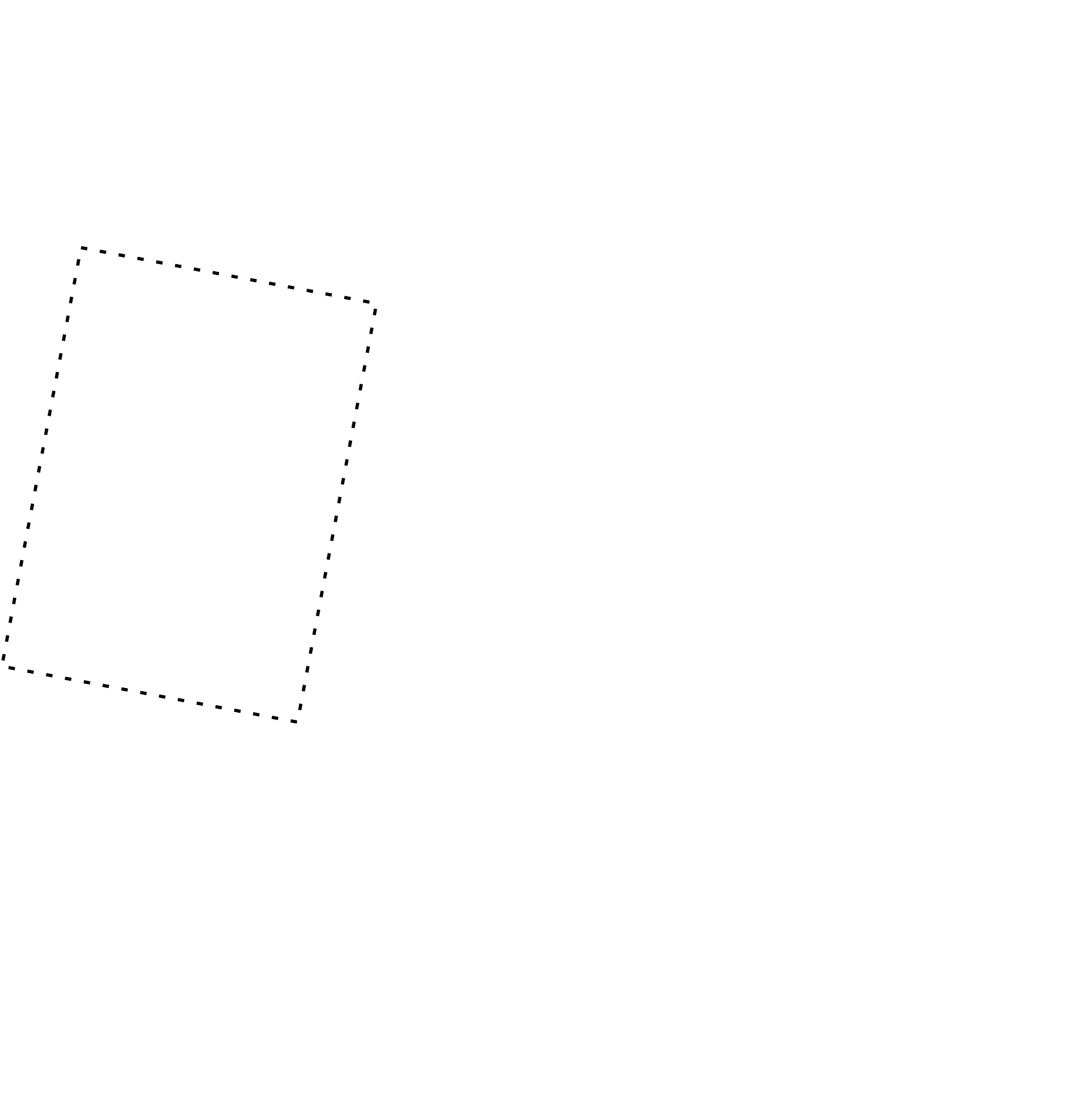
    \caption{\label{fig:squaredecomp} The link diagram is obtained by composing two skeins in $\Sk(\mathcal{D}^2, 2n)$. The skein $\Sk_1$ is enclosed by the square and the skein $\Sk_2$ is outside of it.}
\end{figure} 

Now in $\Sk_2$ with $\sigma$ applied we have at least $2c$ strands connecting the $2c$ points at the top to the $2c$ points at the bottom on the boundary of the disk $\mathcal{D}^2$.

Let $D = \partial (F_G)$ be a near-alternating link diagram and $G\setminus e$ be the graph obtained from $G$ by deleting the single edge $e=(v, v')$ of negative weight $r$. Let $t$ be the total number of paths $W_1, \ldots, W_t$ from $v$ to $v'$ in $G\setminus e$. Let $2k_i$ be the number of strands with which the state $\sigma$ flows through a path $W_i$ for $1\leq i \leq t$, see Figure \ref{f.flowpath} for an example. 

\begin{figure}[ht]
\def \svgwidth{.5\columnwidth}
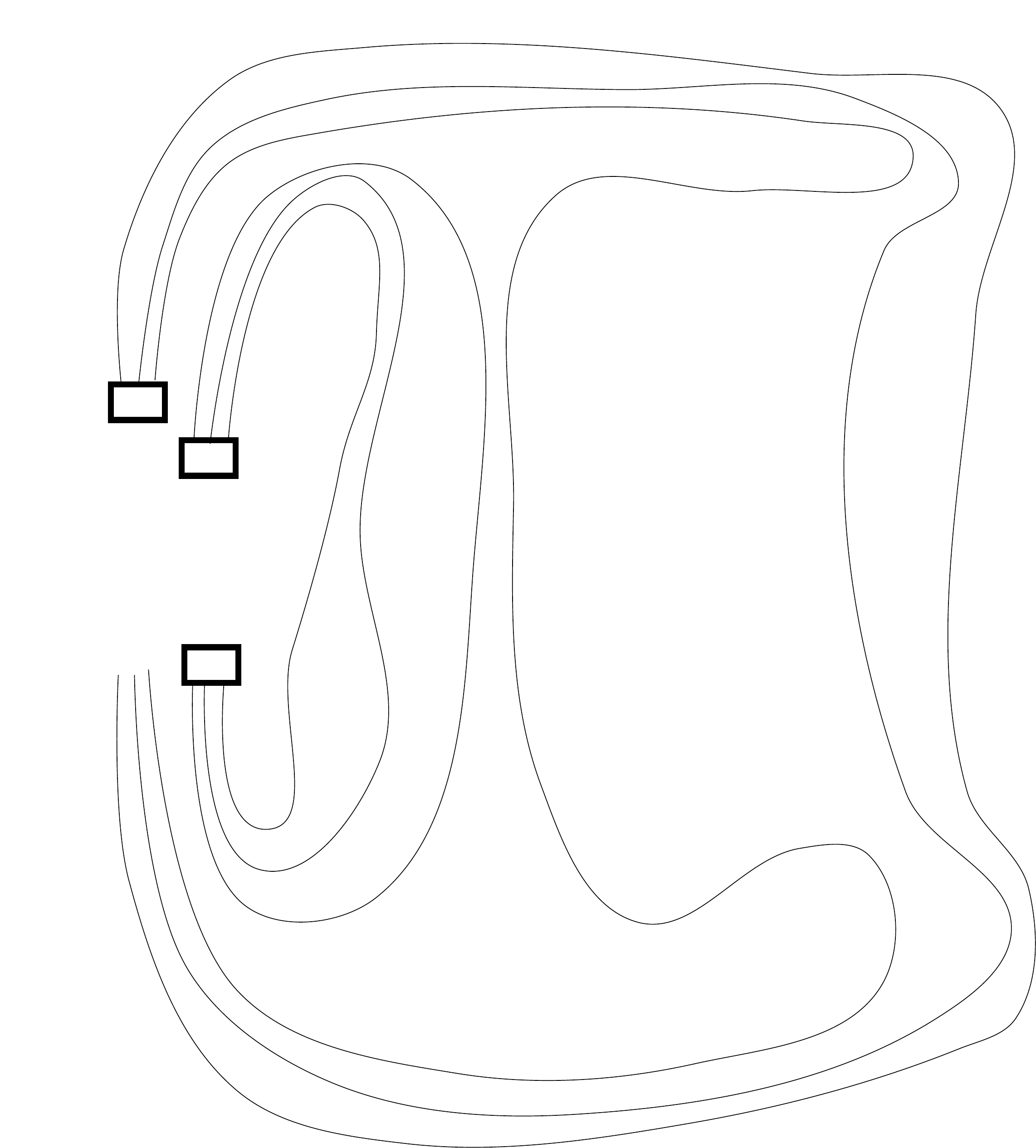
\caption{\label{f.flowpath} There are four paths $W_1, W_2, W_3$, and $W_4$ from $v$ to $v'$ in $G\setminus e$ in this example, where $W_1$ and $W_2$ share an edge. A skein $\Sk^a_{\sigma}$ is shown with its connected component $J^a_{\sigma}$ and disjoint circles with 6 split strands. The state $\sigma$ flows through path $W_1$ with 2 strands, $W_2$ with 2 strands, $W_3$ with 0 strands, and $W_4$ with 2 strands.}
\end{figure} 

We have 
\[\sum_{i=1}^{t} 2k_i \geq 2c. \] 
Without loss of generality we may assume 
 \[\sum_{i=1}^{t} 2k_i = 2c, \]
 since if  $\sum_{i=1}^{t} 2k_i > 2c$ for a state $\sigma$, then $\deg(\sigma, 2c) < \deg(\sigma', 2c)$ for another state $\sigma'$ for which $\sum_{i=1}^t 2k_i' = 2c$. 

We can construct a sequence $s$ from $\sigma_+$ to $\sigma$ by changing the resolution from $+$-to $-$- on the set of crossings $x^n$ for each crossing $x$ in a maximal positive twist region, beginning with the crossings in the twist regions in $W_1$, then $W_2$, and so on until $W_t$. For each walk $W_i$ with $2k_i$ strands flowing through we apply Lemma \ref{lem:wcount} to estimate $\deg(\sigma, 2c)$ relative to $\deg(\sigma_+, 0)$. 

Let 
\[ \deg(\overline{\sigma_+}, 2c) := \text{deg} \left( \frac{\triangle_{2c}}{\theta(n, n, 2c)}(-1)^{rn-rc} A^{d(2c, r) + \sgn(\sigma_+)} \langle \overline{J^{2c}_{\sigma_+}} \sqcup \text{disjoint circles} \rangle \right).\]
For each edge $\epsilon$ of a path $W_i$, $\sigma$ flows through it with at least $2k_i$ strands. Thus, we can find a subsequence $\sigma_1, \ldots, \sigma_f$ in a sequence from $\sigma_+$ to $\sigma$ corresponding to changing the resolutions on the crossings in the $n$-cabled twist region $T^n_\epsilon$ for this edge, that is of length $\omega_\epsilon k_i^2$ and with the skein $\overline{\Sk_{\sigma_f}}$ having $(\omega_\epsilon - 2)k_i$ fewer circles than $\sigma_1$. Recall that $\omega_\epsilon$ is the number of crossings in the twist region corresponding to $\epsilon$. This implies a decrease of degree by at least
\[ -2\omega_\epsilon k_i^2 - 2(\omega_\epsilon-2)k_i.  \]      We sum over all the edges in $W_i$ to get the total amount of decrease in degree for this path. Moving on to the next path for the sequence, it may happen that multiple paths $W_{i_1}, \ldots, W_{i_p}$ share the same edge, but then the decrease in degree from this single edge would be
\[ -2\omega_\epsilon \left(\sum_{j=1}^p k_{i_j} \right)^2 - 2(\omega_\epsilon-2)\sum_{j=1}^p k_{i_j} \leq -2\omega_\epsilon \sum_{j=1}^p k_{i_j}^2 - 2(\omega_\epsilon-2)\sum_{j=1}^p k_{i_j}. \]
 
Thus without loss of generality, we may assume that none of the paths share edges and sum the decrease in degree over the edges of each path to get
\begin{align}
\deg(\sigma, 2c) &\leq  \deg(\overline{\sigma_+}, 2c) - \left(  \sum_{i=1}^{t} (\omega - 2)(2k_i^2+2k_i)+4k_i^2 \right) \\
\intertext{where $\omega = \min_{1\leq i\leq t} \left\{ \ell(W_i) \right\}$. Recall $\ell(W_i)$ is the length of a path defined by \eqref{eq.l}. We get } 
\deg(\sigma, 2c) &\leq \deg(\sigma_+, 0) - \left(  \sum_{i=1}^{t} (\omega - 2)(2k_i^2+2k_i)+4k_i^2 \right)-2c^2r - 2cr.
\end{align} 
Since $\sum_{i=1}^{t} 2k_i = 2c$, the $k_i$'s form a partition of $c$. The following lemma shows that we may replace it by a minimal partition. 

\begin{defn} Let $P = \{n_1, \ldots,  n_t\}$ be a nonnegative integer partition of $n$ where the $n_i$'s may be zero, so $n = n_1 + \cdots + n_t$. We say that a partition of $n$ into $t$ parts is a \emph{minimal partition}, denoted by $P_m$, if it has the minimal $m = \max_{1\leq i \leq t} n_i$ out of all partitions of $n$ into $t$ parts. 
\end{defn} 

\begin{lem} \label{lem:part} Fix $n$ and $t$. A minimal partition $P_m=\{m_1, \ldots, m_t \}$ of $n$ into $t$ parts is unique up to rearrangement of indices. If $P=\{n_1, \ldots n_t \}$ is another partition of $n$ into $t$ parts,  then
\[\sum_{i=1}^t m_i^2 \leq \sum_{i=1}^t n_i^2.  \]  
\end{lem} 
\begin{proof}
A minimal partition $P_m$ may be constructed as follows. If $n\leq t$ then the partition has $m_1=m_2=\cdots =m_n = 1$ and $m_{n+1} = m_{n+2}=\cdots m_{t} = 0$. If $n>t$, let $j = n \pmod{t}$. The partition $P_m$ has $m_1=m_2=\cdots = m_j = \floor{n/t}+1$ and $m_{j+1} = m_{j+2} = \cdots = m_{t} = \floor{n/t}.$ The partition is minimal, since we may obtain any other partition of $n$ into $t$ parts from $P_m$ by subtracting 1's from a non-zero summand and adding 1 to any other. Similarly, it is unique up to rearrangement.

For the statement that $\sum_{i=1}^t m_i^2 \leq \sum_{i=1}^t n_i^2$, there is nothing to prove if $P = P_m$. Let $m' = \max_{1\leq i \leq t} n_i$ and $m = \max_{1\leq i \leq t} m_i$. Since $P_m$ is minimal and unique up to rearrangement we can assume that $m'> m$, $m' = n_1$ in $P$, and $m = m_1$ in $P_m$. Suppose $m' = m + k$ for some integer $k > 0$. This means that we may write
\[ P = \{m_1+k, m_2-k_2, \ldots, m_t-k_t\},   \] 
where $k_2, \ldots, k_t \geq 0$ and $k = k_2 + \cdots + k_t$. Now we have 
\begin{align*}
 \sum_{i=1}^t n_i^2 &= (m_1+k)^2 + (m_2-k_2)^2+\cdots + (m_k-k_t)^2 = \left( \sum_{i=1}^t m_i^2 \right) + 2m_1k + k^2 + \sum_{i=2}^t (-2m_ik_i + k_i^2). \\ 
\intertext{Thus} 
& 2m_1k+k^2 + \sum_{i=2}^t (-2m_ik_i + k_i^2) \geq 2m_1k + k^2-2m_1k + \sum_{i=2}^t (k_i)^2 \geq 0.
 \end{align*}
 This concludes the proof of the lemma.
\end{proof}
Finally, replacing $\{k_i\}$ by a minimal partition $P_m = \{m_1, \ldots, m_t\}$ using Lemma \ref{lem:part}, we have 
\begin{equation}
\deg(\sigma, 2c) \leq \deg(\sigma_+, 0) - \left(\sum_{i=1}^{t} (\omega - 2)(2m_i^2+2m_i) + 4m_i^2\right)- \left( 2c^2r+2cr \right). 
\end{equation} 
If $|r| < \frac{\omega}{t}$ with $|r| \geq 2$ and $t> 2$, then the difference
\begin{equation}- \left(\sum_{i=1}^{t} (\omega - 2)(2m_i^2+2m_i)+4m_i^2\right)- \left( 2c^2r+2cr \right)  \label{eqn:inequality} \end{equation} is negative, so \[ \deg(\sigma, 2c) < \deg(\sigma_+, 0) \] for every other Kauffman state $\sigma$ with $2c> 0$ split strands. Since we also know this inequality for $\sigma$ with $2c=0$ split strands, this shows that $\deg \langle　D^n_{\jwproj} \rangle = \deg(\sigma_+, 0)$ and finishes the proof of the theorem.

\section{Boundary slope and Euler characteristic} \label{sec:jsurface}
In this section, we prove Theorem \ref{thm:jsurface} and verify that there exists an essential spanning surface which realizes the Jones slope $js_K = \{-2c_-(D) -2r\}$ and the quantity $jx_K=\{c(D)-|s_+(D)|+r \}$ of a near-alternating link $K$ determined in Section \ref{sec:jslope}. 
Let $D$ be a near-alternating diagram with surface $F_G$, such that $D = \partial (F_G)$ for a 2-connected, weighted planar graph $G$ as in Definition \ref{defn:near-alternating}, is called a \emph{pretzel surface}. It is shown to be essential under certain conditions on the graph $G$ in \cite{OR12}. 

\begin{thm}{{\cite[Theorem 2.15]{OR12}}} \label{thm:pretzele}
Let $G$ be a 2-connected planar graph in $S^2$ with edges $e_1, \ldots, e_n$ having weights $\omega_1, \ldots, \omega_n \in \mathbb{Z}$. 
\begin{enumerate}
\item If $|\omega_i| \geq 3$ for all $i$, then the surface $F_G$ is essential. 
\item If $\omega_1 \leq -2$ and $\omega_i \geq 2$ for $i = 2, \ldots, n$, and the surface $F_G$ is not essential, then $G$ has an edge, say $e_2$, that is parallel to $e_1$ (i.e., $e_2$ is another edge on the same pair of vertices as $e_1$) such that $\omega_1=-2$ and $\omega_2 = 2$ or $3$.
\end{enumerate} 
\end{thm} 
\begin{rem}
Note that the original wording of the theorem in \cite{OR12} says ``algebraically incompressible and boundary incompressible" instead of ``essential." 
\end{rem}

The surface $F_G$ is clearly also a state surface from the state that chooses the $-$-resolution on all the crossings in the single negative twist region of $D$, and the $+$-resolution on all the rest of the crossings. A formula for the boundary slope of a state surface for a knot is given by the following lemma.  
\begin{lem}[{\cite{FKP13}}] \label{lem:stateslope} Let $D$ be a diagram of an oriented knot $K$, and let $\sigma$ be a Kauffman state of $D$. Then the state surface $S_{\sigma}(D)$ has as its boundary slope
\[ 2c_+^-(\sigma)-2c_-^+(\sigma),\] where  $c_+^-(\sigma)$ is the number of positive crossings where the $-$-resolution is chosen, and $c_-^+(\sigma)$ is the number of negative crossings where the $+$-resolution is chosen.
\end{lem} 

If $K$ is a near-alternating knot, we can apply Theorem \ref{thm:pretzele} to show that $F_G$ is an essential surface for $K$. If the maximal negative twist region of weight $r<0$ in a near-alternating diagram $D$ of $K$ has $r=-2$, the only way the surface $F_G$ is not essential via condition (2) of Theorem \ref{thm:pretzele}, is if $G$ has an edge $e_2$, that is parallel to $e_1$ corresponding to the negative twist region, such that $e_2$ has weight $2$ or $3$. However, the condition on the diagram being near-alternating implies that if an edge is parallel to $e_1$, then it must have more than 6 crossings, since it would give a path in $G\setminus e_1$ between $v$ and $v'$ where $e_1 = (v, v')$, and we require that the length of such a path be greater than $2t$, where $t$ is the total number of paths, while $t>2$. 

 We verify that $F_G$ is indeed a Jones surface realizing the Jones slope $js_K$ and $jx_K$ from Theorem \ref{thm:degree} by computing its boundary slope and Euler characteristic using Lemma \ref{lem:stateslope}. \\

\paragraph{\textbf{Boundary slope}}
A pretzel surface comes from the state $\sigma$ which chooses the $-$-resolution at each crossing in the negative twist region of the near-alternating diagram $D$, and this is the only difference between $\sigma$  and the all-$+$ state. Either all these crossings are positive, or they are all negative. We use Lemma \ref{lem:stateslope} to compare the boundary slope of this state to the boundary slope of the all-$+$ state which is $2c_+^-(\sigma_+)-2c_-^+(\sigma_+)=0-2c_-^+(\sigma_+)=-2c_-(D)$. Suppose the crossings in the twist region are positive, then we get $2c_+^-(\sigma)-2c_-^+(\sigma)=2(c_+^-(\sigma_+)-r)-2c_-^+(\sigma_+)=-2c_-(D)-2r$ as the boundary slope. If the crossings in the twist region are negative, we also get 
\begin{equation} \label{e.bslope}
2c_+^-(\sigma)-2c_-^+(\sigma)=2c_+^-(\sigma_+)-2(c_-^+(\sigma_+)+r)=-2c_-(D) - 2r \end{equation}  for the boundary slope, and we are done. \\

\paragraph{\textbf{Euler characteristic}}
It is clear that the Euler characteristic of the surface is 
\begin{equation} \label{e.echar} \chi(S_+(D)) - r = (|s_+(D)|-r)-c(D) = -(c(D)-|s_+(D)|+r). \end{equation}

\paragraph{\textbf{Proof of Theorem \ref{thm:jsurface}}}
We obtain the degree $d(n)$ of the $n$th colored Jones polynomial $J_K(v, n)$ by adjusting the degree of the Kauffman bracket from Theorem \ref{thm:bracketdegree} by the writhe. The essential surface of Theorem \ref{thm:pretzele} realizes $js_K$ and $jx_K$ by the preceding computation of boundary slope and Euler characteristic  of this surface.

\begin{lem} \label{lem:nabad} A near-alternating link is $-$-adequate.
\end{lem} 
\begin{proof}
Applying the $-$-resolution to all the crossings in a near-alternating diagram $D$, we see that the all-$-$ state graph of $D$ is given by the dual graph of $G\setminus e$ with $|r|-1$ vertices attached from the single negative twist region. Since $|r|\geq 2$, each of the segments resulting from applying the $-$-resolution to the crossings in the negative twist region connects a pair of distinct vertices in $s_-(D)$, so if $D$ is not $-$-adequate, then $D^e = \partial(F_{G\setminus e})$ is not $-$-adequate. Note that $D^e$ is an alternating diagram, and $D^e$ is reduced because the graph $G\setminus e$ is required to be 2-connected from the assumption on a near-alternating diagram. Otherwise, a vertex of the edge corresponding to the nugatory crossing would be a cut vertex, contradicting the assumption that $G\setminus e$ is 2-connected by condition \eqref{d.case2} of Definition \ref{defn:near-alternating}. Thus, $D^e$ is adequate by \cite[Proposition 5.3]{Lic97} since it is reduced and alternating. This implies that $D$ is $-$-adequate. 
 \end{proof} 
 
 \begin{cor} \label{cor:ssj}
Near-alternating knots satisfy the Strong Slope Conjecture. 
\end{cor} 

\begin{proof} 
By Theorem \ref{thm:degree}, which is directly implied by Theorem \ref{thm:bracketdegree} by substituting $A= v^{-1}$ and adding the writhe term, the minimum degree of the $n$th colored Jones polynomial is 
\begin{align*} 
 d(n) &=  -(n-1)^2c(D) -2(n-1)|s_+(D)| + \omega(D) (n^2-1) - 2r(n^2-n). 
\intertext{Expanding and gathering terms of $n$ with the same powers, we get} 
d(n) &= n^2(-c(D)-2r+\omega(D)) +n(2c(D)-2|s_+(D)|+2r)+ (-c(D)  + 2|s_+(D)|-\omega(D)). \\
\intertext{Since $c(D) = c_-(D) + c_+(D)$ and $\omega(D) = c_+(D) - c_-(D)$, we get}
d(n)&=  n^2(-2c_-(D)-2r) +n(2c(D)-2|s_+(D)|+2r)+ (-c(D)  + 2|s_+(D)|-\omega(D)).
\intertext{This means that }
js_K &= \{-2c_-(D) - 2r \}, \text{ and } jx_K = \{ c(D) - |s_+(D)|+r \}. 
\end{align*} 
These match the boundary slope of $F_G$ computed by \eqref{e.bslope} and the negative of the Euler characteristic of $F_G$ computed by \eqref{e.echar}, respectively.
As for $js^*_K$ and $jx^*_K$, Lemma \ref{lem:nabad} and \cite{FKP11} prove the existence of an essential surface realizing the statement of Theorem \ref{thm:jsurface} concerning the quadratic and linear growth rates of $d^*(n)$. This concludes the proof of Theorem \ref{thm:jsurface}.
 
\end{proof} 

\section{Near-alternating knots are not adequate} \label{sec:nadequate}

We show that a near-alternating knot does not admit an adequate diagram. The criterion for an adequate knot from the colored Jones polynomial is the following result due to Kalfagianni \cite{Kal16}. For large enough $n$ let 
\[s_1(n)n^2 + s_2(n)n + s_3(n) = d^*(n)-d(n)=(a_j^*-a_j)n^2+(b_j^*-b_j)n+(c_j^*-c_j).  \] 
\begin{thm}[{\cite[Theorem 4.2]{Kal16}}] \label{thm: adequatecrit} For a knot $K$ let $c(K)$ and $g_T(K)$ denote the crossing number and the Turaev genus of $K$, respectively. The knot $K$ is adequate if and only if for some $n > n_K$, we have 
\begin{equation}
s_1(n) = 2c(K), \text{ and } s_2(n) = 4-4g_T(K) -2c(K). 
\end{equation} 
Furthermore, every diagram of $K$ that realizes $c(K)$ is adequate and it also realizes $g_T(K)$. 
\end{thm} 

We will begin by proving the analogue of \cite[Lemma 8]{LT88} concerning the Kauffman polynomial for a near-alternating knot. Recall that for a link diagram $D$, the Kauffman two-variable polynomial $\Lambda_D(a, z)$ is defined uniquely by the following \cite[Theorem 15.5]{Lic97}
\begin{itemize}
\item $\Lambda_{\vcenter{\hbox{\includegraphics[scale=.10]{circ.png}}}}(a, z)=1$, where $\vcenter{\hbox{\includegraphics[scale=.10]{circ.png}}}$ is the standard diagram of the unknot.
\item $\Lambda_D(a, z)$ is unchanged by Reidemeister moves of Type II and III on the diagram D. 
\item $\Lambda_{\ \vcenter{\hbox{\includegraphics[scale=.1]{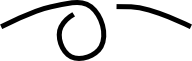}}} \ }(a, z) = a \Lambda{\ \vcenter{\hbox{\includegraphics[scale=.1]{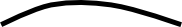}}}\ }(a, z)
$. The kink in the diagram $D$ is locally straightened out by a Reidemeister move of Type I at the expense of multiplying by $a$.
\item The Kauffman polynomials of diagrams locally differing in the following pictures are related as follows.
\begin{equation} \label{eqn:k2poly}
\Lambda_{\vcenter{\hbox{\includegraphics[scale=.15]{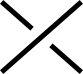}}}} (a, z)+ \Lambda_{\vcenter{\hbox{\includegraphics[scale=.15]{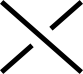}}}}(a, z) = z\left(\Lambda_{\vcenter{\hbox{\includegraphics[scale=.15]{crossing2.png}}}}(a, z) + \Lambda_{\vcenter{\hbox{\includegraphics[scale=.15]{crossing3.png}}}}(a, z)\right). 
\end{equation}
\end{itemize}
Diagrams which locally differ in one of the four pictures in \eqref{eqn:k2poly} are denoted by $D_+$, $D_-$, $D_0$, and $D_{\infty}$, respectively. 

We will need the following useful results by Thistlethwaite \cite{Thi88} with a minor change of notation.
\begin{thm}[{\cite[Theorem 4]{Thi88}}] \label{thm:kconnect} Let $D$ be a $c(D)$-crossing link diagram which is a connected sum of link diagrams $D_1, \ldots, D_k$. Let $\Lambda(a, z) = \sum_{r, s} u_{r, s} a^rz^s$ for $D$, and let $b_1, \ldots, b_k$ be the lengths of the longest bridges of $D_1, \ldots, D_k$, respectively. Then for each non-zero coefficient $u_{r, s}$, $|r|+s\leq c(D)$ and $s\leq c(D)-(b_1 + \cdots + b_k)$. 
\end{thm} 

\begin{thm}[{\cite[Theorem 5]{Thi88}}] \label{thm:coefflambda} Let $D$ be a connected, alternating diagram with $c(D)\geq 3$ crossings, and let $G$ be the graph associated with the black-and-white coloring of the regions of $D$ for which the crossings of $D$ all have positive sign. Let $\Lambda_D(a, z) = \sum p_s(a) z^s$, and let $\chi_G(x,y) = \sum v_{r, s} x^ry^s.$ (Here $\chi_G(x, y)$ is the Tutte polynomial of $G$.) Then 
\begin{align*}
p_{c(D)-1}(a)&= v_{1, 0}a^{-1} + v_{0, 1}a, \text{ and } \\ 
p_{c(D)-2}(a)&= v_{2, 0}a^{-2} + (v_{2, 0} + v_{0, 2}) + v_{0, 2} a^2.  
\end{align*}
\end{thm} 

In fact, Thistlethwaite remarks immediately following this theorem in \cite{Thi88} that the coefficient $p_{c(D)-1}(a)$ may be written as $\kappa(a+a^{-1})$ with $\kappa > 0$  if $D$ is a prime, alternating diagram with at least two crossings.

We prove a mild generalization of \cite[Lemma 8]{LT88} using the same argument which applies in the setting of near-alternating diagrams.
\begin{lem} \label{lem:gennearalt} Let $D$ be a near-alternating diagram of a link with a maximal negative twist region of weight $r < 0$ with $|r|\geq 2$. Then, the $z$-degree of $\Lambda_D(a, z)$ is $c(D)-2$.
\end{lem} 
\begin{proof} 
We induct on $|r| \geq 2$. Note that if $D$ is a near-alternating diagram with a negative twist region of weight $r<0$ and $|r|\geq 2$, then the same diagram with the maximal negative twist region replaced by a negative twist region of 2 crossings is still near-alternating. Thus it is valid to consider the base case with $|r|=2$ fixing the rest of the diagram $D$. For $|r|=2$, switching the top crossing in the twist region with weight $r$ results in an alternating diagram $D_-$ isotopic to one with $c(D)-2$ crossings by a Type II Reidemeister move. By Theorem \ref{thm:kconnect}, we see that the $z$-degree of $\Lambda_{D_-}(a, z)$ is strictly less than $c(D)-2$. One of the nullifications of this crossing results in a non-alternating diagram $D_0$, with $c(D)-1$ crossings and a bridge of length 3. Thus by Theorem \ref{thm:kconnect}, the $z$-degree of $\Lambda_{D_0}(a, z)$ is at most $c(D)-4$. The other nullification produces a removable kink and results in a prime $(c(D)-2)$-crossing alternating diagram $D_{\infty}$, as required by condition \eqref{d.case2} in Definition \ref{defn:near-alternating} defining a near-alternating diagram. Applying Theorem \ref{thm:coefflambda} and the subsequent remark, we get that the $z^{c(D)-3}$ term of $\Lambda_{D_{\infty}}(a, z)$ has coefficient $\kappa a^{-1}(a^{-1}+a)$ with $\kappa > 0$. Plugging this into the defining relation \eqref{eqn:k2poly} with $D_+=D$, $D_-$, $D_0$, and $D_{\infty}$,  we get that the coefficient of $z^{c(D)-2}$ in $\Lambda_D(a, z)$ is the same as the coefficient of $z^{c(D)-3}$ in $\Lambda_{D_{\infty}}(a, z)$, which is nonzero. This takes care of the base case. For $|r| > 2$, $D_0$ is a near-alternating diagram with $|r|-1$ negative crossings in the negative twist region, and that is where we apply the inductive hypothesis. We get  
\[\Lambda_{D_+}(a, z) + \underbrace{\Lambda_{D_-}(a, z)}_{\text{$z$-degree $\leq c(D)-3$}} = z(\underbrace{\Lambda_{D_0}(a, z)}_{\text{$z$-degree $= c(D)-3$}} + \underbrace{\Lambda_{D_{\infty}}(a, z)}_{\text{$z$-degree $\leq c(D)-4$}}).  \] 
This shows that the $z$-degree of $\Lambda_{D}(a, z)=\Lambda_{D_+}(a, z)$ is determined by the $z$-degree of $\Lambda_{D_0}(a, z)$ with the same coefficient. After multiplying $\Lambda_{D_0}(a, z)$ by $z$, we finish the proof of the theorem. 
\end{proof}

Using Theorem \ref{thm: adequatecrit}, \ref{thm:kconnect}, \ref{thm:coefflambda}, and Lemma \ref{lem:gennearalt}, we prove Theorem \ref{thm:naknoadequate}, which we restate here. 

\begin{restate2}
A near-alternating knot does not admit an adequate diagram. 
\end{restate2}
\begin{proof} 
Given a knot $K$ with a near-alternating diagram $D$ having a negative twist region of weight $r<0$ such that $|r|>2$, suppose that $K$ also admits a non-alternating, adequate diagram $D_A$. Then $D_A$ has a bridge of length $\geq 2$ and $c(D_A) = c(D)+r$ by Theorem \ref{thm:degree} and \ref{thm: adequatecrit}. But this contradicts Lemma \ref{lem:gennearalt} by Theorem \ref{thm:kconnect}, since Lemma \ref{lem:gennearalt} implies that  the $z$-degree of $\Lambda_D(a, z)$ for $D$ is $c(D)-2$, but Theorem \ref{thm:kconnect} applied to $D_A$  would imply that $\Lambda_{D_A}(a, z)$ has $z$-degree $\leq c(D)+r-2$. This is because $D$ and $D_A$ are related by a sequence of Type I, II, and III Reidemeister moves. A Type I Reidemeister move only affects the $a$-degree of $\Lambda_D(a, z)$, while the Type II and III moves leave $\Lambda_D(a, z)$ invariant. Thus the only other possibility is that it admits a reduced, alternating diagram with $c(D)+r$, with $|r|=1$ imposed by Lemma \ref{lem:gennearalt}, but this contradicts the assumption that $|r|>2$. 
\end{proof}

\section{Stable coefficients and volume bounds} \label{sec:cvolume}
In this section we prove Theorem \ref{thm:tail}, which we reprint here for reference. 

\begin{restate3}
 Let $K$ be a link admitting a near-alternating diagram $D = \partial (F_G)$, where $G$ is a finite 2-connected, weighted planar graph with a single negatively-weighted edge of weight $r < 0$. Then 
\begin{enumerate}[(1)]
\item the first and second coefficient, $\alpha_{0, n}, \alpha_{1, n}$, respectively, of the reduced colored Jones polynomial $\widehat{J_K}(v, n)$ of a near-alternating link $K$ are stable. The last and penultimate coefficient, $\alpha'_{0, n}, \alpha'_{1, n}$, respectively, are also stable.
\item Write $\alpha = \alpha_{0, n}$ and $\beta = \alpha_{1, n}$, and write $\alpha' = \alpha'_{0, n}$ and $\beta' = \alpha'_{1, n}$ for $n> 3$. We have $|\alpha| = 1$ and $|\beta| = \chi_1(s_{\sigma}(D)')$, where $\sigma$ is the Kauffman state giving the state surface $F_G$ and $\chi_1(s_{\sigma}(D)')$ is the first Betti number of the reduced graph of $s_{\sigma}(D)$. Similarly, we have $|\alpha'|=1$ and $|\beta'| = \chi_1(s_{-}(D)')$.
\end{enumerate} 
Furthermore,  if the diagram $D$ is also prime and twist-reduced with more than 7 crossings in each twist region, then $K$ is hyperbolic, and
\[.35367(|\beta|+|\beta'| -1)  < vol(S^3\setminus K) < 30v_3(|\beta|+|\beta'| - 2). \]
Here $v_3\approx 1.0149$ is the volume of a regular ideal tetrahedron.
In other words, there is a function on the stable coefficients of $K$ which is coarsely related to the volume of $S^3\setminus K$. 
\end{restate3}

Recall the $n$th-reduced colored Jones polynomial is defined as 
\begin{equation}
\widehat{J}_K(v, n) = J_K(v, n)/J_{\vcenter{\hbox{\includegraphics[scale=.05]{circ.png}}}}(v, n). 
\end{equation}
Note that since a near-alternating link $K$ with a near-alternating diagram $D$ is $-$-adequate, if we write  
\begin{equation} \label{eqn:scoeff} \widehat{J}_K(v, n) = \alpha_n v^{\widehat{d}(n)} + \beta_n v^{\widehat{d}(n)+4} + \cdots + \beta'_n v^{\widehat{d}^*(n) -4} + \alpha'_n v^{\widehat{d}^*(n)} , \end{equation}
where $\widehat{d}(n)$ is the minimum degree and $\widehat{d}^*(n)$ is the maximum degree of $\widehat{J}_K(v, n)$, respectively, then $|\beta'_n|= \chi(s_-(D)')$ and $|\alpha'_n|= 1$ by \cite[Theorem 3.1]{DL06}. So what we need to determine is $|\alpha_n|$ and $|\beta_n|$. 
 
We will first establish the stability of coefficients in Section \ref{subsec:stabcoeff}, then prove the two-sided volume bounds in Section \ref{subsec:2-sidedvol}.
\subsection{Stability of coefficients} \label{subsec:stabcoeff}
We shall apply the following result from \cite{DL06} to a $+$-adequate diagram approximating the alternating link diagram $D$.

\begin{thm}{\cite[Theorem 3.1]{DL06}} \label{thm:dlcoeff} Let $D$ be a $+$-adequate link diagram and $K$ be the link with diagram $D$. Write 
$\widehat{J}_K(v, n)$ as in \eqref{eqn:scoeff}. Then we have for all $n$, 
\[|\alpha_n| = 1 \text{ and } |\beta_n| = \chi_1(s_+(D)') .   \]  
\end{thm} 
 
\begin{lem} Let $D$ be a near-alternating link diagram and $K$ be the link with diagram $D$. Write $\widehat{J}_K(v, n)$ as in \eqref{eqn:scoeff}. Then we have for all $n$, 
\[|\alpha_n| = 1 \text{ and } |\beta_n| = |\chi_1(s_{\sigma}(D)')|,   \]
where $\sigma$ is the Kauffman state that chooses the $-$-resolution on crossings in the negative twist region of $D$ and the $+$-resolution for all other crossings.
\end{lem} 
 
\begin{proof}
From the proof of Theorem \ref{thm:bracketdegree} we see that the skein in the state sum realizing the degree comes from the state $\sigma_+$ which restricts to the $+$-resolution on crossings outside of the maximal negative twist region.
The last coefficient $\alpha_n$ is just the last coefficient of  $\langle \Sk^0_{\sigma_{+}} \rangle$ from $\langle  D^n_{\jwproj} \rangle$ of \eqref{eq:gssum} realizing the degree, so $|\alpha_n| = 1$. For the penultimate coefficient $\beta_n$, as long as $\frac{\omega}{t}>|r|$ with $|r|, t \geq 2$,  the inequality \eqref{eqn:inequality} implies that no skein $\Sk^a_{\sigma}$ from another state $\sigma$  with $c > 0$ split strands contributes to the penultimate coefficient. Therefore, we need only to consider the contribution of other skeins $\sigma$ with $c=0$ split strands. 

For a skein $\sigma$ with $0$ split strands we may remove the $r$ half twists on $n$ strands on the portion of the skein decorated by idempotents by reversing the fusion and untwisting of the maximal negative twist region, so 
\[\sum_{a \ : \ a, \ n, \, n \text{ admissible }}\langle \Sk^a_{\sigma} \rangle = (-1)^{nr}A^{r(n^2+2n)}\langle \Sk_{\sigma} \rangle, \] where $\Sk_{\sigma}$ is the new skein without the $r$ half twists on $n$ strands. In a process similar to that in \cite{DL06}, we consider Kauffman states (now on all the crossings of $\Sk_{\sigma}$) which chooses the $-$-resolution on a single crossing corresponding to a segment between a pair of circles in the state graph $s_+(\overline{\Sk_{\sigma}})$. They determine the penultimate coefficient of $\langle 
D^n_{\jwproj} \rangle$ since all other terms have lower degree.  Let $D_e$ be the reduced, alternating diagram obtained from $D$ by removing from $s_+(D)$ the edges corresponding to the crossings in the negative twist region of $D$, then recovering a link diagram by reversing the application of the all-$+$ Kauffman state. See Figure \ref{fig:dr} below for an example. 

\begin{figure}[H]
\def \svgwidth{.7\columnwidth}
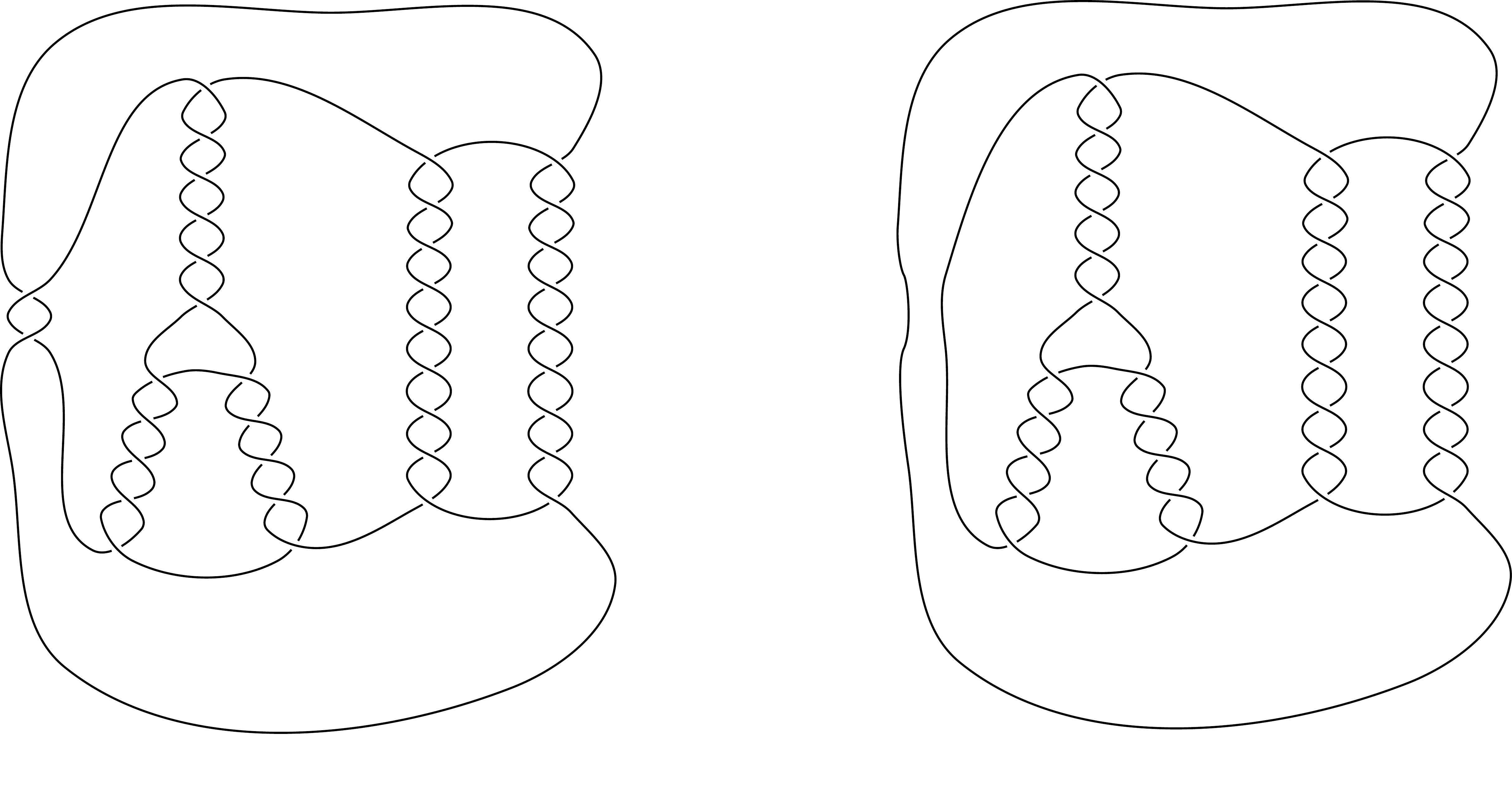
\caption{\label{fig:dr}}
\end{figure}

We know $D_e$ is reduced because of condition \eqref{d.case2} on $D$ in Definition \ref{defn:near-alternating} of a near-alternating diagram, since $D_e = \partial(F_{G/e})$ where $G/e$ is the graph $G$ with the negative edge $e$ contracted. There is a bijection between the set of Kauffman states of $D_e$ which contribute to the last and penultimate coefficients of $\langle D_e \rangle$ and the set 
\[ SC:=\{\Sk_{\sigma}: \sigma \text{ chooses the $-$-resolution on a single crossing of } c(D^n)\setminus r^n \}\]  by 
\[ \sigma \in SC \mapsto \sigma \text{ on } D_e.\]  
This implies that the penultimate coefficient of the sum
\[ \sum_{\sigma \text{ with } c = 0, \ \sigma \in SC}  \langle \Sk^a_{\sigma} \rangle = \sum_{\sigma \text{ with } c = 0, \ \sigma\in SC} (-1)^{nr}A^{r(n^2+2n)}\langle \Sk_{\sigma} \rangle \] is equal to the 2nd coefficient of the colored Jones polynomial of the link with the diagram $D_e$. Thus, they also have the same 2nd coefficient for the reduced polynomial. Since $D_e$ is adequate, we may apply Theorem \ref{thm:dlcoeff} to $D_e$. This gives that the 2nd coefficient of its reduced colored Jones polynomial is $e'_r-v_r+1$, where $e'_r$ is the number of edges in the reduced all-$+$ state graph $s_+(D_e)'$ and $v_r$ is the number of vertices of $s_+(D_e)'$. We compare this to the data from $D$,  where $e'$ is the number of edges of the reduced graph $s_{\sigma}(D)'$ and  $v$ is the number of vertices in $s_{\sigma}(D)'$. We get
\begin{equation}
 |\beta_n| = e_r'-v_r+1 = e'+r-(v+r)+1 = e'-v+1 = |\chi_1(s_{\sigma}(D)')| 
  \label{eq:2ndcoeff}
\end{equation}
The stability of these coefficients follows from the stability of the lst and 2nd coefficient  of the colored Jones polynomial of the link represented by $D_e$ since the computation was done independent of $n$.
\end{proof}  
\subsection{Two-sided volume bounds from stable coefficients $\alpha$, $\beta$, $\alpha'$, and $\beta'$} \label{subsec:2-sidedvol}

The following theorem from \cite{FKP08} provides volume bounds on a hyperbolic link complement based on the number of twist regions in a diagram of the link. 

\begin{thm}[{\cite[Theorem 1.2]{FKP08}}] \label{thm:twestimate}
 Let $K\subset S^3$ be a link with a prime, twist-reduced diagram $D$. Assume that $D$ has $tw(D)>2$ twist regions, and that each region contains at least  7 crossings. Then $K$ is a hyperbolic link satisfying 
\[0.70735 (tw(D)-1) < vol(S^3\setminus K) < 10v_3(tw(D)-1),  \] where $v_3\approx 1.0149$ is the volume of a regular ideal tetrahedron.  
\end{thm} 

\begin{thm}[{\cite[Theorem 1.5]{FKP08}}] \label{thm:twadequate} Let $K$ be a link in $S^3$ with an adequate diagram $D$ such that every twist region of $D$ contains at least 3 crossings. Then 
\[ \frac{1}{3} tw(D) + 1 \leq |\beta| +  |\beta'| \leq 2 tw(D). \] 
\end{thm} 

We use Theorem \ref{thm:twestimate} and Theorem \ref{thm:twadequate} to relate the number of twist regions $tw(D)$ of a link diagram $D$ to the stable coefficients $\alpha, \beta, \alpha',$ and $\beta'$, obtained in the previous section. In particular we show the following: 

\begin{lem} \label{lem:twistestimate} Let $K$ be a link with a near-alternating diagram that is prime and twist-reduced with at least 3 crossings in every positive twist region of $D$. Then 
\begin{align*}
&|\beta| + |\beta'|-1 \leq 2(tw(D)-1), \text{ and } |\beta| + |\beta'|-2 \geq \frac{tw(D)-1}{3}.
\end{align*} 
\end{lem} 

\begin{proof} Let $D^e = \partial(F_{G\setminus e})$ be the link diagram corresponding to $G\setminus e$ as in Definition \ref{defn:near-alternating}, see Figure \ref{fig:de} for an example.

\begin{figure}[H]
\def \svgwidth{.7\columnwidth}
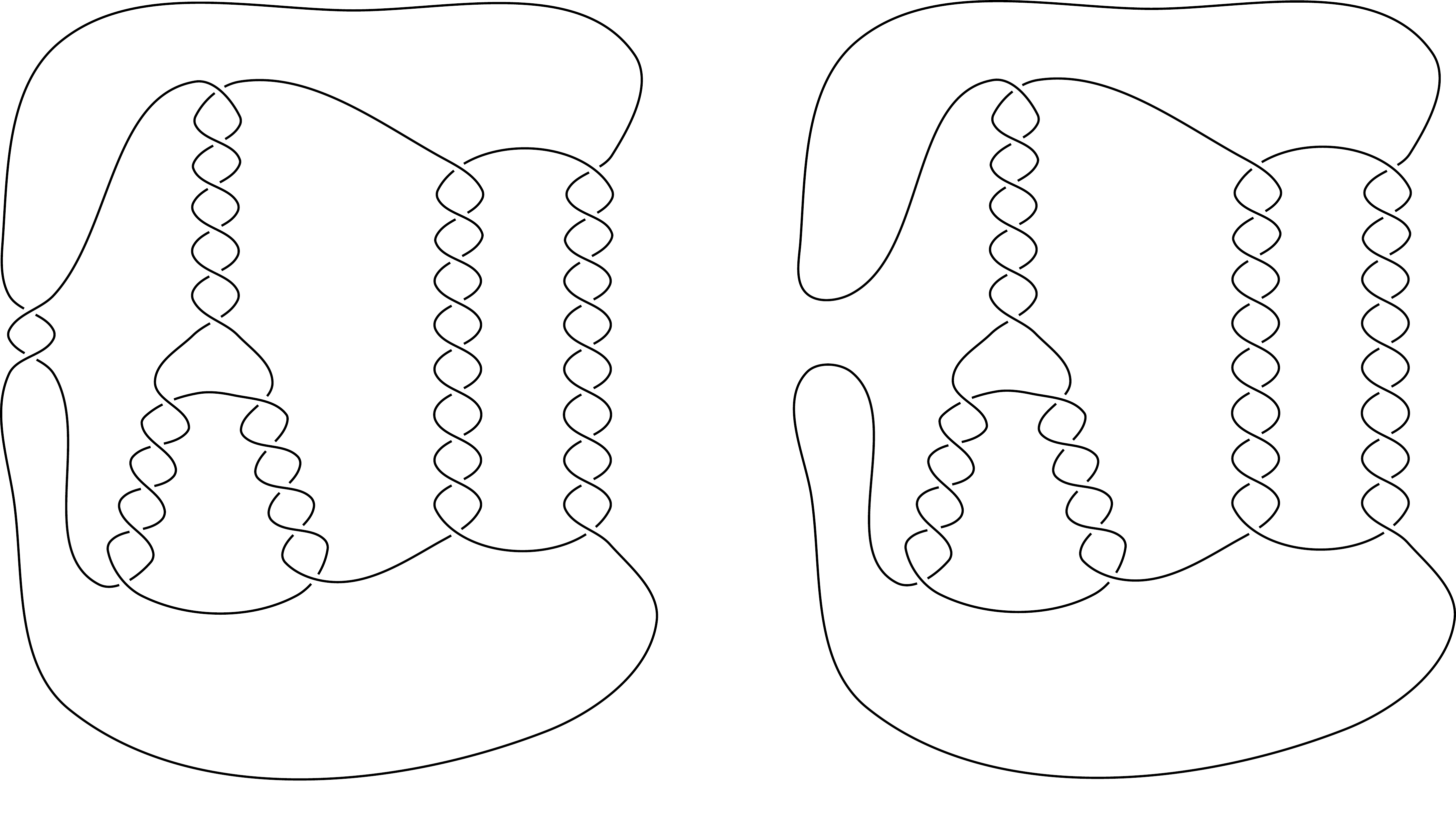
\caption{\label{fig:de}}
\end{figure}

We can immediately apply Theorem \ref{thm:twadequate} to $D^e$. By assumption, $D^e$ is prime, alternating, and twist-reduced. Let $e'_{+}, v_{+}$ be the number of edges and vertices in the reduced all-$+$ state graph of $D^e$, and $e'_{-}, v_{-}$ be the number of edges and vertices in the reduced all-$-$ state graph of $D^e$. In particular we get
\begin{align*}
& \frac{tw(D^e)}{3} + 1\leq e'_{+} + e'_{-} -v_{+} - v_{-} + 2 \leq 2tw(D^e).
%&e'_{+} + e'_{-} -v_{+} - v_{-} + 2 \geq \frac{tw(D^e)}{3} + 1.
\intertext{Since $D$ has one more twist region than $D^e$, this gives}
&\frac{tw(D)-1}{3} + 1\leq e'_{+} + e'_{-} -v_{+} - v_{-} + 2 \leq 2(tw(D)-1).
\intertext{Note that $D$ is $-$-adequate and we assume that $|r|>2$. Let $e'$ be the number of edges in the reduced graph of $s_{\sigma}(D)$ and $v = |s_{\sigma}(D)|$, and let $e'_D$, $v_D$ be the number of edges and the number of vertices in the reduced graph of $s_{-}(D)$, respectively. Using the result \eqref{eq:2ndcoeff} above on $|\beta|$, $|\beta'|$ we get} 
|\beta| + |\beta'| &= e'-v +1 + e'_D - v_D + 1. \\ 
\intertext{Substituting for quantities from $s_+(D^e)'$ and $s_-(D^e)'$ gives } 
%&= e'_{-}+1-(v_{-}+1)+1+(e'_{+}+r)-(v_{+}+(r-1))+1\\
|\beta| + |\beta'|&= e'_{+}+e'_{-}-v_+-v_{-}+2+1.
\intertext{So then} 
&|\beta| + |\beta'|-1 \leq 2(tw(D)-1) \text{ and }|\beta| + |\beta'|-2\geq \frac{tw(D)-1}{3}.
\end{align*}
\end{proof} 

\paragraph{\textbf{Proof of Theorem \ref{thm:tail}}}

Lemma \ref{lem:twistestimate} combined with Theorem \ref{thm:twestimate} then implies that 
\[ .35367(|\beta|+|\beta'| -1)  < vol(S^3\setminus K) < 30v_3(|\beta|+|\beta'| - 2). \] \qed

\bibliographystyle{amsalpha}
\bibliography{references}
\end{document}

%% file: nearaexa.pdf_tex
%% Creator: Inkscape inkscape 0.91, www.inkscape.org
%% Accompanies image file 'nearaexa.pdf' (pdf, eps, ps)
\begingroup%
  \makeatletter%
  \providecommand\color[2][]{%
    \errmessage{(Inkscape) Color is used for the text in Inkscape, but the package 'color.sty' is not loaded}%
    \renewcommand\color[2][]{}%
  }%
  \providecommand\transparent[1]{%
    \errmessage{(Inkscape) Transparency is used (non-zero) for the text in Inkscape, but the package 'transparent.sty' is not loaded}%
    \renewcommand\transparent[1]{}%
  }%
  \providecommand\rotatebox[2]{#2}%
  \ifx\svgwidth\undefined%
    \setlength{\unitlength}{4124.459375bp}%
    \ifx\svgscale\undefined%
      \relax%
    \else%
      \setlength{\unitlength}{\unitlength * \real{\svgscale}}%
    \fi%
  \else%
    \setlength{\unitlength}{\svgwidth}%
  \fi%
  \global\let\svgwidth\undefined%
  \global\let\svgscale\undefined%
  \makeatother%
  \begin{picture}(1,0.31285393)%
    \put(0,0){\includegraphics[width=\unitlength,page=1]{nearaexa.pdf}}%
    \put(0.34954633,0.30695922){\color[rgb]{0,0,1}\makebox(0,0)[lb]{\smash{$G$}}}%
    \put(0.51074986,0.1743066){\color[rgb]{0,0,1}\makebox(0,0)[lb]{\smash{$-2$}}}%
    \put(0.59333952,0.208791){\color[rgb]{0,0,1}\makebox(0,0)[lb]{\smash{$6$}}}%
    \put(0.57055998,0.11866133){\color[rgb]{0,0,1}\makebox(0,0)[lb]{\smash{$5$}}}%
    \put(0.62449597,0.08809562){\color[rgb]{0,0,1}\makebox(0,0)[lb]{\smash{$5$}}}%
    \put(0.68106679,0.15824){\color[rgb]{0,0,1}\makebox(0,0)[lb]{\smash{$9$}}}%
    \put(0.72878213,0.15824){\color[rgb]{0,0,1}\makebox(0,0)[lb]{\smash{$9$}}}%
    \put(0.62307027,0.29559504){\color[rgb]{0,0,1}\makebox(0,0)[lb]{\smash{$G$}}}%
    \put(0.64971732,0.27937058){\color[rgb]{0,0,1}\makebox(0,0)[lb]{\smash{$v$}}}%
    \put(0.64900449,0.0011403){\color[rgb]{0,0,1}\makebox(0,0)[lb]{\smash{$v'$}}}%
    \put(0.84395589,0.20985267){\color[rgb]{0,0,1}\makebox(0,0)[lb]{\smash{$6$}}}%
    \put(0.82117639,0.11972301){\color[rgb]{0,0,1}\makebox(0,0)[lb]{\smash{$5$}}}%
    \put(0.87511234,0.08915731){\color[rgb]{0,0,1}\makebox(0,0)[lb]{\smash{$5$}}}%
    \put(0.93168311,0.15930168){\color[rgb]{0,0,1}\makebox(0,0)[lb]{\smash{$9$}}}%
    \put(0.97939845,0.15930168){\color[rgb]{0,0,1}\makebox(0,0)[lb]{\smash{$9$}}}%
    \put(0.90033374,0.28043226){\color[rgb]{0,0,1}\makebox(0,0)[lb]{\smash{$v$}}}%
    \put(0.89962086,0.00142611){\color[rgb]{0,0,1}\makebox(0,0)[lb]{\smash{$v'$}}}%
    \put(0.87368668,0.30014809){\color[rgb]{0,0,1}\makebox(0,0)[lb]{\smash{$G\setminus e$}}}%
  \end{picture}%
\endgroup%

%% file: Kauffmanstate.pdf_tex
%% Creator: Inkscape inkscape 0.91, www.inkscape.org
%% Accompanies image file 'Kauffmanstate.pdf' (pdf, eps, ps)
\begingroup%
  \makeatletter%
  \providecommand\color[2][]{%
    \errmessage{(Inkscape) Color is used for the text in Inkscape, but the package 'color.sty' is not loaded}%
    \renewcommand\color[2][]{}%
  }%
  \providecommand\transparent[1]{%
    \errmessage{(Inkscape) Transparency is used (non-zero) for the text in Inkscape, but the package 'transparent.sty' is not loaded}%
    \renewcommand\transparent[1]{}%
  }%
  \providecommand\rotatebox[2]{#2}%
  \ifx\svgwidth\undefined%
    \setlength{\unitlength}{161.82550049bp}%
    \ifx\svgscale\undefined%
      \relax%
    \else%
      \setlength{\unitlength}{\unitlength * \real{\svgscale}}%
    \fi%
  \else%
    \setlength{\unitlength}{\svgwidth}%
  \fi%
  \global\let\svgwidth\undefined%
  \global\let\svgscale\undefined%
  \makeatother%
  \begin{picture}(1,1.04250961)%
    \put(0,0){\includegraphics[width=\unitlength,page=1]{Kauffmanstate.pdf}}%
    \put(0.75549458,0.0079923){\color[rgb]{0,0,0}\makebox(0,0)[lb]{\smash{$+$}}}%
    \put(0,0){\includegraphics[width=\unitlength,page=2]{Kauffmanstate.pdf}}%
    \put(0.7626265,0.64273952){\color[rgb]{0,0,0}\makebox(0,0)[lb]{\smash{$-$}}}%
    \put(0,0){\includegraphics[width=\unitlength,page=3]{Kauffmanstate.pdf}}%
  \end{picture}%
\endgroup%

%% file: statesurface.pdf_tex
%% Creator: Inkscape inkscape 0.91, www.inkscape.org
%% Accompanies image file 'statesurface.pdf' (pdf, eps, ps)
\begingroup%
  \makeatletter%
  \providecommand\color[2][]{%
    \errmessage{(Inkscape) Color is used for the text in Inkscape, but the package 'color.sty' is not loaded}%
    \renewcommand\color[2][]{}%
  }%
  \providecommand\transparent[1]{%
    \errmessage{(Inkscape) Transparency is used (non-zero) for the text in Inkscape, but the package 'transparent.sty' is not loaded}%
    \renewcommand\transparent[1]{}%
  }%
  \providecommand\rotatebox[2]{#2}%
  \ifx\svgwidth\undefined%
    \setlength{\unitlength}{165.88061523bp}%
    \ifx\svgscale\undefined%
      \relax%
    \else%
      \setlength{\unitlength}{\unitlength * \real{\svgscale}}%
    \fi%
  \else%
    \setlength{\unitlength}{\svgwidth}%
  \fi%
  \global\let\svgwidth\undefined%
  \global\let\svgscale\undefined%
  \makeatother%
  \begin{picture}(1,0.39144308)%
    \put(0.1237671,0.00779692){\color[rgb]{0,0,0}\makebox(0,0)[lb]{\smash{$+$}}}%
    \put(0.76309408,0.00779692){\color[rgb]{0,0,0}\makebox(0,0)[lb]{\smash{$-$}}}%
    \put(0,0){\includegraphics[width=\unitlength,page=1]{statesurface.pdf}}%
  \end{picture}%
\endgroup%

%% file: sigmasurface.pdf_tex
%% Creator: Inkscape inkscape 0.91, www.inkscape.org
%% Accompanies image file 'sigmasurface.pdf' (pdf, eps, ps)
\begingroup%
  \makeatletter%
  \providecommand\color[2][]{%
    \errmessage{(Inkscape) Color is used for the text in Inkscape, but the package 'color.sty' is not loaded}%
    \renewcommand\color[2][]{}%
  }%
  \providecommand\transparent[1]{%
    \errmessage{(Inkscape) Transparency is used (non-zero) for the text in Inkscape, but the package 'transparent.sty' is not loaded}%
    \renewcommand\transparent[1]{}%
  }%
  \providecommand\rotatebox[2]{#2}%
  \ifx\svgwidth\undefined%
    \setlength{\unitlength}{949.63839313bp}%
    \ifx\svgscale\undefined%
      \relax%
    \else%
      \setlength{\unitlength}{\unitlength * \real{\svgscale}}%
    \fi%
  \else%
    \setlength{\unitlength}{\svgwidth}%
  \fi%
  \global\let\svgwidth\undefined%
  \global\let\svgscale\undefined%
  \makeatother%
  \begin{picture}(1,1.18312433)%
    \put(0,0){\includegraphics[width=\unitlength,page=1]{sigmasurface.pdf}}%
  \end{picture}%
\endgroup%

%% file: almost-alternating.pdf_tex
%% Creator: Inkscape inkscape 0.92.2, www.inkscape.org
%% Accompanies image file 'almost-alternating.pdf' (pdf, eps, ps)
\begingroup%
  \makeatletter%
  \providecommand\color[2][]{%
    \errmessage{(Inkscape) Color is used for the text in Inkscape, but the package 'color.sty' is not loaded}%
    \renewcommand\color[2][]{}%
  }%
  \providecommand\transparent[1]{%
    \errmessage{(Inkscape) Transparency is used (non-zero) for the text in Inkscape, but the package 'transparent.sty' is not loaded}%
    \renewcommand\transparent[1]{}%
  }%
  \providecommand\rotatebox[2]{#2}%
  \ifx\svgwidth\undefined%
    \setlength{\unitlength}{654.86406337bp}%
    \ifx\svgscale\undefined%
      \relax%
    \else%
      \setlength{\unitlength}{\unitlength * \real{\svgscale}}%
    \fi%
  \else%
    \setlength{\unitlength}{\svgwidth}%
  \fi%
  \global\let\svgwidth\undefined%
  \global\let\svgscale\undefined%
  \makeatother%
  \begin{picture}(1,0.46140662)%
    \put(0,0){\includegraphics[width=\unitlength,page=1]{almost-alternating.pdf}}%
    \put(0.35272838,0.0063527){\color[rgb]{0,0,0}\makebox(0,0)[lb]{\smash{wrap the strand below the diagram}}}%
    \put(0.45302174,0.30248831){\color[rgb]{0,0,0}\makebox(0,0)[lb]{\smash{$\simeq$}}}%
    \put(0,0){\includegraphics[width=\unitlength,page=2]{almost-alternating.pdf}}%
    \put(-0.00277371,0.32850252){\color[rgb]{0,0,0}\makebox(0,0)[lb]{\smash{negative}}}%
    \put(0.05659273,0.28709385){\color[rgb]{0,0,0}\makebox(0,0)[lb]{\smash{twist}}}%
    \put(0.03514864,0.25160548){\color[rgb]{0,0,0}\makebox(0,0)[lb]{\smash{region}}}%
  \end{picture}%
\endgroup%

%% file: temperley-liebgen.pdf_tex
%% Creator: Inkscape inkscape 0.91, www.inkscape.org
%% Accompanies image file 'temperley-liebgen.pdf' (pdf, eps, ps)
\begingroup%
  \makeatletter%
  \providecommand\color[2][]{%
    \errmessage{(Inkscape) Color is used for the text in Inkscape, but the package 'color.sty' is not loaded}%
    \renewcommand\color[2][]{}%
  }%
  \providecommand\transparent[1]{%
    \errmessage{(Inkscape) Transparency is used (non-zero) for the text in Inkscape, but the package 'transparent.sty' is not loaded}%
    \renewcommand\transparent[1]{}%
  }%
  \providecommand\rotatebox[2]{#2}%
  \ifx\svgwidth\undefined%
    \setlength{\unitlength}{344.47871094bp}%
    \ifx\svgscale\undefined%
      \relax%
    \else%
      \setlength{\unitlength}{\unitlength * \real{\svgscale}}%
    \fi%
  \else%
    \setlength{\unitlength}{\svgwidth}%
  \fi%
  \global\let\svgwidth\undefined%
  \global\let\svgscale\undefined%
  \makeatother%
  \begin{picture}(1,0.35219167)%
    \put(0,0){\includegraphics[width=\unitlength,page=1]{temperley-liebgen.pdf}}%
    \put(0.60285979,0.00375454){\color[rgb]{0,0,0}\makebox(0,0)[lb]{\smash{$i$}}}%
    \put(0.64569023,0.00375454){\color[rgb]{0,0,0}\makebox(0,0)[lb]{\smash{$i+1$}}}%
    \put(0.11714212,0.27984842){\color[rgb]{0,0,0}\makebox(0,0)[lb]{\smash{$1_n$}}}%
    \put(0.62329856,0.28161402){\color[rgb]{0,0,0}\makebox(0,0)[lb]{\smash{$e^i_n$}}}%
  \end{picture}%
\endgroup%

%% file: jwrecursion.pdf_tex
%% Creator: Inkscape inkscape 0.48.4, www.inkscape.org
%% Accompanies image file 'jwrecursion.pdf' (pdf, eps, ps)
\begingroup%
  \makeatletter%
  \providecommand\color[2][]{%
    \errmessage{(Inkscape) Color is used for the text in Inkscape, but the package 'color.sty' is not loaded}%
    \renewcommand\color[2][]{}%
  }%
  \providecommand\transparent[1]{%
    \errmessage{(Inkscape) Transparency is used (non-zero) for the text in Inkscape, but the package 'transparent.sty' is not loaded}%
    \renewcommand\transparent[1]{}%
  }%
  \providecommand\rotatebox[2]{#2}%
  \ifx\svgwidth\undefined%
    \setlength{\unitlength}{524.85bp}%
    \ifx\svgscale\undefined%
      \relax%
    \else%
      \setlength{\unitlength}{\unitlength * \real{\svgscale}}%
    \fi%
  \else%
    \setlength{\unitlength}{\svgwidth}%
  \fi%
  \global\let\svgwidth\undefined%
  \global\let\svgscale\undefined%
  \makeatother%
  \begin{picture}(1,0.13246642)%
    \put(0,0){\includegraphics[width=\unitlength]{jwrecursion.pdf}}%
    \put(0.25344104,0.06136465){\color[rgb]{0,0,0}\makebox(0,0)[lb]{\smash{$=$}}}%
    \put(0.65870252,0.06105037){\color[rgb]{0,0,0}\makebox(0,0)[lb]{\smash{$\mathlarger{\frac{[n-1]}{[n]}}$}}}%
    \put(0.34129392,0.01945587){\color[rgb]{0,0,0}\makebox(0,0)[lb]{\smash{$n$}}}%
    \put(0.49678573,0.03464797){\color[rgb]{0,0,0}\makebox(0,0)[lb]{\smash{$1$}}}%
    \put(0.04350636,0.01945587){\color[rgb]{0,0,0}\makebox(0,0)[lb]{\smash{$n+1$}}}%
    \put(0.91454701,0.01846998){\color[rgb]{0,0,0}\makebox(0,0)[lb]{\smash{$1$}}}%
    \put(0.77656767,0.01945587){\color[rgb]{0,0,0}\makebox(0,0)[lb]{\smash{$n$}}}%
    \put(0.7682007,0.05733816){\color[rgb]{0,0,0}\makebox(0,0)[lb]{\smash{$n-1$}}}%
    \put(0.77656767,0.1013176){\color[rgb]{0,0,0}\makebox(0,0)[lb]{\smash{$n$}}}%
    \put(0.91454701,0.10230346){\color[rgb]{0,0,0}\makebox(0,0)[lb]{\smash{$1$}}}%
    \put(0.58970271,0.06053869){\color[rgb]{0,0,0}\makebox(0,0)[lb]{\smash{$+$}}}%
  \end{picture}%
\endgroup%

%% file: jwidentity.pdf_tex
%% Creator: Inkscape inkscape 0.91, www.inkscape.org
%% Accompanies image file 'jwidentity.pdf' (pdf, eps, ps)
\begingroup%
  \makeatletter%
  \providecommand\color[2][]{%
    \errmessage{(Inkscape) Color is used for the text in Inkscape, but the package 'color.sty' is not loaded}%
    \renewcommand\color[2][]{}%
  }%
  \providecommand\transparent[1]{%
    \errmessage{(Inkscape) Transparency is used (non-zero) for the text in Inkscape, but the package 'transparent.sty' is not loaded}%
    \renewcommand\transparent[1]{}%
  }%
  \providecommand\rotatebox[2]{#2}%
  \ifx\svgwidth\undefined%
    \setlength{\unitlength}{300.64115029bp}%
    \ifx\svgscale\undefined%
      \relax%
    \else%
      \setlength{\unitlength}{\unitlength * \real{\svgscale}}%
    \fi%
  \else%
    \setlength{\unitlength}{\svgwidth}%
  \fi%
  \global\let\svgwidth\undefined%
  \global\let\svgscale\undefined%
  \makeatother%
  \begin{picture}(1,0.24020306)%
    \put(0,0){\includegraphics[width=\unitlength,page=1]{jwidentity.pdf}}%
    \put(0.0727968,0.19900328){\color[rgb]{0,0,0}\makebox(0,0)[lb]{\smash{$i$}}}%
    \put(0.07127624,0.11404196){\color[rgb]{0,0,0}\makebox(0,0)[lb]{\smash{$i$}}}%
    \put(0.31361546,0.14635392){\color[rgb]{0,0,0}\makebox(0,0)[lb]{\smash{$j$}}}%
    \put(0.09047331,0.01900702){\color[rgb]{0,0,0}\makebox(0,0)[lb]{\smash{$i+j$}}}%
    \put(0.67588887,0.01900702){\color[rgb]{0,0,0}\makebox(0,0)[lb]{\smash{$i+j$}}}%
  \end{picture}%
\endgroup%

%% file: trivalent.pdf_tex
%% Creator: Inkscape inkscape 0.91, www.inkscape.org
%% Accompanies image file 'trivalent.pdf' (pdf, eps, ps)
\begingroup%
  \makeatletter%
  \providecommand\color[2][]{%
    \errmessage{(Inkscape) Color is used for the text in Inkscape, but the package 'color.sty' is not loaded}%
    \renewcommand\color[2][]{}%
  }%
  \providecommand\transparent[1]{%
    \errmessage{(Inkscape) Transparency is used (non-zero) for the text in Inkscape, but the package 'transparent.sty' is not loaded}%
    \renewcommand\transparent[1]{}%
  }%
  \providecommand\rotatebox[2]{#2}%
  \ifx\svgwidth\undefined%
    \setlength{\unitlength}{254.39665527bp}%
    \ifx\svgscale\undefined%
      \relax%
    \else%
      \setlength{\unitlength}{\unitlength * \real{\svgscale}}%
    \fi%
  \else%
    \setlength{\unitlength}{\svgwidth}%
  \fi%
  \global\let\svgwidth\undefined%
  \global\let\svgscale\undefined%
  \makeatother%
  \begin{picture}(1,0.3853796)%
    \put(0.4468055,0.18112753){\color[rgb]{0,0,0}\makebox(0,0)[lb]{\smash{$=$}}}%
    \put(0,0){\includegraphics[width=\unitlength,page=1]{trivalent.pdf}}%
    \put(0.67450458,0.27955307){\color[rgb]{0,0,0}\makebox(0,0)[lb]{\smash{$a$}}}%
    \put(0.96227639,0.27955306){\color[rgb]{0,0,0}\makebox(0,0)[lb]{\smash{$b$}}}%
    \put(0.86468823,0.05887329){\color[rgb]{0,0,0}\makebox(0,0)[lb]{\smash{$c$}}}%
    \put(0,0){\includegraphics[width=\unitlength,page=2]{trivalent.pdf}}%
    \put(0.17610245,0.30263592){\color[rgb]{0,0,0}\makebox(0,0)[lb]{\smash{$x$}}}%
    \put(0.25789393,0.18942683){\color[rgb]{0,0,0}\makebox(0,0)[lb]{\smash{$y$}}}%
    \put(0.09746014,0.18942683){\color[rgb]{0,0,0}\makebox(0,0)[lb]{\smash{$z$}}}%
    \put(0.00937933,0.28589111){\color[rgb]{0,0,0}\makebox(0,0)[lb]{\smash{$a$}}}%
    \put(0.33968561,0.28589111){\color[rgb]{0,0,0}\makebox(0,0)[lb]{\smash{$b$}}}%
    \put(0.21070877,0.06521141){\color[rgb]{0,0,0}\makebox(0,0)[lb]{\smash{$c$}}}%
  \end{picture}%
\endgroup%

%% file: fusion.pdf_tex
%% Creator: Inkscape inkscape 0.91, www.inkscape.org
%% Accompanies image file 'fusion.pdf' (pdf, eps, ps)
\begingroup%
  \makeatletter%
  \providecommand\color[2][]{%
    \errmessage{(Inkscape) Color is used for the text in Inkscape, but the package 'color.sty' is not loaded}%
    \renewcommand\color[2][]{}%
  }%
  \providecommand\transparent[1]{%
    \errmessage{(Inkscape) Transparency is used (non-zero) for the text in Inkscape, but the package 'transparent.sty' is not loaded}%
    \renewcommand\transparent[1]{}%
  }%
  \providecommand\rotatebox[2]{#2}%
  \ifx\svgwidth\undefined%
    \setlength{\unitlength}{1184.33369141bp}%
    \ifx\svgscale\undefined%
      \relax%
    \else%
      \setlength{\unitlength}{\unitlength * \real{\svgscale}}%
    \fi%
  \else%
    \setlength{\unitlength}{\svgwidth}%
  \fi%
  \global\let\svgwidth\undefined%
  \global\let\svgscale\undefined%
  \makeatother%
  \begin{picture}(1,0.17261915)%
    \put(0.10064974,0.08299731){\color[rgb]{0,0,0}\makebox(0,0)[lb]{\smash{$= \stackrel{\mathlarger{\sum}}{\substack{c \ : \  a, b, c  \\ \text{ admissible }}}$ }}}%
    \put(0,0){\includegraphics[width=\unitlength,page=1]{fusion.pdf}}%
    \put(-0.00067306,0.14710325){\color[rgb]{0,0,0}\makebox(0,0)[lb]{\smash{$a$}}}%
    \put(0.09389489,0.14710325){\color[rgb]{0,0,0}\makebox(0,0)[lb]{\smash{$b$}}}%
    \put(0,0){\includegraphics[width=\unitlength,page=2]{fusion.pdf}}%
    \put(0.27525587,0.15094365){\color[rgb]{0,0,0}\makebox(0,0)[lb]{\smash{$a$}}}%
    \put(0.34382454,0.15094365){\color[rgb]{0,0,0}\makebox(0,0)[lb]{\smash{$b$}}}%
    \put(0,0){\includegraphics[width=\unitlength,page=3]{fusion.pdf}}%
    \put(0.32016053,0.08382598){\color[rgb]{0,0,0}\makebox(0,0)[lb]{\smash{$c$}}}%
    \put(0.27525581,0.01584666){\color[rgb]{0,0,0}\makebox(0,0)[lb]{\smash{$a$}}}%
    \put(0.34112255,0.01584666){\color[rgb]{0,0,0}\makebox(0,0)[lb]{\smash{$b$}}}%
    \put(0.21818418,0.0925669){\color[rgb]{0,0,0}\makebox(0,0)[lb]{\smash{$\frac{\triangle_c}{\theta(a, b, c)}$}}}%
    \put(0,0){\includegraphics[width=\unitlength,page=4]{fusion.pdf}}%
    \put(0.49783509,0.14232282){\color[rgb]{0,0,0}\makebox(0,0)[lb]{\smash{$a$}}}%
    \put(0.58902561,0.14232282){\color[rgb]{0,0,0}\makebox(0,0)[lb]{\smash{$b$}}}%
    \put(0.55594068,0.03242456){\color[rgb]{0,0,0}\makebox(0,0)[lb]{\smash{$c$}}}%
    \put(0,0){\includegraphics[width=\unitlength,page=5]{fusion.pdf}}%
    \put(0.89231849,0.13996701){\color[rgb]{0,0,0}\makebox(0,0)[lb]{\smash{$a$}}}%
    \put(0.98520463,0.13996701){\color[rgb]{0,0,0}\makebox(0,0)[lb]{\smash{$b$}}}%
    \put(0.95073285,0.03006875){\color[rgb]{0,0,0}\makebox(0,0)[lb]{\smash{$c$}}}%
    \put(0.60597722,0.0881607){\color[rgb]{0,0,0}\makebox(0,0)[lb]{\smash{$=(-1)^{\frac{a+b-c}{2}}A^{a+b-c+\frac{a^2+b^2-c^2}{2}}$}}}%
    \put(0.42353171,0.09352741){\color[rgb]{0,0,0}\makebox(0,0)[lb]{\smash{and}}}%
  \end{picture}%
\endgroup%

%% file: theta.pdf_tex
%% Creator: Inkscape inkscape 0.48.4, www.inkscape.org
%% Accompanies image file 'theta.pdf' (pdf, eps, ps)
\begingroup%
  \makeatletter%
  \providecommand\color[2][]{%
    \errmessage{(Inkscape) Color is used for the text in Inkscape, but the package 'color.sty' is not loaded}%
    \renewcommand\color[2][]{}%
  }%
  \providecommand\transparent[1]{%
    \errmessage{(Inkscape) Transparency is used (non-zero) for the text in Inkscape, but the package 'transparent.sty' is not loaded}%
    \renewcommand\transparent[1]{}%
  }%
  \providecommand\rotatebox[2]{#2}%
  \ifx\svgwidth\undefined%
    \setlength{\unitlength}{118.13879395bp}%
    \ifx\svgscale\undefined%
      \relax%
    \else%
      \setlength{\unitlength}{\unitlength * \real{\svgscale}}%
    \fi%
  \else%
    \setlength{\unitlength}{\svgwidth}%
  \fi%
  \global\let\svgwidth\undefined%
  \global\let\svgscale\undefined%
  \makeatother%
  \begin{picture}(1,1.23341603)%
    \put(0,0){\includegraphics[width=\unitlength]{theta.pdf}}%
    \put(0.02034155,0.61414276){\color[rgb]{0,0,0}\makebox(0,0)[lb]{\smash{$a$}}}%
    \put(0.43873556,0.61414276){\color[rgb]{0,0,0}\makebox(0,0)[lb]{\smash{$b$}}}%
    \put(0.81843432,0.61414276){\color[rgb]{0,0,0}\makebox(0,0)[lb]{\smash{$c$}}}%
  \end{picture}%
\endgroup%

%% file: jwidentity3.pdf_tex
%% Creator: Inkscape inkscape 0.91, www.inkscape.org
%% Accompanies image file 'jwidentity3.pdf' (pdf, eps, ps)
\begingroup%
  \makeatletter%
  \providecommand\color[2][]{%
    \errmessage{(Inkscape) Color is used for the text in Inkscape, but the package 'color.sty' is not loaded}%
    \renewcommand\color[2][]{}%
  }%
  \providecommand\transparent[1]{%
    \errmessage{(Inkscape) Transparency is used (non-zero) for the text in Inkscape, but the package 'transparent.sty' is not loaded}%
    \renewcommand\transparent[1]{}%
  }%
  \providecommand\rotatebox[2]{#2}%
  \ifx\svgwidth\undefined%
    \setlength{\unitlength}{464.525bp}%
    \ifx\svgscale\undefined%
      \relax%
    \else%
      \setlength{\unitlength}{\unitlength * \real{\svgscale}}%
    \fi%
  \else%
    \setlength{\unitlength}{\svgwidth}%
  \fi%
  \global\let\svgwidth\undefined%
  \global\let\svgscale\undefined%
  \makeatother%
  \begin{picture}(1,0.44077283)%
    \put(0.20679777,0.30472087){\color[rgb]{0,0,0}\makebox(0,0)[lb]{\smash{}}}%
    \put(0,0){\includegraphics[width=\unitlength,page=1]{jwidentity3.pdf}}%
    \put(0.03148377,0.30280449){\color[rgb]{0,0,0}\makebox(0,0)[lb]{\smash{$x$}}}%
    \put(0.13481513,0.23391691){\color[rgb]{0,0,0}\makebox(0,0)[lb]{\smash{$y$}}}%
    \put(0.24503525,0.30280449){\color[rgb]{0,0,0}\makebox(0,0)[lb]{\smash{$z$}}}%
    \put(0.13481513,0.13058555){\color[rgb]{0,0,0}\makebox(0,0)[lb]{\smash{$1$}}}%
    \put(0,0){\includegraphics[width=\unitlength,page=2]{jwidentity3.pdf}}%
    \put(0.74447016,0.30280449){\color[rgb]{0,0,0}\makebox(0,0)[lb]{\smash{$x$}}}%
    \put(0.80646897,0.17880685){\color[rgb]{0,0,0}\makebox(0,0)[lb]{\smash{$y-1$}}}%
    \put(0.92013347,0.30280449){\color[rgb]{0,0,0}\makebox(0,0)[lb]{\smash{$z$}}}%
    \put(0.39314353,0.20808407){\color[rgb]{0,0,0}\makebox(0,0)[lb]{\smash{$=-\frac{[x+y+z][y-1]}{[x+y-1][z+y-1]}$}}}%
  \end{picture}%
\endgroup%

%% file: framenegt.pdf_tex
%% Creator: Inkscape inkscape 0.91, www.inkscape.org
%% Accompanies image file 'framenegt.pdf' (pdf, eps, ps)
\begingroup%
  \makeatletter%
  \providecommand\color[2][]{%
    \errmessage{(Inkscape) Color is used for the text in Inkscape, but the package 'color.sty' is not loaded}%
    \renewcommand\color[2][]{}%
  }%
  \providecommand\transparent[1]{%
    \errmessage{(Inkscape) Transparency is used (non-zero) for the text in Inkscape, but the package 'transparent.sty' is not loaded}%
    \renewcommand\transparent[1]{}%
  }%
  \providecommand\rotatebox[2]{#2}%
  \ifx\svgwidth\undefined%
    \setlength{\unitlength}{246.25202637bp}%
    \ifx\svgscale\undefined%
      \relax%
    \else%
      \setlength{\unitlength}{\unitlength * \real{\svgscale}}%
    \fi%
  \else%
    \setlength{\unitlength}{\svgwidth}%
  \fi%
  \global\let\svgwidth\undefined%
  \global\let\svgscale\undefined%
  \makeatother%
  \begin{picture}(1,0.55504307)%
    \put(0,0){\includegraphics[width=\unitlength,page=1]{framenegt.pdf}}%
    \put(-0.01774361,0.09147483){\color[rgb]{0,0,0}\makebox(0,0)[lb]{\smash{$r$}}}%
    \put(0.60980516,0.09147483){\color[rgb]{0,0,0}\makebox(0,0)[lb]{\smash{$r$}}}%
    \put(0.487661,0.24729029){\color[rgb]{0,0,0}\makebox(0,0)[lb]{\smash{$\leadsto$}}}%
  \end{picture}%
\endgroup%

%% file: case.pdf_tex
%% Creator: Inkscape inkscape 0.92.2, www.inkscape.org
%% Accompanies image file 'case.pdf' (pdf, eps, ps)
\begingroup%
  \makeatletter%
  \providecommand\color[2][]{%
    \errmessage{(Inkscape) Color is used for the text in Inkscape, but the package 'color.sty' is not loaded}%
    \renewcommand\color[2][]{}%
  }%
  \providecommand\transparent[1]{%
    \errmessage{(Inkscape) Transparency is used (non-zero) for the text in Inkscape, but the package 'transparent.sty' is not loaded}%
    \renewcommand\transparent[1]{}%
  }%
  \providecommand\rotatebox[2]{#2}%
  \ifx\svgwidth\undefined%
    \setlength{\unitlength}{322.5bp}%
    \ifx\svgscale\undefined%
      \relax%
    \else%
      \setlength{\unitlength}{\unitlength * \real{\svgscale}}%
    \fi%
  \else%
    \setlength{\unitlength}{\svgwidth}%
  \fi%
  \global\let\svgwidth\undefined%
  \global\let\svgscale\undefined%
  \makeatother%
  \begin{picture}(1,0.97674419)%
    \put(0,0){\includegraphics[width=\unitlength,page=1]{case.pdf}}%
    \put(0.51210746,0.47520368){\color[rgb]{0,0,0}\makebox(0,0)[lb]{\smash{$a$}}}%
    \put(0.14550824,0.47520368){\color[rgb]{0,0,0}\makebox(0,0)[lb]{\smash{$c$}}}%
    \put(0.76766876,0.47520368){\color[rgb]{0,0,0}\makebox(0,0)[lb]{\smash{$c$}}}%
    \put(0,0){\includegraphics[width=\unitlength,page=2]{case.pdf}}%
  \end{picture}%
\endgroup%

%% file: local.pdf_tex
%% Creator: Inkscape inkscape 0.48.4, www.inkscape.org
%% Accompanies image file 'local.pdf' (pdf, eps, ps)
\begingroup%
  \makeatletter%
  \providecommand\color[2][]{%
    \errmessage{(Inkscape) Color is used for the text in Inkscape, but the package 'color.sty' is not loaded}%
    \renewcommand\color[2][]{}%
  }%
  \providecommand\transparent[1]{%
    \errmessage{(Inkscape) Transparency is used (non-zero) for the text in Inkscape, but the package 'transparent.sty' is not loaded}%
    \renewcommand\transparent[1]{}%
  }%
  \providecommand\rotatebox[2]{#2}%
  \ifx\svgwidth\undefined%
    \setlength{\unitlength}{153.225bp}%
    \ifx\svgscale\undefined%
      \relax%
    \else%
      \setlength{\unitlength}{\unitlength * \real{\svgscale}}%
    \fi%
  \else%
    \setlength{\unitlength}{\svgwidth}%
  \fi%
  \global\let\svgwidth\undefined%
  \global\let\svgscale\undefined%
  \makeatother%
  \begin{picture}(1,0.6811878)%
    \put(0,0){\includegraphics[width=\unitlength]{local.pdf}}%
    \put(0.42006384,0.59214674){\color[rgb]{0,0,0}\makebox(0,0)[lb]{\smash{$n-c$}}}%
    \put(0.44766101,0.29342636){\color[rgb]{0,0,0}\makebox(0,0)[lb]{\smash{$x$}}}%
    \put(0.08218541,0.43364966){\color[rgb]{0,0,0}\makebox(0,0)[lb]{\smash{$c$}}}%
    \put(0.85229464,0.42805565){\color[rgb]{0,0,0}\makebox(0,0)[lb]{\smash{$c$}}}%
    \put(0.61660019,0.19226793){\color[rgb]{0,0,0}\makebox(0,0)[lb]{\smash{$y$}}}%
    \put(0.28356991,0.19035592){\color[rgb]{0,0,0}\makebox(0,0)[lb]{\smash{$z$}}}%
    \put(0.55208262,0.03274532){\color[rgb]{0,0,0}\makebox(0,0)[lb]{\smash{$a$}}}%
  \end{picture}%
\endgroup%

%% file: localh.pdf_tex
%% Creator: Inkscape inkscape 0.91, www.inkscape.org
%% Accompanies image file 'localh.pdf' (pdf, eps, ps)
\begingroup%
  \makeatletter%
  \providecommand\color[2][]{%
    \errmessage{(Inkscape) Color is used for the text in Inkscape, but the package 'color.sty' is not loaded}%
    \renewcommand\color[2][]{}%
  }%
  \providecommand\transparent[1]{%
    \errmessage{(Inkscape) Transparency is used (non-zero) for the text in Inkscape, but the package 'transparent.sty' is not loaded}%
    \renewcommand\transparent[1]{}%
  }%
  \providecommand\rotatebox[2]{#2}%
  \ifx\svgwidth\undefined%
    \setlength{\unitlength}{153.225bp}%
    \ifx\svgscale\undefined%
      \relax%
    \else%
      \setlength{\unitlength}{\unitlength * \real{\svgscale}}%
    \fi%
  \else%
    \setlength{\unitlength}{\svgwidth}%
  \fi%
  \global\let\svgwidth\undefined%
  \global\let\svgscale\undefined%
  \makeatother%
  \begin{picture}(1,0.6811878)%
    \put(0,0){\includegraphics[width=\unitlength,page=1]{localh.pdf}}%
    \put(0.32599868,0.59211857){\color[rgb]{0,0,0}\makebox(0,0)[lb]{\smash{$n-c-x$}}}%
    \put(0,0){\includegraphics[width=\unitlength,page=2]{localh.pdf}}%
    \put(0.44757521,0.30384032){\color[rgb]{0,0,0}\makebox(0,0)[lb]{\smash{$x$}}}%
    \put(0.0820996,0.43362147){\color[rgb]{0,0,0}\makebox(0,0)[lb]{\smash{$c$}}}%
    \put(0.85220892,0.42802755){\color[rgb]{0,0,0}\makebox(0,0)[lb]{\smash{$c$}}}%
    \put(0.64170338,0.1861584){\color[rgb]{0,0,0}\makebox(0,0)[lb]{\smash{$y$}}}%
    \put(0.26073502,0.18424639){\color[rgb]{0,0,0}\makebox(0,0)[lb]{\smash{$y$}}}%
    \put(0.55199681,0.03271715){\color[rgb]{0,0,0}\makebox(0,0)[lb]{\smash{$a$}}}%
    \put(0,0){\includegraphics[width=\unitlength,page=3]{localh.pdf}}%
    \put(0.44757521,0.40826192){\color[rgb]{0,0,0}\makebox(0,0)[lb]{\smash{$x$}}}%
    \put(0,0){\includegraphics[width=\unitlength,page=4]{localh.pdf}}%
  \end{picture}%
\endgroup%

%% file: fourex.pdf_tex
%% Creator: Inkscape inkscape 0.91, www.inkscape.org
%% Accompanies image file 'fourex.pdf' (pdf, eps, ps)
\begingroup%
  \makeatletter%
  \providecommand\color[2][]{%
    \errmessage{(Inkscape) Color is used for the text in Inkscape, but the package 'color.sty' is not loaded}%
    \renewcommand\color[2][]{}%
  }%
  \providecommand\transparent[1]{%
    \errmessage{(Inkscape) Transparency is used (non-zero) for the text in Inkscape, but the package 'transparent.sty' is not loaded}%
    \renewcommand\transparent[1]{}%
  }%
  \providecommand\rotatebox[2]{#2}%
  \ifx\svgwidth\undefined%
    \setlength{\unitlength}{911.93046875bp}%
    \ifx\svgscale\undefined%
      \relax%
    \else%
      \setlength{\unitlength}{\unitlength * \real{\svgscale}}%
    \fi%
  \else%
    \setlength{\unitlength}{\svgwidth}%
  \fi%
  \global\let\svgwidth\undefined%
  \global\let\svgscale\undefined%
  \makeatother%
  \begin{picture}(1,0.27453016)%
    \put(0,0){\includegraphics[width=\unitlength,page=1]{fourex.pdf}}%
    \put(0.07731955,0.2647775){\color[rgb]{0,0,0}\makebox(0,0)[lb]{\smash{$n-c-1$}}}%
    \put(0.18271988,0.20538654){\color[rgb]{0,0,0}\makebox(0,0)[lb]{\smash{$c$}}}%
    \put(0.10296493,0.20184188){\color[rgb]{0,0,0}\makebox(0,0)[lb]{\smash{1}}}%
    \put(0.13220842,0.02017777){\color[rgb]{0,0,0}\makebox(0,0)[lb]{\smash{$a$}}}%
    \put(0,0){\includegraphics[width=\unitlength,page=2]{fourex.pdf}}%
    \put(0.34049751,0.26653202){\color[rgb]{0,0,0}\makebox(0,0)[lb]{\smash{$n-c-1$}}}%
    \put(0.44063429,0.20538654){\color[rgb]{0,0,0}\makebox(0,0)[lb]{\smash{$c$}}}%
    \put(0.36087933,0.212369){\color[rgb]{0,0,0}\makebox(0,0)[lb]{\smash{1}}}%
    \put(0.39012283,0.02017777){\color[rgb]{0,0,0}\makebox(0,0)[lb]{\smash{$a$}}}%
    \put(0.60718452,0.26653202){\color[rgb]{0,0,0}\makebox(0,0)[lb]{\smash{$n-c-1$}}}%
    \put(0.70732129,0.20538654){\color[rgb]{0,0,0}\makebox(0,0)[lb]{\smash{$c$}}}%
    \put(0.62756633,0.212369){\color[rgb]{0,0,0}\makebox(0,0)[lb]{\smash{1}}}%
    \put(0.65680983,0.02017777){\color[rgb]{0,0,0}\makebox(0,0)[lb]{\smash{$a$}}}%
    \put(0.87562604,0.26653202){\color[rgb]{0,0,0}\makebox(0,0)[lb]{\smash{$n-c-1$}}}%
    \put(0.97576281,0.20538654){\color[rgb]{0,0,0}\makebox(0,0)[lb]{\smash{$c$}}}%
    \put(0.89414653,0.18956024){\color[rgb]{0,0,0}\makebox(0,0)[lb]{\smash{1}}}%
    \put(0.92525135,0.02017777){\color[rgb]{0,0,0}\makebox(0,0)[lb]{\smash{$a$}}}%
    \put(0.89414653,0.122449){\color[rgb]{0,0,0}\makebox(0,0)[lb]{\smash{1}}}%
    \put(-0.00811735,0.07938312){\color[rgb]{0,0,0}\makebox(0,0)[lb]{\smash{$n-1$}}}%
    \put(0.1655801,0.11973708){\color[rgb]{0,0,0}\makebox(0,0)[lb]{\smash{$n-1$}}}%
    \put(0.24453349,0.07587408){\color[rgb]{0,0,0}\makebox(0,0)[lb]{\smash{$n-1$}}}%
    \put(0.42524902,0.07938312){\color[rgb]{0,0,0}\makebox(0,0)[lb]{\smash{$n-1$}}}%
    \put(0.42524902,0.15131843){\color[rgb]{0,0,0}\makebox(0,0)[lb]{\smash{$n-2$}}}%
    \put(0.00726792,0.20538654){\color[rgb]{0,0,0}\makebox(0,0)[lb]{\smash{$c$}}}%
    \put(0.50420241,0.07938312){\color[rgb]{0,0,0}\makebox(0,0)[lb]{\smash{$n-1$}}}%
    \put(0.50420241,0.15131843){\color[rgb]{0,0,0}\makebox(0,0)[lb]{\smash{$n-2$}}}%
    \put(0.26869137,0.20538654){\color[rgb]{0,0,0}\makebox(0,0)[lb]{\smash{$c$}}}%
    \put(0.54064192,0.20538654){\color[rgb]{0,0,0}\makebox(0,0)[lb]{\smash{$c$}}}%
    \put(0.6901815,0.08991024){\color[rgb]{0,0,0}\makebox(0,0)[lb]{\smash{$n-1$}}}%
    \put(0.77615294,0.07938312){\color[rgb]{0,0,0}\makebox(0,0)[lb]{\smash{$n-1$}}}%
    \put(0.77615294,0.14780939){\color[rgb]{0,0,0}\makebox(0,0)[lb]{\smash{$n-2$}}}%
    \put(0.96388655,0.07938312){\color[rgb]{0,0,0}\makebox(0,0)[lb]{\smash{$n-1$}}}%
    \put(0.96388655,0.14780939){\color[rgb]{0,0,0}\makebox(0,0)[lb]{\smash{$n-2$}}}%
    \put(0.80381988,0.20538654){\color[rgb]{0,0,0}\makebox(0,0)[lb]{\smash{$c$}}}%
  \end{picture}%
\endgroup%

%% file: induction.pdf_tex
%% Creator: Inkscape inkscape 0.91, www.inkscape.org
%% Accompanies image file 'induction.pdf' (pdf, eps, ps)
\begingroup%
  \makeatletter%
  \providecommand\color[2][]{%
    \errmessage{(Inkscape) Color is used for the text in Inkscape, but the package 'color.sty' is not loaded}%
    \renewcommand\color[2][]{}%
  }%
  \providecommand\transparent[1]{%
    \errmessage{(Inkscape) Transparency is used (non-zero) for the text in Inkscape, but the package 'transparent.sty' is not loaded}%
    \renewcommand\transparent[1]{}%
  }%
  \providecommand\rotatebox[2]{#2}%
  \ifx\svgwidth\undefined%
    \setlength{\unitlength}{944.60927734bp}%
    \ifx\svgscale\undefined%
      \relax%
    \else%
      \setlength{\unitlength}{\unitlength * \real{\svgscale}}%
    \fi%
  \else%
    \setlength{\unitlength}{\svgwidth}%
  \fi%
  \global\let\svgwidth\undefined%
  \global\let\svgscale\undefined%
  \makeatother%
  \begin{picture}(1,0.26376241)%
    \put(0,0){\includegraphics[width=\unitlength,page=1]{induction.pdf}}%
    \put(0.09326009,0.25434714){\color[rgb]{0,0,0}\makebox(0,0)[lb]{\smash{$n-c-x$}}}%
    \put(0.21195233,0.19531699){\color[rgb]{0,0,0}\makebox(0,0)[lb]{\smash{$c$}}}%
    \put(0.1631883,0.01651554){\color[rgb]{0,0,0}\makebox(0,0)[lb]{\smash{$a$}}}%
    \put(0,0){\includegraphics[width=\unitlength,page=2]{induction.pdf}}%
    \put(0.33208901,0.25604096){\color[rgb]{0,0,0}\makebox(0,0)[lb]{\smash{$n-c-x$}}}%
    \put(0.44908741,0.19701081){\color[rgb]{0,0,0}\makebox(0,0)[lb]{\smash{$c$}}}%
    \put(0.358541,0.13769267){\color[rgb]{0,0,0}\makebox(0,0)[lb]{\smash{1}}}%
    \put(0.40371104,0.01820936){\color[rgb]{0,0,0}\makebox(0,0)[lb]{\smash{$a$}}}%
    \put(0.59801907,0.25604096){\color[rgb]{0,0,0}\makebox(0,0)[lb]{\smash{$n-c-x$}}}%
    \put(0.71840511,0.19701081){\color[rgb]{0,0,0}\makebox(0,0)[lb]{\smash{$c$}}}%
    \put(0.65495988,0.13938649){\color[rgb]{0,0,0}\makebox(0,0)[lb]{\smash{1}}}%
    \put(0.6696411,0.01820936){\color[rgb]{0,0,0}\makebox(0,0)[lb]{\smash{$a$}}}%
    \put(0.12617312,0.06987016){\color[rgb]{0,0,0}\makebox(0,0)[lb]{\smash{$x-1$}}}%
    \put(0.36754276,0.06987016){\color[rgb]{0,0,0}\makebox(0,0)[lb]{\smash{$x-1$}}}%
    \put(0.63008518,0.07241089){\color[rgb]{0,0,0}\makebox(0,0)[lb]{\smash{$x-1$}}}%
    \put(0.13495649,0.15971233){\color[rgb]{0,0,0}\makebox(0,0)[lb]{\smash{1}}}%
    \put(0,0){\includegraphics[width=\unitlength,page=3]{induction.pdf}}%
    \put(0.86564295,0.25604096){\color[rgb]{0,0,0}\makebox(0,0)[lb]{\smash{$n-c-x$}}}%
    \put(0.97755988,0.19701081){\color[rgb]{0,0,0}\makebox(0,0)[lb]{\smash{$c$}}}%
    \put(0.89876713,0.18173203){\color[rgb]{0,0,0}\makebox(0,0)[lb]{\smash{1}}}%
    \put(0.92879592,0.01820936){\color[rgb]{0,0,0}\makebox(0,0)[lb]{\smash{$a$}}}%
    \put(0.89876713,0.12541161){\color[rgb]{0,0,0}\makebox(0,0)[lb]{\smash{1}}}%
    \put(0.89516832,0.07241089){\color[rgb]{0,0,0}\makebox(0,0)[lb]{\smash{$x-1$}}}%
    \put(0.10584726,0.20706974){\color[rgb]{0,0,0}\makebox(0,0)[lb]{\smash{$x-1$}}}%
    \put(0.35060454,0.20706974){\color[rgb]{0,0,0}\makebox(0,0)[lb]{\smash{$x-1$}}}%
    \put(0.20190584,0.10891545){\color[rgb]{0,0,0}\makebox(0,0)[lb]{\smash{$n-x$}}}%
    \put(0.01438087,0.08342074){\color[rgb]{0,0,0}\makebox(0,0)[lb]{\smash{$n-x$}}}%
    \put(0.43766352,0.14053804){\color[rgb]{0,0,0}\makebox(0,0)[lb]{\smash{$n-2$}}}%
    \put(0.6165346,0.20791665){\color[rgb]{0,0,0}\makebox(0,0)[lb]{\smash{$x-1$}}}%
    \put(0.5206608,0.14053804){\color[rgb]{0,0,0}\makebox(0,0)[lb]{\smash{$n-2$}}}%
    \put(0.26678028,0.07503901){\color[rgb]{0,0,0}\makebox(0,0)[lb]{\smash{$n-x$}}}%
    \put(0.04765159,0.19531699){\color[rgb]{0,0,0}\makebox(0,0)[lb]{\smash{$c$}}}%
    \put(0.28139903,0.19531699){\color[rgb]{0,0,0}\makebox(0,0)[lb]{\smash{$c$}}}%
    \put(0.55410438,0.19701081){\color[rgb]{0,0,0}\makebox(0,0)[lb]{\smash{$c$}}}%
    \put(0.70717403,0.07503901){\color[rgb]{0,0,0}\makebox(0,0)[lb]{\smash{$n-x$}}}%
    \put(0.79355896,0.07503901){\color[rgb]{0,0,0}\makebox(0,0)[lb]{\smash{$n-x$}}}%
    \put(0.79167232,0.14561951){\color[rgb]{0,0,0}\makebox(0,0)[lb]{\smash{$n-2$}}}%
    \put(0.96952364,0.14900716){\color[rgb]{0,0,0}\makebox(0,0)[lb]{\smash{$n-2$}}}%
    \put(0.81664679,0.19701081){\color[rgb]{0,0,0}\makebox(0,0)[lb]{\smash{$c$}}}%
    \put(0.9663288,0.07503901){\color[rgb]{0,0,0}\makebox(0,0)[lb]{\smash{$n-x$}}}%
    \put(0.52254741,0.07503901){\color[rgb]{0,0,0}\makebox(0,0)[lb]{\smash{$n-x$}}}%
    \put(0.43616248,0.07503901){\color[rgb]{0,0,0}\makebox(0,0)[lb]{\smash{$n-x$}}}%
  \end{picture}%
\endgroup%

%% file: induction2.pdf_tex
%% Creator: Inkscape inkscape 0.91, www.inkscape.org
%% Accompanies image file 'induction2.pdf' (pdf, eps, ps)
\begingroup%
  \makeatletter%
  \providecommand\color[2][]{%
    \errmessage{(Inkscape) Color is used for the text in Inkscape, but the package 'color.sty' is not loaded}%
    \renewcommand\color[2][]{}%
  }%
  \providecommand\transparent[1]{%
    \errmessage{(Inkscape) Transparency is used (non-zero) for the text in Inkscape, but the package 'transparent.sty' is not loaded}%
    \renewcommand\transparent[1]{}%
  }%
  \providecommand\rotatebox[2]{#2}%
  \ifx\svgwidth\undefined%
    \setlength{\unitlength}{558.26875bp}%
    \ifx\svgscale\undefined%
      \relax%
    \else%
      \setlength{\unitlength}{\unitlength * \real{\svgscale}}%
    \fi%
  \else%
    \setlength{\unitlength}{\svgwidth}%
  \fi%
  \global\let\svgwidth\undefined%
  \global\let\svgscale\undefined%
  \makeatother%
  \begin{picture}(1,0.44199575)%
    \put(0,0){\includegraphics[width=\unitlength,page=1]{induction2.pdf}}%
    \put(0.35146864,0.19452826){\color[rgb]{0,0,0}\makebox(0,0)[lb]{\smash{$=-\frac{[2n-x-1][x-2]}{[n-2]^2}$}}}%
    \put(-0.00190321,0.12460611){\color[rgb]{0,0,0}\makebox(0,0)[lb]{\smash{$n-x$}}}%
    \put(0.29159311,0.12460611){\color[rgb]{0,0,0}\makebox(0,0)[lb]{\smash{$n-x$}}}%
    \put(0,0){\includegraphics[width=\unitlength,page=2]{induction2.pdf}}%
    \put(0.12536492,0.42893081){\color[rgb]{0,0,0}\makebox(0,0)[lb]{\smash{$n-c-x$}}}%
    \put(0.30040208,0.32761684){\color[rgb]{0,0,0}\makebox(0,0)[lb]{\smash{$c$}}}%
    \put(0.16708212,0.30176462){\color[rgb]{0,0,0}\makebox(0,0)[lb]{\smash{1}}}%
    \put(0.21789184,0.02507884){\color[rgb]{0,0,0}\makebox(0,0)[lb]{\smash{$a$}}}%
    \put(0.16708212,0.2064686){\color[rgb]{0,0,0}\makebox(0,0)[lb]{\smash{1}}}%
    \put(0.16099282,0.11678963){\color[rgb]{0,0,0}\makebox(0,0)[lb]{\smash{$x-1$}}}%
    \put(0.64294749,0.12226847){\color[rgb]{0,0,0}\makebox(0,0)[lb]{\smash{$n-x$}}}%
    \put(0,0){\includegraphics[width=\unitlength,page=3]{induction2.pdf}}%
    \put(0.76862249,0.42874532){\color[rgb]{0,0,0}\makebox(0,0)[lb]{\smash{$n-c-x$}}}%
    \put(0.94652565,0.32886435){\color[rgb]{0,0,0}\makebox(0,0)[lb]{\smash{$c$}}}%
    \put(0.8640154,0.02632637){\color[rgb]{0,0,0}\makebox(0,0)[lb]{\smash{$a$}}}%
    \put(0.80711639,0.11803716){\color[rgb]{0,0,0}\makebox(0,0)[lb]{\smash{$x-2$}}}%
    \put(0.93357782,0.12226847){\color[rgb]{0,0,0}\makebox(0,0)[lb]{\smash{$n-x$}}}%
    \put(0.81193283,0.30463063){\color[rgb]{0,0,0}\makebox(0,0)[lb]{\smash{1}}}%
    \put(0.01666777,0.33048284){\color[rgb]{0,0,0}\makebox(0,0)[lb]{\smash{$c$}}}%
    \put(0.67584849,0.32761684){\color[rgb]{0,0,0}\makebox(0,0)[lb]{\smash{$c$}}}%
    \put(-0.00190321,0.24211224){\color[rgb]{0,0,0}\makebox(0,0)[lb]{\smash{$n-2$}}}%
    \put(0.29042912,0.24211224){\color[rgb]{0,0,0}\makebox(0,0)[lb]{\smash{$n-2$}}}%
    \put(0.64867954,0.24211224){\color[rgb]{0,0,0}\makebox(0,0)[lb]{\smash{$n-2$}}}%
    \put(0.93527985,0.24211224){\color[rgb]{0,0,0}\makebox(0,0)[lb]{\smash{$n-2$}}}%
  \end{picture}%
\endgroup%

%% file: induction3.pdf_tex
%% Creator: Inkscape inkscape 0.91, www.inkscape.org
%% Accompanies image file 'induction3.pdf' (pdf, eps, ps)
\begingroup%
  \makeatletter%
  \providecommand\color[2][]{%
    \errmessage{(Inkscape) Color is used for the text in Inkscape, but the package 'color.sty' is not loaded}%
    \renewcommand\color[2][]{}%
  }%
  \providecommand\transparent[1]{%
    \errmessage{(Inkscape) Transparency is used (non-zero) for the text in Inkscape, but the package 'transparent.sty' is not loaded}%
    \renewcommand\transparent[1]{}%
  }%
  \providecommand\rotatebox[2]{#2}%
  \ifx\svgwidth\undefined%
    \setlength{\unitlength}{898.69384766bp}%
    \ifx\svgscale\undefined%
      \relax%
    \else%
      \setlength{\unitlength}{\unitlength * \real{\svgscale}}%
    \fi%
  \else%
    \setlength{\unitlength}{\svgwidth}%
  \fi%
  \global\let\svgwidth\undefined%
  \global\let\svgscale\undefined%
  \makeatother%
  \begin{picture}(1,0.27964097)%
    \put(0,0){\includegraphics[width=\unitlength,page=1]{induction3.pdf}}%
    \put(0.08061488,0.26440358){\color[rgb]{0,0,0}\makebox(0,0)[lb]{\smash{$n-c-x$}}}%
    \put(0.19646943,0.20235753){\color[rgb]{0,0,0}\makebox(0,0)[lb]{\smash{$c$}}}%
    \put(0.14521406,0.01442089){\color[rgb]{0,0,0}\makebox(0,0)[lb]{\smash{$a$}}}%
    \put(0.10986838,0.07139164){\color[rgb]{0,0,0}\makebox(0,0)[lb]{\smash{$x-2$}}}%
    \put(0.18967018,0.07798207){\color[rgb]{0,0,0}\makebox(0,0)[lb]{\smash{}}}%
    \put(-0.00118227,0.07452281){\color[rgb]{0,0,0}\makebox(0,0)[lb]{\smash{$n-x$}}}%
    \put(0.09663813,0.2182294){\color[rgb]{0,0,0}\makebox(0,0)[lb]{\smash{$x-2$}}}%
    \put(0,0){\includegraphics[width=\unitlength,page=2]{induction3.pdf}}%
    \put(0.32398799,0.26556167){\color[rgb]{0,0,0}\makebox(0,0)[lb]{\smash{$n-c-x$}}}%
    \put(0.24575157,0.07568093){\color[rgb]{0,0,0}\makebox(0,0)[lb]{\smash{$n-x$}}}%
    \put(0,0){\includegraphics[width=\unitlength,page=3]{induction3.pdf}}%
    \put(0.43984254,0.20351565){\color[rgb]{0,0,0}\makebox(0,0)[lb]{\smash{$c$}}}%
    \put(0.38858717,0.01557901){\color[rgb]{0,0,0}\makebox(0,0)[lb]{\smash{$a$}}}%
    \put(0.35324149,0.06898904){\color[rgb]{0,0,0}\makebox(0,0)[lb]{\smash{$x-2$}}}%
    \put(0.42111718,0.07568093){\color[rgb]{0,0,0}\makebox(0,0)[lb]{\smash{}}}%
    \put(0.33979723,0.21997429){\color[rgb]{0,0,0}\makebox(0,0)[lb]{\smash{$x-2$}}}%
    \put(0,0){\includegraphics[width=\unitlength,page=4]{induction3.pdf}}%
    \put(0.34921237,0.13276579){\color[rgb]{0,0,0}\makebox(0,0)[lb]{\smash{$1$}}}%
    \put(0,0){\includegraphics[width=\unitlength,page=5]{induction3.pdf}}%
    \put(0.59104489,0.26697264){\color[rgb]{0,0,0}\makebox(0,0)[lb]{\smash{$n-c-x$}}}%
    \put(0.51280852,0.07709189){\color[rgb]{0,0,0}\makebox(0,0)[lb]{\smash{$n-x$}}}%
    \put(0,0){\includegraphics[width=\unitlength,page=6]{induction3.pdf}}%
    \put(0.70689939,0.20492661){\color[rgb]{0,0,0}\makebox(0,0)[lb]{\smash{$c$}}}%
    \put(0.65564401,0.01698996){\color[rgb]{0,0,0}\makebox(0,0)[lb]{\smash{$a$}}}%
    \put(0.62029845,0.07039999){\color[rgb]{0,0,0}\makebox(0,0)[lb]{\smash{$x-2$}}}%
    \put(0.68105263,0.07709189){\color[rgb]{0,0,0}\makebox(0,0)[lb]{\smash{}}}%
    \put(0.61240923,0.22613955){\color[rgb]{0,0,0}\makebox(0,0)[lb]{\smash{$x-2$}}}%
    \put(0,0){\includegraphics[width=\unitlength,page=7]{induction3.pdf}}%
    \put(0.62575928,0.13534111){\color[rgb]{0,0,0}\makebox(0,0)[lb]{\smash{$1$}}}%
    \put(0,0){\includegraphics[width=\unitlength,page=8]{induction3.pdf}}%
    \put(0.85173415,0.26556167){\color[rgb]{0,0,0}\makebox(0,0)[lb]{\smash{$n-c-x$}}}%
    \put(0.9675887,0.20351565){\color[rgb]{0,0,0}\makebox(0,0)[lb]{\smash{$c$}}}%
    \put(0.91633333,0.01557901){\color[rgb]{0,0,0}\makebox(0,0)[lb]{\smash{$a$}}}%
    \put(0.88098765,0.06720868){\color[rgb]{0,0,0}\makebox(0,0)[lb]{\smash{$x-2$}}}%
    \put(0.96488654,0.07568093){\color[rgb]{0,0,0}\makebox(0,0)[lb]{\smash{}}}%
    \put(0.77349767,0.07568093){\color[rgb]{0,0,0}\makebox(0,0)[lb]{\smash{$n-x$}}}%
    \put(0.86775741,0.22650896){\color[rgb]{0,0,0}\makebox(0,0)[lb]{\smash{$x-2$}}}%
    \put(0,0){\includegraphics[width=\unitlength,page=9]{induction3.pdf}}%
    \put(0.88503215,0.12908746){\color[rgb]{0,0,0}\makebox(0,0)[lb]{\smash{$1$}}}%
    \put(0.88439005,0.16825675){\color[rgb]{0,0,0}\makebox(0,0)[lb]{\smash{$1$}}}%
    \put(0.81889385,0.27549074){\color[rgb]{0,0,0}\makebox(0,0)[lb]{\smash{}}}%
    \put(0.10999085,0.13722297){\color[rgb]{0,0,0}\makebox(0,0)[lb]{\smash{$1$}}}%
    \put(0.18969017,0.1338652){\color[rgb]{0,0,0}\makebox(0,0)[lb]{\smash{$n-2$}}}%
    \put(-0.00080852,0.1338652){\color[rgb]{0,0,0}\makebox(0,0)[lb]{\smash{$n-2$}}}%
    \put(0.18753606,0.07452281){\color[rgb]{0,0,0}\makebox(0,0)[lb]{\smash{$n-x$}}}%
    \put(0.02377436,0.20235753){\color[rgb]{0,0,0}\makebox(0,0)[lb]{\smash{$c$}}}%
    \put(0.43348023,0.16560729){\color[rgb]{0,0,0}\makebox(0,0)[lb]{\smash{$n-3$}}}%
    \put(0.43348023,0.12109825){\color[rgb]{0,0,0}\makebox(0,0)[lb]{\smash{$n-2$}}}%
    \put(0.24654226,0.12465897){\color[rgb]{0,0,0}\makebox(0,0)[lb]{\smash{$n-2$}}}%
    \put(0.42912881,0.07568093){\color[rgb]{0,0,0}\makebox(0,0)[lb]{\smash{$n-x$}}}%
    \put(0.5135965,0.1602662){\color[rgb]{0,0,0}\makebox(0,0)[lb]{\smash{$n-3$}}}%
    \put(0.70330715,0.07709189){\color[rgb]{0,0,0}\makebox(0,0)[lb]{\smash{$n-x$}}}%
    \put(0.51181614,0.12287861){\color[rgb]{0,0,0}\makebox(0,0)[lb]{\smash{$n-2$}}}%
    \put(0.70765591,0.12821969){\color[rgb]{0,0,0}\makebox(0,0)[lb]{\smash{$n-2$}}}%
    \put(0.27070819,0.20351565){\color[rgb]{0,0,0}\makebox(0,0)[lb]{\smash{$c$}}}%
    \put(0.53776243,0.20351565){\color[rgb]{0,0,0}\makebox(0,0)[lb]{\smash{$c$}}}%
    \put(0.79057638,0.20492661){\color[rgb]{0,0,0}\makebox(0,0)[lb]{\smash{$c$}}}%
    \put(0.77530965,0.15314475){\color[rgb]{0,0,0}\makebox(0,0)[lb]{\smash{$n-3$}}}%
    \put(0.95690653,0.15314475){\color[rgb]{0,0,0}\makebox(0,0)[lb]{\smash{$n-3$}}}%
    \put(0.96043564,0.07568093){\color[rgb]{0,0,0}\makebox(0,0)[lb]{\smash{$n-x$}}}%
    \put(0,0){\includegraphics[width=\unitlength,page=10]{induction3.pdf}}%
  \end{picture}%
\endgroup%

%% file: casefuntwist.pdf_tex
%% Creator: Inkscape inkscape 0.91, www.inkscape.org
%% Accompanies image file 'casefuntwist.pdf' (pdf, eps, ps)
\begingroup%
  \makeatletter%
  \providecommand\color[2][]{%
    \errmessage{(Inkscape) Color is used for the text in Inkscape, but the package 'color.sty' is not loaded}%
    \renewcommand\color[2][]{}%
  }%
  \providecommand\transparent[1]{%
    \errmessage{(Inkscape) Transparency is used (non-zero) for the text in Inkscape, but the package 'transparent.sty' is not loaded}%
    \renewcommand\transparent[1]{}%
  }%
  \providecommand\rotatebox[2]{#2}%
  \ifx\svgwidth\undefined%
    \setlength{\unitlength}{1036.08066406bp}%
    \ifx\svgscale\undefined%
      \relax%
    \else%
      \setlength{\unitlength}{\unitlength * \real{\svgscale}}%
    \fi%
  \else%
    \setlength{\unitlength}{\svgwidth}%
  \fi%
  \global\let\svgwidth\undefined%
  \global\let\svgscale\undefined%
  \makeatother%
  \begin{picture}(1,0.40610248)%
    \put(0,0){\includegraphics[width=\unitlength,page=1]{casefuntwist.pdf}}%
    \put(0.24180461,0.21240919){\color[rgb]{0,0,0}\makebox(0,0)[lb]{\smash{$a$}}}%
    \put(0.40251181,0.21549776){\color[rgb]{0,0,0}\makebox(0,0)[lb]{\smash{$c$}}}%
    \put(0.02548366,0.21081788){\color[rgb]{0,0,0}\makebox(0,0)[lb]{\smash{$c$}}}%
    \put(0.20826548,0.3761736){\color[rgb]{0,0,0}\makebox(0,0)[lb]{\smash{$n-c$}}}%
    \put(0.28378241,0.27164063){\color[rgb]{0,0,0}\makebox(0,0)[lb]{\smash{$\frac{a}{2}$}}}%
    \put(0.14325281,0.27164063){\color[rgb]{0,0,0}\makebox(0,0)[lb]{\smash{$\frac{a}{2}$}}}%
    \put(0,0){\includegraphics[width=\unitlength,page=2]{casefuntwist.pdf}}%
    \put(0.45942671,0.21081788){\color[rgb]{0,0,0}\makebox(0,0)[lb]{\smash{$c-\frac{a}{2}$}}}%
    \put(0.74567538,0.38080645){\color[rgb]{0,0,0}\makebox(0,0)[lb]{\smash{$n-c$}}}%
    \put(0.20358559,0.04858205){\color[rgb]{0,0,0}\makebox(0,0)[lb]{\smash{$n-c$}}}%
    \put(0.28069385,0.15813596){\color[rgb]{0,0,0}\makebox(0,0)[lb]{\smash{$\frac{a}{2}$}}}%
    \put(0.14325281,0.15813596){\color[rgb]{0,0,0}\makebox(0,0)[lb]{\smash{$\frac{a}{2}$}}}%
    \put(0,0){\includegraphics[width=\unitlength,page=3]{casefuntwist.pdf}}%
    \put(0.70972187,0.21549776){\color[rgb]{0,0,0}\makebox(0,0)[lb]{\smash{$\frac{a}{2}$}}}%
    \put(0.98565056,0.21393778){\color[rgb]{0,0,0}\makebox(0,0)[lb]{\smash{$c-\frac{a}{2}$}}}%
    \put(0.74786649,0.04784912){\color[rgb]{0,0,0}\makebox(0,0)[lb]{\smash{$n-c$}}}%
    \put(0.10919718,0.00218523){\color[rgb]{0,0,0}\makebox(0,0)[lb]{\smash{$J^a_{\sigma}$, where $\frac{a}{2}\leq c$}}}%
    \put(0.6960241,0.00218523){\color[rgb]{0,0,0}\makebox(0,0)[lb]{\smash{$\overline{J^a_{\sigma}}$ }}}%
    \put(0.79774586,0.21549776){\color[rgb]{0,0,0}\makebox(0,0)[lb]{\smash{$\frac{a}{2}$}}}%
  \end{picture}%
\endgroup%

%% file: skeinflow.pdf_tex
%% Creator: Inkscape inkscape 0.91, www.inkscape.org
%% Accompanies image file 'skeinflow.pdf' (pdf, eps, ps)
\begingroup%
  \makeatletter%
  \providecommand\color[2][]{%
    \errmessage{(Inkscape) Color is used for the text in Inkscape, but the package 'color.sty' is not loaded}%
    \renewcommand\color[2][]{}%
  }%
  \providecommand\transparent[1]{%
    \errmessage{(Inkscape) Transparency is used (non-zero) for the text in Inkscape, but the package 'transparent.sty' is not loaded}%
    \renewcommand\transparent[1]{}%
  }%
  \providecommand\rotatebox[2]{#2}%
  \ifx\svgwidth\undefined%
    \setlength{\unitlength}{1450.50712891bp}%
    \ifx\svgscale\undefined%
      \relax%
    \else%
      \setlength{\unitlength}{\unitlength * \real{\svgscale}}%
    \fi%
  \else%
    \setlength{\unitlength}{\svgwidth}%
  \fi%
  \global\let\svgwidth\undefined%
  \global\let\svgscale\undefined%
  \makeatother%
  \begin{picture}(1,0.30888921)%
    \put(0,0){\includegraphics[width=\unitlength,page=1]{skeinflow.pdf}}%
    \put(0.45293046,0.00765383){\color[rgb]{0,0,0}\makebox(0,0)[lb]{\smash{All-$+$}}}%
    \put(0.83425133,0.00367112){\color[rgb]{0,0,0}\makebox(0,0)[lb]{\smash{$\sigma$}}}%
    \put(0,0){\includegraphics[width=\unitlength,page=2]{skeinflow.pdf}}%
    \put(0.02373444,0.2954801){\color[rgb]{0,0,0}\makebox(0,0)[lb]{\smash{$n=3$}}}%
    \put(0.16713257,0.2954801){\color[rgb]{0,0,0}\makebox(0,0)[lb]{\smash{$n=3$}}}%
    \put(0.09488266,0.00843804){\color[rgb]{0,0,0}\makebox(0,0)[lb]{\smash{$x^n$}}}%
    \put(0,0){\includegraphics[width=\unitlength,page=3]{skeinflow.pdf}}%
  \end{picture}%
\endgroup%

%% file: markings.pdf_tex
%% Creator: Inkscape inkscape 0.91, www.inkscape.org
%% Accompanies image file 'markings.pdf' (pdf, eps, ps)
\begingroup%
  \makeatletter%
  \providecommand\color[2][]{%
    \errmessage{(Inkscape) Color is used for the text in Inkscape, but the package 'color.sty' is not loaded}%
    \renewcommand\color[2][]{}%
  }%
  \providecommand\transparent[1]{%
    \errmessage{(Inkscape) Transparency is used (non-zero) for the text in Inkscape, but the package 'transparent.sty' is not loaded}%
    \renewcommand\transparent[1]{}%
  }%
  \providecommand\rotatebox[2]{#2}%
  \ifx\svgwidth\undefined%
    \setlength{\unitlength}{454.14648438bp}%
    \ifx\svgscale\undefined%
      \relax%
    \else%
      \setlength{\unitlength}{\unitlength * \real{\svgscale}}%
    \fi%
  \else%
    \setlength{\unitlength}{\svgwidth}%
  \fi%
  \global\let\svgwidth\undefined%
  \global\let\svgscale\undefined%
  \makeatother%
  \begin{picture}(1,1.02009156)%
    \put(0.2241503,0.59955125){\color[rgb]{0,0,0}\makebox(0,0)[lb]{\smash{}}}%
    \put(0.26567246,0.63352394){\color[rgb]{0,0,0}\makebox(0,0)[lb]{\smash{}}}%
    \put(0,0){\includegraphics[width=\unitlength,page=1]{markings.pdf}}%
    \put(0.21589784,0.97726397){\color[rgb]{0,0,0}\makebox(0,0)[lb]{\smash{$n$}}}%
    \put(0.64923819,0.97726397){\color[rgb]{0,0,0}\makebox(0,0)[lb]{\smash{$n$}}}%
    \put(0.46075496,0.06047488){\color[rgb]{0,0,0}\makebox(0,0)[lb]{\smash{$x^n$}}}%
    \put(0,0){\includegraphics[width=\unitlength,page=2]{markings.pdf}}%
    \put(0.21407689,0.65006559){\color[rgb]{0,0,0}\makebox(0,0)[lb]{\smash{$U_3$}}}%
    \put(0.21407689,0.48095716){\color[rgb]{0,0,0}\makebox(0,0)[lb]{\smash{$L_3$}}}%
    \put(0.46069335,0.79098928){\color[rgb]{0,0,0}\makebox(0,0)[lb]{\smash{$U_1$}}}%
    \put(0.46421644,0.33298728){\color[rgb]{0,0,0}\makebox(0,0)[lb]{\smash{$L_1$}}}%
    \put(0.60917393,0.5580291){\color[rgb]{0,0,0}\makebox(0,0)[lb]{\smash{$C^u_3=C^{\ell}_3$}}}%
    \put(0.59809399,0.67120417){\color[rgb]{0,0,0}\rotatebox{57.06704211}{\makebox(0,0)[lb]{\smash{$C^u_2$}}}}%
    \put(0.3667083,0.30962476){\color[rgb]{0,0,0}\rotatebox{59.5912394}{\makebox(0,0)[lb]{\smash{$C^{\ell}_1$}}}}%
    \put(0.54524784,0.70995826){\color[rgb]{0,0,0}\rotatebox{57.06704211}{\makebox(0,0)[lb]{\smash{$C^u_1$}}}}%
    \put(0.30329329,0.33428627){\color[rgb]{0,0,0}\rotatebox{59.5912394}{\makebox(0,0)[lb]{\smash{$C^{\ell}_2$}}}}%
    \put(0.21589784,0.138768){\color[rgb]{0,0,0}\makebox(0,0)[lb]{\smash{$n$}}}%
    \put(0.63866892,0.138768){\color[rgb]{0,0,0}\makebox(0,0)[lb]{\smash{$n$}}}%
    \put(0,0){\includegraphics[width=\unitlength,page=3]{markings.pdf}}%
    \put(-0.00259329,0.0083595){\color[rgb]{0,0,0}\makebox(0,0)[lb]{\smash{Left}}}%
    \put(0.93102618,0.0083595){\color[rgb]{0,0,0}\makebox(0,0)[lb]{\smash{Right}}}%
  \end{picture}%
\endgroup%

%% file: skeinrepreplace.pdf_tex
%% Creator: Inkscape inkscape 0.48.4, www.inkscape.org
%% Accompanies image file 'skeinrepreplace.pdf' (pdf, eps, ps)
\begingroup%
  \makeatletter%
  \providecommand\color[2][]{%
    \errmessage{(Inkscape) Color is used for the text in Inkscape, but the package 'color.sty' is not loaded}%
    \renewcommand\color[2][]{}%
  }%
  \providecommand\transparent[1]{%
    \errmessage{(Inkscape) Transparency is used (non-zero) for the text in Inkscape, but the package 'transparent.sty' is not loaded}%
    \renewcommand\transparent[1]{}%
  }%
  \providecommand\rotatebox[2]{#2}%
  \ifx\svgwidth\undefined%
    \setlength{\unitlength}{600bp}%
    \ifx\svgscale\undefined%
      \relax%
    \else%
      \setlength{\unitlength}{\unitlength * \real{\svgscale}}%
    \fi%
  \else%
    \setlength{\unitlength}{\svgwidth}%
  \fi%
  \global\let\svgwidth\undefined%
  \global\let\svgscale\undefined%
  \makeatother%
  \begin{picture}(1,0.25333333)%
    \put(0,0){\includegraphics[width=\unitlength]{skeinrepreplace.pdf}}%
    \put(0.46057141,0.10857143){\color[rgb]{0,0,0}\makebox(0,0)[lb]{\smash{$\leftrightarrow$}}}%
  \end{picture}%
\endgroup%

%% file: repsigma.pdf_tex
%% Creator: Inkscape inkscape 0.91, www.inkscape.org
%% Accompanies image file 'repsigma.pdf' (pdf, eps, ps)
\begingroup%
  \makeatletter%
  \providecommand\color[2][]{%
    \errmessage{(Inkscape) Color is used for the text in Inkscape, but the package 'color.sty' is not loaded}%
    \renewcommand\color[2][]{}%
  }%
  \providecommand\transparent[1]{%
    \errmessage{(Inkscape) Transparency is used (non-zero) for the text in Inkscape, but the package 'transparent.sty' is not loaded}%
    \renewcommand\transparent[1]{}%
  }%
  \providecommand\rotatebox[2]{#2}%
  \ifx\svgwidth\undefined%
    \setlength{\unitlength}{1543.40615234bp}%
    \ifx\svgscale\undefined%
      \relax%
    \else%
      \setlength{\unitlength}{\unitlength * \real{\svgscale}}%
    \fi%
  \else%
    \setlength{\unitlength}{\svgwidth}%
  \fi%
  \global\let\svgwidth\undefined%
  \global\let\svgscale\undefined%
  \makeatother%
  \begin{picture}(1,0.37388018)%
    \put(0,0){\includegraphics[width=\unitlength,page=1]{repsigma.pdf}}%
    \put(0.18584955,0.00697645){\color[rgb]{0,0,0}\makebox(0,0)[lb]{\smash{}}}%
    \put(0.00004859,0.00458997){\color[rgb]{0,0,0}\makebox(0,0)[lb]{\smash{The graphical representation of $\sigma$}}}%
    \put(0.66742695,0.00391181){\color[rgb]{0,0,0}\makebox(0,0)[lb]{\smash{The actual skein $x^n_{\sigma}$}}}%
  \end{picture}%
\endgroup%

%% file: flowconfig.pdf_tex
%% Creator: Inkscape inkscape 0.91, www.inkscape.org
%% Accompanies image file 'flowconfig.pdf' (pdf, eps, ps)
\begingroup%
  \makeatletter%
  \providecommand\color[2][]{%
    \errmessage{(Inkscape) Color is used for the text in Inkscape, but the package 'color.sty' is not loaded}%
    \renewcommand\color[2][]{}%
  }%
  \providecommand\transparent[1]{%
    \errmessage{(Inkscape) Transparency is used (non-zero) for the text in Inkscape, but the package 'transparent.sty' is not loaded}%
    \renewcommand\transparent[1]{}%
  }%
  \providecommand\rotatebox[2]{#2}%
  \ifx\svgwidth\undefined%
    \setlength{\unitlength}{1436.40205078bp}%
    \ifx\svgscale\undefined%
      \relax%
    \else%
      \setlength{\unitlength}{\unitlength * \real{\svgscale}}%
    \fi%
  \else%
    \setlength{\unitlength}{\svgwidth}%
  \fi%
  \global\let\svgwidth\undefined%
  \global\let\svgscale\undefined%
  \makeatother%
  \begin{picture}(1,0.3060209)%
    \put(0,0){\includegraphics[width=\unitlength,page=1]{flowconfig.pdf}}%
    \put(0.02507905,0.29248012){\color[rgb]{0,0,0}\makebox(0,0)[lb]{\smash{$n=3$}}}%
    \put(0.09581203,0.00261937){\color[rgb]{0,0,0}\makebox(0,0)[lb]{\smash{$x^n$}}}%
    \put(0.20050096,0.28961251){\color[rgb]{0,0,0}\makebox(0,0)[lb]{\smash{}}}%
    \put(0.17265361,0.28961251){\color[rgb]{0,0,0}\makebox(0,0)[lb]{\smash{}}}%
    \put(0,0){\includegraphics[width=\unitlength,page=2]{flowconfig.pdf}}%
    \put(0.39934753,0.29136622){\color[rgb]{0,0,0}\makebox(0,0)[lb]{\smash{$n=3$}}}%
    \put(0.53524262,0.29136622){\color[rgb]{0,0,0}\makebox(0,0)[lb]{\smash{$n=3$}}}%
    \put(0.45448599,0.00150549){\color[rgb]{0,0,0}\makebox(0,0)[lb]{\smash{$x^n$}}}%
    \put(0,0){\includegraphics[width=\unitlength,page=3]{flowconfig.pdf}}%
    \put(0.75168878,0.29102696){\color[rgb]{0,0,0}\makebox(0,0)[lb]{\smash{}}}%
    \put(0.88758388,0.29102696){\color[rgb]{0,0,0}\makebox(0,0)[lb]{\smash{$n=3$}}}%
    \put(0.80682724,0.00116622){\color[rgb]{0,0,0}\makebox(0,0)[lb]{\smash{}}}%
    \put(0,0){\includegraphics[width=\unitlength,page=4]{flowconfig.pdf}}%
    \put(0.74630913,0.28961251){\color[rgb]{0,0,0}\makebox(0,0)[lb]{\smash{$n=3$}}}%
    \put(0.81314279,0.01113895){\color[rgb]{0,0,0}\makebox(0,0)[lb]{\smash{$x^n_{\sigma}$}}}%
  \end{picture}%
\endgroup%

%% file: line.pdf_tex
%% Creator: Inkscape inkscape 0.48.4, www.inkscape.org
%% Accompanies image file 'line.pdf' (pdf, eps, ps)
\begingroup%
  \makeatletter%
  \providecommand\color[2][]{%
    \errmessage{(Inkscape) Color is used for the text in Inkscape, but the package 'color.sty' is not loaded}%
    \renewcommand\color[2][]{}%
  }%
  \providecommand\transparent[1]{%
    \errmessage{(Inkscape) Transparency is used (non-zero) for the text in Inkscape, but the package 'transparent.sty' is not loaded}%
    \renewcommand\transparent[1]{}%
  }%
  \providecommand\rotatebox[2]{#2}%
  \ifx\svgwidth\undefined%
    \setlength{\unitlength}{1382.31279297bp}%
    \ifx\svgscale\undefined%
      \relax%
    \else%
      \setlength{\unitlength}{\unitlength * \real{\svgscale}}%
    \fi%
  \else%
    \setlength{\unitlength}{\svgwidth}%
  \fi%
  \global\let\svgwidth\undefined%
  \global\let\svgscale\undefined%
  \makeatother%
  \begin{picture}(1,0.18911629)%
    \put(0,0){\includegraphics[width=\unitlength]{line.pdf}}%
    \put(0.38462841,0.06044276){\color[rgb]{0,0,0}\makebox(0,0)[lb]{\smash{$h=\frac{1}{n}$}}}%
    \put(0.38578589,0.01761599){\color[rgb]{0,0,0}\makebox(0,0)[lb]{\smash{$h=-\frac{1}{n}$}}}%
    \put(0.38462841,0.0997971){\color[rgb]{0,0,0}\makebox(0,0)[lb]{\smash{$h=\frac{2}{n}$}}}%
    \put(0.340532,0.00102071){\color[rgb]{0,0,0}\makebox(0,0)[lb]{\smash{$H$}}}%
    \put(0.93558913,0.06044276){\color[rgb]{0,0,0}\makebox(0,0)[lb]{\smash{$h=\frac{1}{n}$}}}%
    \put(0.93674661,0.01761599){\color[rgb]{0,0,0}\makebox(0,0)[lb]{\smash{$h=-\frac{1}{n}$}}}%
    \put(0.93558913,0.0997971){\color[rgb]{0,0,0}\makebox(0,0)[lb]{\smash{$h=\frac{2}{n}$}}}%
    \put(0.88802025,0.00102071){\color[rgb]{0,0,0}\makebox(0,0)[lb]{\smash{$H$}}}%
  \end{picture}%
\endgroup%

%% file: firstset.pdf_tex
%% Creator: Inkscape inkscape 0.48.4, www.inkscape.org
%% Accompanies image file 'firstset.pdf' (pdf, eps, ps)
\begingroup%
  \makeatletter%
  \providecommand\color[2][]{%
    \errmessage{(Inkscape) Color is used for the text in Inkscape, but the package 'color.sty' is not loaded}%
    \renewcommand\color[2][]{}%
  }%
  \providecommand\transparent[1]{%
    \errmessage{(Inkscape) Transparency is used (non-zero) for the text in Inkscape, but the package 'transparent.sty' is not loaded}%
    \renewcommand\transparent[1]{}%
  }%
  \providecommand\rotatebox[2]{#2}%
  \ifx\svgwidth\undefined%
    \setlength{\unitlength}{538.45bp}%
    \ifx\svgscale\undefined%
      \relax%
    \else%
      \setlength{\unitlength}{\unitlength * \real{\svgscale}}%
    \fi%
  \else%
    \setlength{\unitlength}{\svgwidth}%
  \fi%
  \global\let\svgwidth\undefined%
  \global\let\svgscale\undefined%
  \makeatother%
  \begin{picture}(1,0.9582065)%
    \put(0,0){\includegraphics[width=\unitlength]{firstset.pdf}}%
  \end{picture}%
\endgroup%

%% file: firstsetH.pdf_tex
%% Creator: Inkscape inkscape 0.48.4, www.inkscape.org
%% Accompanies image file 'firstsetH.pdf' (pdf, eps, ps)
\begingroup%
  \makeatletter%
  \providecommand\color[2][]{%
    \errmessage{(Inkscape) Color is used for the text in Inkscape, but the package 'color.sty' is not loaded}%
    \renewcommand\color[2][]{}%
  }%
  \providecommand\transparent[1]{%
    \errmessage{(Inkscape) Transparency is used (non-zero) for the text in Inkscape, but the package 'transparent.sty' is not loaded}%
    \renewcommand\transparent[1]{}%
  }%
  \providecommand\rotatebox[2]{#2}%
  \ifx\svgwidth\undefined%
    \setlength{\unitlength}{547.59833984bp}%
    \ifx\svgscale\undefined%
      \relax%
    \else%
      \setlength{\unitlength}{\unitlength * \real{\svgscale}}%
    \fi%
  \else%
    \setlength{\unitlength}{\svgwidth}%
  \fi%
  \global\let\svgwidth\undefined%
  \global\let\svgscale\undefined%
  \makeatother%
  \begin{picture}(1,0.46413231)%
    \put(0,0){\includegraphics[width=\unitlength]{firstsetH.pdf}}%
    \put(0.15021984,0.21358799){\color[rgb]{0,0,0}\rotatebox{-38.93822967}{\makebox(0,0)[lb]{\smash{$C^u_{i+1}$}}}}%
    \put(0.17199958,0.3051824){\color[rgb]{0,0,0}\rotatebox{-40.34715999}{\makebox(0,0)[lb]{\smash{$C^u_{i}$}}}}%
    \put(-0.00145522,0.09092933){\color[rgb]{0,0,0}\makebox(0,0)[lb]{\smash{$H$}}}%
  \end{picture}%
\endgroup%

%% file: skeinHcomb1.pdf_tex
%% Creator: Inkscape inkscape 0.48.4, www.inkscape.org
%% Accompanies image file 'skeinHcomb1.pdf' (pdf, eps, ps)
\begingroup%
  \makeatletter%
  \providecommand\color[2][]{%
    \errmessage{(Inkscape) Color is used for the text in Inkscape, but the package 'color.sty' is not loaded}%
    \renewcommand\color[2][]{}%
  }%
  \providecommand\transparent[1]{%
    \errmessage{(Inkscape) Transparency is used (non-zero) for the text in Inkscape, but the package 'transparent.sty' is not loaded}%
    \renewcommand\transparent[1]{}%
  }%
  \providecommand\rotatebox[2]{#2}%
  \ifx\svgwidth\undefined%
    \setlength{\unitlength}{745.79262695bp}%
    \ifx\svgscale\undefined%
      \relax%
    \else%
      \setlength{\unitlength}{\unitlength * \real{\svgscale}}%
    \fi%
  \else%
    \setlength{\unitlength}{\svgwidth}%
  \fi%
  \global\let\svgwidth\undefined%
  \global\let\svgscale\undefined%
  \makeatother%
  \begin{picture}(1,0.16197532)%
    \put(0,0){\includegraphics[width=\unitlength]{skeinHcomb1.pdf}}%
    \put(0.48332027,0.07724701){\color[rgb]{0,0,0}\makebox(0,0)[lb]{\smash{or}}}%
    \put(0.06957542,0.07835383){\color[rgb]{0,0,0}\makebox(0,0)[lb]{\smash{$-x$}}}%
    \put(0.14221146,0.12994803){\color[rgb]{0,0,0}\makebox(0,0)[lb]{\smash{$x'$}}}%
    \put(0.16167303,0.020381){\color[rgb]{0,0,0}\makebox(0,0)[lb]{\smash{$x''$}}}%
    \put(0.56301016,0.07835383){\color[rgb]{0,0,0}\makebox(0,0)[lb]{\smash{$-x$}}}%
    \put(0.63779158,0.12994803){\color[rgb]{0,0,0}\makebox(0,0)[lb]{\smash{$x'$}}}%
    \put(0.65939854,0.020381){\color[rgb]{0,0,0}\makebox(0,0)[lb]{\smash{$x''$}}}%
    \put(-0.00106849,0.07892578){\color[rgb]{0,0,0}\makebox(0,0)[lb]{\smash{$C^u_i$}}}%
    \put(0.94671987,0.07892578){\color[rgb]{0,0,0}\makebox(0,0)[lb]{\smash{$C^u_i$}}}%
  \end{picture}%
\endgroup%

%% file: skeinHcomb2.pdf_tex
%% Creator: Inkscape inkscape 0.91, www.inkscape.org
%% Accompanies image file 'skeinHcomb2.pdf' (pdf, eps, ps)
\begingroup%
  \makeatletter%
  \providecommand\color[2][]{%
    \errmessage{(Inkscape) Color is used for the text in Inkscape, but the package 'color.sty' is not loaded}%
    \renewcommand\color[2][]{}%
  }%
  \providecommand\transparent[1]{%
    \errmessage{(Inkscape) Transparency is used (non-zero) for the text in Inkscape, but the package 'transparent.sty' is not loaded}%
    \renewcommand\transparent[1]{}%
  }%
  \providecommand\rotatebox[2]{#2}%
  \ifx\svgwidth\undefined%
    \setlength{\unitlength}{835.39257812bp}%
    \ifx\svgscale\undefined%
      \relax%
    \else%
      \setlength{\unitlength}{\unitlength * \real{\svgscale}}%
    \fi%
  \else%
    \setlength{\unitlength}{\svgwidth}%
  \fi%
  \global\let\svgwidth\undefined%
  \global\let\svgscale\undefined%
  \makeatother%
  \begin{picture}(1,0.14460267)%
    \put(0,0){\includegraphics[width=\unitlength,page=1]{skeinHcomb2.pdf}}%
    \put(0.49593365,0.06894962){\color[rgb]{0,0,0}\makebox(0,0)[lb]{\smash{or}}}%
    \put(0,0){\includegraphics[width=\unitlength,page=2]{skeinHcomb2.pdf}}%
    \put(0.11714094,0.05804137){\color[rgb]{0,0,0}\makebox(0,0)[lb]{\smash{$-p$}}}%
    \put(0.29985456,0.07835584){\color[rgb]{0,0,0}\makebox(0,0)[lb]{\smash{$-p'$}}}%
    \put(0.7390342,0.07646001){\color[rgb]{0,0,0}\makebox(0,0)[lb]{\smash{$-p'$}}}%
    \put(0.55339906,0.05421083){\color[rgb]{0,0,0}\makebox(0,0)[lb]{\smash{$-p$}}}%
    \put(0.18003854,0.11805632){\color[rgb]{0,0,0}\makebox(0,0)[lb]{\smash{$x'$}}}%
    \put(0.1935815,0.0165034){\color[rgb]{0,0,0}\makebox(0,0)[lb]{\smash{$x''$}}}%
    \put(0.62821106,0.11805632){\color[rgb]{0,0,0}\makebox(0,0)[lb]{\smash{$x'$}}}%
    \put(0.64366929,0.00884233){\color[rgb]{0,0,0}\makebox(0,0)[lb]{\smash{$x''$}}}%
    \put(0,0){\includegraphics[width=\unitlength,page=3]{skeinHcomb2.pdf}}%
    \put(-0.0009539,0.0665323){\color[rgb]{0,0,0}\makebox(0,0)[lb]{\smash{$C^u_i$}}}%
    \put(0.95243447,0.0665323){\color[rgb]{0,0,0}\makebox(0,0)[lb]{\smash{$C^u_i$}}}%
  \end{picture}%
\endgroup%

%% file: skeinHcomb3n.pdf_tex
%% Creator: Inkscape inkscape 0.91, www.inkscape.org
%% Accompanies image file 'skeinHcomb3n.pdf' (pdf, eps, ps)
\begingroup%
  \makeatletter%
  \providecommand\color[2][]{%
    \errmessage{(Inkscape) Color is used for the text in Inkscape, but the package 'color.sty' is not loaded}%
    \renewcommand\color[2][]{}%
  }%
  \providecommand\transparent[1]{%
    \errmessage{(Inkscape) Transparency is used (non-zero) for the text in Inkscape, but the package 'transparent.sty' is not loaded}%
    \renewcommand\transparent[1]{}%
  }%
  \providecommand\rotatebox[2]{#2}%
  \ifx\svgwidth\undefined%
    \setlength{\unitlength}{759.72206416bp}%
    \ifx\svgscale\undefined%
      \relax%
    \else%
      \setlength{\unitlength}{\unitlength * \real{\svgscale}}%
    \fi%
  \else%
    \setlength{\unitlength}{\svgwidth}%
  \fi%
  \global\let\svgwidth\undefined%
  \global\let\svgscale\undefined%
  \makeatother%
  \begin{picture}(1,0.21586842)%
    \put(0,0){\includegraphics[width=\unitlength,page=1]{skeinHcomb3n.pdf}}%
    \put(0.11065226,0.07822407){\color[rgb]{0,0,0}\makebox(0,0)[lb]{\smash{$-p$}}}%
    \put(0.36909263,0.07671977){\color[rgb]{0,0,0}\makebox(0,0)[lb]{\smash{$-p'$}}}%
    \put(0.05574494,0.127415){\color[rgb]{0,0,0}\makebox(0,0)[lb]{\smash{$-p_1$}}}%
    \put(0.21610434,0.12891933){\color[rgb]{0,0,0}\makebox(0,0)[lb]{\smash{$-p_1'$}}}%
    \put(0.17093229,0.12636198){\color[rgb]{0,0,0}\makebox(0,0)[lb]{\smash{$y$}}}%
    \put(0.66585016,0.1267381){\color[rgb]{0,0,0}\makebox(0,0)[lb]{\smash{$-p$}}}%
    \put(0.88322284,0.1267381){\color[rgb]{0,0,0}\makebox(0,0)[lb]{\smash{$-p'$}}}%
    \put(0.61049156,0.07521548){\color[rgb]{0,0,0}\makebox(0,0)[lb]{\smash{$-p_1$}}}%
    \put(0.71654536,0.06288013){\color[rgb]{0,0,0}\makebox(0,0)[lb]{\smash{$-p_1'$}}}%
    \put(0.69367984,0.08348919){\color[rgb]{0,0,0}\makebox(0,0)[lb]{\smash{$y$}}}%
    \put(0.50233169,0.1015409){\color[rgb]{0,0,0}\makebox(0,0)[lb]{\smash{or}}}%
    \put(-0.00131113,0.07491463){\color[rgb]{0,0,0}\makebox(0,0)[lb]{\smash{$C^u_i$}}}%
    \put(-0.00055897,0.12470726){\color[rgb]{0,0,0}\makebox(0,0)[lb]{\smash{$C^u_{i+1}$}}}%
    \put(0.31564687,-0.31665717){\color[rgb]{0,0,0}\makebox(0,0)[lt]{\begin{minipage}{0.20684256\unitlength}\raggedright \end{minipage}}}%
    \put(0.54716011,0.13207838){\color[rgb]{0,0,0}\makebox(0,0)[lb]{\smash{$C^u_i$}}}%
    \put(0.54716011,0.07310945){\color[rgb]{0,0,0}\makebox(0,0)[lb]{\smash{$C^u_{i-1}$}}}%
    \put(-0.06343913,0.21210764){\color[rgb]{0,0,0}\makebox(0,0)[lt]{\begin{minipage}{0.18803869\unitlength}\raggedright \end{minipage}}}%
  \end{picture}%
\endgroup%

%% file: skeinHcomb4.pdf_tex
%% Creator: Inkscape inkscape 0.91, www.inkscape.org
%% Accompanies image file 'skeinHcomb4.pdf' (pdf, eps, ps)
\begingroup%
  \makeatletter%
  \providecommand\color[2][]{%
    \errmessage{(Inkscape) Color is used for the text in Inkscape, but the package 'color.sty' is not loaded}%
    \renewcommand\color[2][]{}%
  }%
  \providecommand\transparent[1]{%
    \errmessage{(Inkscape) Transparency is used (non-zero) for the text in Inkscape, but the package 'transparent.sty' is not loaded}%
    \renewcommand\transparent[1]{}%
  }%
  \providecommand\rotatebox[2]{#2}%
  \ifx\svgwidth\undefined%
    \setlength{\unitlength}{685.28467374bp}%
    \ifx\svgscale\undefined%
      \relax%
    \else%
      \setlength{\unitlength}{\unitlength * \real{\svgscale}}%
    \fi%
  \else%
    \setlength{\unitlength}{\svgwidth}%
  \fi%
  \global\let\svgwidth\undefined%
  \global\let\svgscale\undefined%
  \makeatother%
  \begin{picture}(1,0.18298872)%
    \put(0,0){\includegraphics[width=\unitlength,page=1]{skeinHcomb4.pdf}}%
    \put(0.48900876,0.09091135){\color[rgb]{0,0,0}\makebox(0,0)[lb]{\smash{or}}}%
    \put(0.02422921,0.14694576){\color[rgb]{0,0,0}\makebox(0,0)[lb]{\smash{}}}%
    \put(0,0){\includegraphics[width=\unitlength,page=2]{skeinHcomb4.pdf}}%
    \put(0.00569453,0.02292058){\color[rgb]{0,0,0}\makebox(0,0)[lb]{\smash{$-p_{m-1}$}}}%
    \put(0.3398337,0.01383028){\color[rgb]{0,0,0}\makebox(0,0)[lb]{\smash{$-p'_{m-1}$}}}%
    \put(0.16494851,0.09475425){\color[rgb]{0,0,0}\makebox(0,0)[lb]{\smash{$y_m$}}}%
    \put(0.06505259,0.08558183){\color[rgb]{0,0,0}\makebox(0,0)[lb]{\smash{$-x_1$}}}%
    \put(-0.00499129,0.08391412){\color[rgb]{0,0,0}\makebox(0,0)[lb]{\smash{$-p_m$}}}%
    \put(0.227988,0.09892353){\color[rgb]{0,0,0}\makebox(0,0)[lb]{\smash{$-p'_m$}}}%
    \put(0.1580167,0.0111676){\color[rgb]{0,0,0}\makebox(0,0)[lb]{\smash{$y_{m-1}$}}}%
    \put(0.86974541,0.1409054){\color[rgb]{0,0,0}\makebox(0,0)[lb]{\smash{$-p'_{m-1}$}}}%
    \put(0.61126107,0.14046604){\color[rgb]{0,0,0}\makebox(0,0)[lb]{\smash{$-p_{m-1}$}}}%
    \put(0.54388815,0.08972286){\color[rgb]{0,0,0}\makebox(0,0)[lb]{\smash{$-p_m$}}}%
    \put(0.68764488,0.08471973){\color[rgb]{0,0,0}\makebox(0,0)[lb]{\smash{$y_m$}}}%
    \put(0.61426557,0.0813843){\color[rgb]{0,0,0}\makebox(0,0)[lb]{\smash{$-x_m$}}}%
    \put(0.76169129,0.07204512){\color[rgb]{0,0,0}\makebox(0,0)[lb]{\smash{$-p_m$}}}%
    \put(0.74468062,0.15893289){\color[rgb]{0,0,0}\makebox(0,0)[lb]{\smash{$y_{m-1}$}}}%
    \put(0.30896354,0.08844377){\color[rgb]{0,0,0}\makebox(0,0)[lb]{\smash{$-z_m$}}}%
    \put(0,0){\includegraphics[width=\unitlength,page=3]{skeinHcomb4.pdf}}%
    \put(0.94559243,0.16010031){\color[rgb]{0,0,0}\makebox(0,0)[lb]{\smash{$-z_m-1$}}}%
    \put(0.83118742,0.08538683){\color[rgb]{0,0,0}\makebox(0,0)[lb]{\smash{$-z_m$}}}%
  \end{picture}%
\endgroup%

%% file: basearc.pdf_tex
%% Creator: Inkscape inkscape 0.91, www.inkscape.org
%% Accompanies image file 'basearc.pdf' (pdf, eps, ps)
\begingroup%
  \makeatletter%
  \providecommand\color[2][]{%
    \errmessage{(Inkscape) Color is used for the text in Inkscape, but the package 'color.sty' is not loaded}%
    \renewcommand\color[2][]{}%
  }%
  \providecommand\transparent[1]{%
    \errmessage{(Inkscape) Transparency is used (non-zero) for the text in Inkscape, but the package 'transparent.sty' is not loaded}%
    \renewcommand\transparent[1]{}%
  }%
  \providecommand\rotatebox[2]{#2}%
  \ifx\svgwidth\undefined%
    \setlength{\unitlength}{411.02037805bp}%
    \ifx\svgscale\undefined%
      \relax%
    \else%
      \setlength{\unitlength}{\unitlength * \real{\svgscale}}%
    \fi%
  \else%
    \setlength{\unitlength}{\svgwidth}%
  \fi%
  \global\let\svgwidth\undefined%
  \global\let\svgscale\undefined%
  \makeatother%
  \begin{picture}(1,0.38802386)%
    \put(0,0){\includegraphics[width=\unitlength,page=1]{basearc.pdf}}%
    \put(0.90367435,0.04268124){\color[rgb]{0,0,0}\makebox(0,0)[lb]{\smash{$h=0$}}}%
    \put(0,0){\includegraphics[width=\unitlength,page=2]{basearc.pdf}}%
    \put(0.62840126,0.32629595){\color[rgb]{0,0,0}\makebox(0,0)[lb]{\smash{$p$}}}%
  \end{picture}%
\endgroup%

%% file: circlmerge.pdf_tex
%% Creator: Inkscape inkscape 0.91, www.inkscape.org
%% Accompanies image file 'circlmerge.pdf' (pdf, eps, ps)
\begingroup%
  \makeatletter%
  \providecommand\color[2][]{%
    \errmessage{(Inkscape) Color is used for the text in Inkscape, but the package 'color.sty' is not loaded}%
    \renewcommand\color[2][]{}%
  }%
  \providecommand\transparent[1]{%
    \errmessage{(Inkscape) Transparency is used (non-zero) for the text in Inkscape, but the package 'transparent.sty' is not loaded}%
    \renewcommand\transparent[1]{}%
  }%
  \providecommand\rotatebox[2]{#2}%
  \ifx\svgwidth\undefined%
    \setlength{\unitlength}{1019.82398807bp}%
    \ifx\svgscale\undefined%
      \relax%
    \else%
      \setlength{\unitlength}{\unitlength * \real{\svgscale}}%
    \fi%
  \else%
    \setlength{\unitlength}{\svgwidth}%
  \fi%
  \global\let\svgwidth\undefined%
  \global\let\svgscale\undefined%
  \makeatother%
  \begin{picture}(1,0.32120387)%
    \put(0,0){\includegraphics[width=\unitlength,page=1]{circlmerge.pdf}}%
    \put(0.03941559,0.00963803){\color[rgb]{0,0,0}\makebox(0,0)[lb]{\smash{$\sigma_f'$}}}%
    \put(0.21943675,0.00277508){\color[rgb]{0,0,0}\makebox(0,0)[lb]{\smash{$\Sk_{\sigma_1=\sigma_+}$}}}%
    \put(0,0){\includegraphics[width=\unitlength,page=2]{circlmerge.pdf}}%
    \put(0.4359447,0.00277508){\color[rgb]{0,0,0}\makebox(0,0)[lb]{\smash{$\Sk_{\sigma_2}$}}}%
    \put(0.64617708,0.00277508){\color[rgb]{0,0,0}\makebox(0,0)[lb]{\smash{$\Sk_{\sigma_3}$}}}%
    \put(0.83130706,0.00277508){\color[rgb]{0,0,0}\makebox(0,0)[lb]{\smash{$\Sk_{\sigma_4=\sigma_f'}$}}}%
  \end{picture}%
\endgroup%

%% file: squaredecomp.pdf_tex
%% Creator: Inkscape inkscape 0.91, www.inkscape.org
%% Accompanies image file 'squaredecomp.pdf' (pdf, eps, ps)
\begingroup%
  \makeatletter%
  \providecommand\color[2][]{%
    \errmessage{(Inkscape) Color is used for the text in Inkscape, but the package 'color.sty' is not loaded}%
    \renewcommand\color[2][]{}%
  }%
  \providecommand\transparent[1]{%
    \errmessage{(Inkscape) Transparency is used (non-zero) for the text in Inkscape, but the package 'transparent.sty' is not loaded}%
    \renewcommand\transparent[1]{}%
  }%
  \providecommand\rotatebox[2]{#2}%
  \ifx\svgwidth\undefined%
    \setlength{\unitlength}{1084.21201172bp}%
    \ifx\svgscale\undefined%
      \relax%
    \else%
      \setlength{\unitlength}{\unitlength * \real{\svgscale}}%
    \fi%
  \else%
    \setlength{\unitlength}{\svgwidth}%
  \fi%
  \global\let\svgwidth\undefined%
  \global\let\svgscale\undefined%
  \makeatother%
  \begin{picture}(1,1.03627364)%
    \put(0.26148261,0.92341556){\color[rgb]{0,0,0}\makebox(0,0)[lb]{\smash{$\Sk_2$}}}%
    \put(0.04328596,0.57240356){\color[rgb]{0,0,0}\makebox(0,0)[lb]{\smash{$\Sk_1$}}}%
    \put(0,0){\includegraphics[width=\unitlength,page=1]{squaredecomp.pdf}}%
    \put(0.24583356,0.8453264){\color[rgb]{0,0,0}\makebox(0,0)[lb]{\smash{$n$}}}%
    \put(0,0){\includegraphics[width=\unitlength,page=2]{squaredecomp.pdf}}%
  \end{picture}%
\endgroup%

%% file: flowpath.pdf_tex
%% Creator: Inkscape inkscape 0.91, www.inkscape.org
%% Accompanies image file 'flowpath.pdf' (pdf, eps, ps)
\begingroup%
  \makeatletter%
  \providecommand\color[2][]{%
    \errmessage{(Inkscape) Color is used for the text in Inkscape, but the package 'color.sty' is not loaded}%
    \renewcommand\color[2][]{}%
  }%
  \providecommand\transparent[1]{%
    \errmessage{(Inkscape) Transparency is used (non-zero) for the text in Inkscape, but the package 'transparent.sty' is not loaded}%
    \renewcommand\transparent[1]{}%
  }%
  \providecommand\rotatebox[2]{#2}%
  \ifx\svgwidth\undefined%
    \setlength{\unitlength}{1096.09990234bp}%
    \ifx\svgscale\undefined%
      \relax%
    \else%
      \setlength{\unitlength}{\unitlength * \real{\svgscale}}%
    \fi%
  \else%
    \setlength{\unitlength}{\svgwidth}%
  \fi%
  \global\let\svgwidth\undefined%
  \global\let\svgscale\undefined%
  \makeatother%
  \begin{picture}(1,1.10807405)%
    \put(0,0){\includegraphics[width=\unitlength,page=1]{flowpath.pdf}}%
    \put(-0.00090876,0.64640859){\color[rgb]{0,0,0}\makebox(0,0)[lb]{\smash{$n=3$}}}%
    \put(0.23035274,0.58280646){\color[rgb]{0,0,0}\makebox(0,0)[lb]{\smash{$3$}}}%
    \put(0,0){\includegraphics[width=\unitlength,page=2]{flowpath.pdf}}%
    \put(0.7418807,1.09975621){\color[rgb]{0,0,1}\makebox(0,0)[lb]{\smash{$G$}}}%
    \put(0.54064773,0.98089324){\color[rgb]{0,0,0.98823529}\makebox(0,0)[lb]{\smash{$v$}}}%
    \put(0.34879868,0.12904182){\color[rgb]{0,0,1}\makebox(0,0)[lb]{\smash{$v'$}}}%
    \put(0.33837211,0.28439777){\color[rgb]{0,0,1}\makebox(0,0)[lb]{\smash{$W_1$}}}%
    \put(0.50936802,0.2604167){\color[rgb]{0,0,1}\makebox(0,0)[lb]{\smash{$W_2$}}}%
    \put(0.63657227,0.75359382){\color[rgb]{0,0,1}\makebox(0,0)[lb]{\smash{$W_3$}}}%
    \put(0.94519892,0.43662584){\color[rgb]{0,0,0}\makebox(0,0)[lb]{\smash{}}}%
    \put(0.93164447,0.40326079){\color[rgb]{0,0,1}\makebox(0,0)[lb]{\smash{$W_4$}}}%
    \put(0,0){\includegraphics[width=\unitlength,page=3]{flowpath.pdf}}%
    \put(0.09738424,0.57346181){\color[rgb]{0,0,0}\makebox(0,0)[lb]{\smash{$a$}}}%
    \put(0.079242,0.76322553){\color[rgb]{0,0,1}\makebox(0,0)[lb]{\smash{$e$}}}%
  \end{picture}%
\endgroup%

%% file: dr.pdf_tex
%% Creator: Inkscape inkscape 0.91, www.inkscape.org
%% Accompanies image file 'dr.pdf' (pdf, eps, ps)
\begingroup%
  \makeatletter%
  \providecommand\color[2][]{%
    \errmessage{(Inkscape) Color is used for the text in Inkscape, but the package 'color.sty' is not loaded}%
    \renewcommand\color[2][]{}%
  }%
  \providecommand\transparent[1]{%
    \errmessage{(Inkscape) Transparency is used (non-zero) for the text in Inkscape, but the package 'transparent.sty' is not loaded}%
    \renewcommand\transparent[1]{}%
  }%
  \providecommand\rotatebox[2]{#2}%
  \ifx\svgwidth\undefined%
    \setlength{\unitlength}{2328.83847656bp}%
    \ifx\svgscale\undefined%
      \relax%
    \else%
      \setlength{\unitlength}{\unitlength * \real{\svgscale}}%
    \fi%
  \else%
    \setlength{\unitlength}{\svgwidth}%
  \fi%
  \global\let\svgwidth\undefined%
  \global\let\svgscale\undefined%
  \makeatother%
  \begin{picture}(1,0.52575511)%
    \put(0.12852908,-0.09271509){\color[rgb]{0,0,0}\makebox(0,0)[lb]{\smash{}}}%
    \put(0.82506518,0.60093109){\color[rgb]{0,0,0}\makebox(0,0)[lt]{\begin{minipage}{0.81199507\unitlength}\raggedright \end{minipage}}}%
    \put(0.79525837,0.00605239){\color[rgb]{0,0,0}\makebox(0,0)[lb]{\smash{$D_e$}}}%
    \put(0.17842261,0.00181757){\color[rgb]{0,0,0}\makebox(0,0)[lb]{\smash{$D$}}}%
    \put(0,0){\includegraphics[width=\unitlength,page=1]{dr.pdf}}%
  \end{picture}%
\endgroup%

%% file: de.pdf_tex
%% Creator: Inkscape inkscape 0.91, www.inkscape.org
%% Accompanies image file 'de.pdf' (pdf, eps, ps)
\begingroup%
  \makeatletter%
  \providecommand\color[2][]{%
    \errmessage{(Inkscape) Color is used for the text in Inkscape, but the package 'color.sty' is not loaded}%
    \renewcommand\color[2][]{}%
  }%
  \providecommand\transparent[1]{%
    \errmessage{(Inkscape) Transparency is used (non-zero) for the text in Inkscape, but the package 'transparent.sty' is not loaded}%
    \renewcommand\transparent[1]{}%
  }%
  \providecommand\rotatebox[2]{#2}%
  \ifx\svgwidth\undefined%
    \setlength{\unitlength}{2100.95273438bp}%
    \ifx\svgscale\undefined%
      \relax%
    \else%
      \setlength{\unitlength}{\unitlength * \real{\svgscale}}%
    \fi%
  \else%
    \setlength{\unitlength}{\svgwidth}%
  \fi%
  \global\let\svgwidth\undefined%
  \global\let\svgscale\undefined%
  \makeatother%
  \begin{picture}(1,0.57249088)%
    \put(0.18542163,0.00201473){\color[rgb]{0,0,0}\makebox(0,0)[lb]{\smash{$D$}}}%
    \put(0.739184,0.00262383){\color[rgb]{0,0,0}\makebox(0,0)[lb]{\smash{$D^e$}}}%
    \put(0,0){\includegraphics[width=\unitlength,page=1]{de.pdf}}%
  \end{picture}%
\endgroup%